\documentclass{amsart}

\usepackage{amssymb, amsmath}
\usepackage[fleqn,tbtags]{mathtools}
\mathtoolsset{showonlyrefs,showmanualtags}
\usepackage{amscd}
\usepackage[active]{srcltx}
\usepackage{verbatim}
\usepackage{epsfig} %vkladani eps obrazku (vektorovych)
\usepackage{epstopdf}
\usepackage{subfigure}
\usepackage{newsymbol}
\let\emptyset \undefined
\let\ge       \undefined
\let\le       \undefined
\newsymbol\le          1336  \let\leq\le
\newsymbol\ge          133E  \let\geq\ge
\newsymbol\emptyset    203F
\newsymbol\notle       230A
\newsymbol\notge       230B

\theoremstyle{plain}
\newtheorem{theorem}{Theorem}[section]
\theoremstyle{remark}
\newtheorem{remark}[theorem]{Remark}
\newtheorem{example}[theorem]{Example}

\theoremstyle{plain}

\newtheorem{lemma}[theorem]{Lemma}
\newtheorem{proposition}[theorem]{Proposition}

\numberwithin{equation}{section}

\begin{document}

\title[Parameter estimation for SPDEs of second order]
{Parameter estimation for stochastic partial differential equations of second order}

\author{Josef Jan\' ak}
\address{Department of Mathematics\\
University of Economics in Prague\\
Ekonomick\' a 957, 148 00 Prague 4\\
Czech Republic}
\email{janj04@vse.cz}



\date\today

\begin{abstract} 
Stochastic partial differential equations of second order with two unknown parameters are studied. Based on ergodicity, two suitable families of minimum constrast estimators are introduced. Strong consistency and asymptotic normality of estimators are proved. The results are applied to hyperbolic equations perturbed by Brownian noise.
\end{abstract}

\thanks{This paper has been produced with contribution of long term
institutional support of research activities by Faculty of Informatics
and Statistics, University of Economics, Prague.\\
This paper was supported by the GA\v CR Grant no. 15--08819S}

\maketitle

\section{Introduction}
Statistical inference for stochastic partial differential equations driven by standard Brownian motion has been recently extensively studied.
While many authors use maximum likelihood estimators (MLE) as the most frequent tool (for example \cite{tudor},
where the parameter is linearly built in the drift), we are interested in minimum contrast estimator (MCE),
which has been studied since 1980's (see \cite{koskiloges} and \cite{koskiloges 2}). In more recent works, the (MCE) has also been provided for
the SPDEs driven by fractional Brownian motion (for example \cite{maslowskitudor} or \cite{maslowskipospisil}).

In this work, we study parameter estimation for SPDEs of second order, in particular, for the following wave equation with strong damping
\begin{align}
\frac{\partial^2 u}{\partial t^2} (t, \xi) &= b\Delta u(t, \xi) - 2a \frac{\partial u}{\partial t}(t, \xi) + \eta (t, \xi), \quad (t, \xi) \in \mathbb R_+ \times D, \label{example 0} \\
u(0, \xi) &= u_1(\xi), \quad \xi \in D, \notag \\
\frac{\partial u}{\partial t} (0, \xi) &= u_2(\xi), \quad \xi \in D, \notag \\
u(t, \xi) &= 0, \quad (t, \xi) \in \mathbb R_+ \times \partial D, \notag
\end{align}
where $D \subset \mathbb R^d$ is a bounded domain with a smooth boundary and $\eta$ is a random noise.

The aim of the paper is to provide strongly consistent estimators of unknown parameters $a$ and $b$,
based on the observation of the trajectory of the process $(u(t, \xi), \, 0 \leq t \leq T, \, \xi \in D)$,
which is the solution to \eqref{example 0}, up to time $T$. In order to do so, we follow up the work \cite{maslowskipospisil},
where minimum contrast estimators based on ergodic theorems were derived for analogous parabolic problems.

The present paper analyzes the problem second order in time. Strongly continuous semigroup $(S(t), t \geq 0)$ generated by the operator in the drift part is computed. The form of covariance operator $Q_{\infty}^{(a,b)}$, the covariance operator of the invariant measure of system \eqref{example 0}, is found and a strongly consistent family of estimators is established, which corresponds to the "classical" approach (cf. \cite{maslowskipospisil}). Moreover, an alternative family of estimators is proposed, and comparison of some basic properties shows, that this new family of estimators is in some sense better then the "classical" one. (See Theorem \ref{is smaller} for more detail.)

%its rewriting to the form of linear stochastic partial differential equation, the computation of strongly continuous semigroup $(S(t), t \geq 0)$ (which is generated by the operator in the drift part), the computation of covariance operator $Q_{\infty}^{(a,b)}$ (which is the covariance ope\-ra\-tor of the invariant measure of system \eqref{example 0}), and establishing some strongly consistent family of estimators, which are derived by some sort of "classical" approach. Moreover, another family of estimators is proposed, and from comparison of some basic properties, it shows up that this new family of estimators is in some way better then the "classical" one. (See Theorem \ref{is smaller} for precise statement.)

Note that in \cite{maslowskipospisil} the driving noise is a fractional Brownian motion (fBm) while in the present paper only standard Wiener process is considered. The main difficulty consists in the fact that the dependence of $Q_{\infty}^{(a,b)}$ on parameters in the present case is complicated and not explicit. However, statistical inference for (fBm)--driven second order equation will be studied in a forthcoming paper.

The paper is organized as follows. The Section \ref{preliminaries} summarizes some basic facts on stochastic linear partial differential equations, which is mostly due to \cite{dapratozabczyk}. In Section \ref{main results}, we introduce the setup as well as some assumptions which are needed. Then we compute the form of semigroup $(S(t), t \geq 0)$ and the form of covariance operator $Q_{\infty}^{(a,b)}$ for three different cases. Although the forms of semigroup $(S(t), t \geq 0)$ are different, all three formulae for the covariance operator coincide. These results are summarized in Subsection \ref{subsection summary}.

In Section \ref{parameter estimation}, the family of strongly consistent estimators $(\hat{a}_T, \hat{b}_T)$ is derived, which specify the general result from \cite{maslowskipospisil} to the present (second order in time) case. Moreover, new family of strongly consistent estimators $(\tilde{a}_T, \tilde{b}_T)$ is proposed. The asymptotic normality of both $(\hat{a}_T, \hat{b}_T)$ and $(\tilde{a}_T, \tilde{b}_T)$ is shown in Section \ref{section: AN}. In the end of this section, we show the possible advantage of the "new" estimators and we give an example of the so--called diagonal case, where the formulae may be considerably simplified. In Section \ref{section: examples}, we consider two basic examples where our general results are applied: the wave equation (Example \ref{example 1}) and the plate equation (Example \ref{example 2}). The results are illustrated by some numerical simulations in Section \ref{section: implementation}.

If $U$ and $V$ are Hilbert spaces, then $\mathcal L(U, V)$, $\mathcal L_2 (U, V)$ and $\mathcal L_1 (U, V)$ denote the respective spaces of all linear bounded, Hilbert--Schmidt and trace class operators from $U$ to $V$. Also $\mathcal L(V)$ stands for $\mathcal L(V, V)$, etc.

\section{Preliminaries} \label{preliminaries}
Given separable Hilbert spaces $U$ and $V$, we consider the equation
\begin{align}
dX(t) &= \mathcal A X(t) \, dt + \Phi \, dB(t), \label{linearequation} \\
X(0) &= x_0, \notag
\end{align}
where $(B(t), t \geq 0)$ is a standard cylindrical Brownian motion on $U$, $\mathcal A: \text{Dom}(\mathcal A) \rightarrow V, \, \text{Dom}(\mathcal A) \subset V$, $\mathcal A$ is the infinitesimal generator of a strongly continuous semigroup $(S(t), t \geq 0)$ on $V$, $\Phi \in \mathcal L (U,V)$ and $x_0 \in V$ is a random variable. We assume that $\mathbb E \| x_0 \|_V^2 < \infty$ and that $x_0$ and $(B(t), t \geq 0)$ are stochastically independent.

We also consider the following two conditions:
\begin{itemize}
\item[(A1)] $\Phi \in \mathcal L_2 (U,V)$,
\item[(A2)] There exist constants $K > 0$ and $\rho > 0$ such that for all $t \geq 0$
$$
\| S(t) \|_{\mathcal L (V)} \leq Ke^{- \rho t}, \quad t \geq 0.
$$
\end{itemize}

The condition (A1) means that the perturbing noise is, in fact, a genuine $V$--valued Brownian motion and the condition (A2) is the exponential stability of the semigroup generated by $\mathcal A$.

\begin{proposition}
If (A1) is satisfied, then equation \eqref{linearequation} admits a mild solution
\begin{equation}
X^{x_0}(t) = S(t)x_0 + Z(t), \quad t \geq 0,
\end{equation}
where $(Z(t), t \geq 0)$ is the convolution integral
\begin{equation}
Z(t) = \int_0^t S(t-u) \Phi \, dB(u).
\end{equation}

The process $(Z(t), t \geq 0)$ is a $V$--continuous centered Gaussian process with covariance operator given by the formula
\begin{equation}
Q_t = \int_0^t S(u) \Phi \Phi^* S^* (u) \, du.
\end{equation}
\end{proposition}
\begin{proof}
See \cite{dapratozabczyk}.
\end{proof}

\begin{proposition} \label{existence of invariant measure}
If (A1), (A2) are satisfied, then there is a unique invariant measure $\mu_{\infty} = N \left( 0, Q_{\infty} \right)$ for the equation \eqref{linearequation} and
$$
w^* - \lim_{t \rightarrow \infty} \mu_t^{x_0} = \mu_{\infty}
$$
for each initial condition $x_0 \in V$, where $\mu_t^{x_0} = \mathrm{Law} \, (X^{x_0}(t))$ and $\mathrm{Law} \, (\cdot)$ denotes the probability distribution.

The covariance operator $Q_{\infty}$ takes the form
\begin{equation} \label{qinfty}
Q_{\infty} = \int_0^{\infty} S(t) \Phi \Phi^* S^* (t) \, dt.
\end{equation}
\end{proposition}
\begin{proof}
See \cite{dapratozabczyk}.
\end{proof}

\section{The semigroup and covariance operator} \label{main results}
To interpret stochastic wave equation \eqref{example 0} rigorously, we rewrite it as a first order system in a standard way. Assume that $\{e_n, n \in \mathbb N \}$ is an orthonormal basis in $L^2(D)$ and the operator $A: \text{Dom}(A) \subset L^2(D) \rightarrow L^2(D)$ is such that
\begin{itemize}
\item[(i)] $Ae_n = - \alpha_n e_n$,
\item[(ii)] $\forall n \in \mathbb N \quad \alpha_n > 0$,
\item[(iii)] $\alpha_n \rightarrow \infty$ for $n \rightarrow \infty$.
\end{itemize}

These assumptions cover the case when the set $D \subset \mathbb R^d$ is open, bounded and the boundary $\partial D$ is sufficiently smooth, the operator $A = \Delta|_{\mathrm{Dom}(A)}$ and $\mathrm{Dom}(A) = H^2(D) \cap H^1_0 (D)$.

Next let us assume that $\Phi_1$ is a Hilbert--Schmidt operator on $L^2(D)$ such that $Q = \Phi_1 \Phi_1^*$ is strictly positive. Since $Q$ is a symmetric nuclear operator on $L^2(D)$ then there exists an orthonormal basis $\{e^{\prime}_n, n \in \mathbb N \}$ of $L^2(D)$ consisting of eigenvectors of $Q$, that is
\begin{itemize}
\item[(iv)] $Qe^{\prime}_n = \lambda_n e^{\prime}_n$,
\item[(v)] $\forall n \in \mathbb N \quad \lambda_n > 0$,
\item[(vi)] $\sum_{n=1}^{\infty} \lambda_n < \infty$.
\end{itemize}

Consider the Hilbert space $V = \text{Dom}((-A)^{\frac{1}{2}}) \times L^2(D)$ endowed with the inner product
\begin{align}
\left\langle \left( \begin{array}{r}
x_1\\
x_2
\end{array} \right), \left( \begin{array}{r}
y_1\\
y_2
\end{array} \right) \right\rangle_V &= \left\langle x_1, y_1 \right\rangle_{\text{Dom}((-A)^{\frac{1}{2}})} + \left\langle x_2, y_2 \right\rangle_{L^2(D)} \notag \\
&= \left\langle (-A)^{\frac{1}{2}}x_1, (-A)^{\frac{1}{2}}y_1 \right\rangle_{L^2(D)} + \left\langle x_2, y_2 \right\rangle_{L^2(D)},
\end{align}
for $(x_1, x_2)^{\top}, (y_1, y_2)^{\top} \in V$.

Also, consider the linear equation
\begin{align}
dX(t) &= \mathcal A X(t) \, dt + \Phi \, dB(t), \label{linear equation with parameters} \\
X(0) &= x_0 = \left( \begin{array}{c}
u_1\\
u_2
\end{array} \right), \notag
\end{align}
where the linear operator $\mathcal A: \text{Dom}(\mathcal A) = \text{Dom}(A) \times \text{Dom}((-A)^{\frac{1}{2}}) \rightarrow V$ is defined by
$$
\mathcal Ax =  \mathcal A \left( \begin{array}{r}
x_1\\
x_2
\end{array} \right) = \left( \begin{array}{cc}
0&I\\
bA&-2aI
\end{array} \right) \left( \begin{array}{r}
x_1\\
x_2
\end{array} \right), \quad \forall x = \left( \begin{array}{r}
x_1\\
x_2
\end{array} \right) \in \text{Dom}(\mathcal A),
$$
$a > 0, \, b > 0$ are unknown parameters (which are to be estimated), $u_1 \in \text{Dom}((-A)^{\frac{1}{2}}),$ $u_2 \in L^2(D)$, $x_0 = (u_1, u_2)^{\top} \in V$ satisfies $\mathbb E \| x_0 \|_V^2 < \infty$, where $\| \cdot \|_V := \sqrt{\left\langle \cdot, \cdot \right\rangle_V}$, and the linear operator $\Phi : U = V \rightarrow V$ is defined by
$$
\Phi = \left( \begin{array}{cc}
0&0\\
0&\Phi_1
\end{array} \right).
$$

With no loss of generality, we assume that the driving process in \eqref{linear equation with parameters} takes the form $(0, B(t))^{\top}$, where $(B(t), t \geq 0)$ is a standard cylindrical Brownian motion on $L^2(D)$.

Note that since the operator $\Phi_1$ is Hilbert--Schmidt in $L^2(D)$, the operator $\Phi$ is Hilbert--Schmidt in $V$.

The form of the eigenvalues of the operator $\mathcal A$ depends on whether $a^2 - b \alpha_n$ is negative, positive, or equal to zero. So in order to compute the form of the semigroup $(S(t), t \geq 0)$, we have to consider these three different cases, compute appropriate semigroups $(S_1(t), t \geq 0), \, (S_2(t), t \geq 0)$ and $(S_3(t), t \geq 0)$ and then combine them together to obtain the resulting formula (see Theorem \ref{form of S} below).

First let us divide $\mathbb N$ into three (disjoint) sets in this way: $\mathbb N = N_1 \cup N_2 \cup N_3$, where
\begin{align}
N_1 &= \left\{ n \in \mathbb N, \alpha_n > \frac{a^2}{b} \right\}, \\
N_2 &= \left\{ n \in \mathbb N, \alpha_n < \frac{a^2}{b} \right\}, \\
N_3 &= \left\{ n \in \mathbb N, \alpha_n = \frac{a^2}{b} \right\}.
\end{align}

Since $\alpha_n \rightarrow \infty$, the sets $N_2$ and $N_3$ are finite (or even empty) sets, while the set $N_1$ is infinite. Let us also write the space $V$ as a direct sum of three closed linear subspaces
\begin{equation}
V = V_1 \oplus V_2 \oplus V_3,
\end{equation}
where
\begin{equation}
V_i = \text{span} \, \{f_n, n \in N_i \} \times \text{span} \, \{e_n, n \in N_i \}, \quad i = 1, 2, 3.
\end{equation}
Note that the orthonormal basis of the space $\text{Dom}((-A)^{\frac{1}{2}})$ is $\{ f_n, n \in \mathbb N \}$, where $f_n = \frac{1}{\sqrt{\alpha_n}} e_n$.

\subsection{Case $\alpha_n > \frac{a^2}{b}$}
In the case $\alpha_n > \frac{a^2}{b}$, the eigenvalues $\{ l_n^{1,2}, n \in \mathbb N \}$ of the operator $\mathcal A$ are
\begin{equation}
l_n^{1,2} = -a \pm i \sqrt{b \alpha_n - a^2}
\end{equation}
and the operator $\mathcal A$ generates a $C_0$-semigroup on $V$, which is also exponentially stable (the real parts of the eigenvalues $l_n^{1,2}$ are negative). The form of the semigroup $(S_1(t), t \geq 0)$ is given in Lemma \ref{form of S1} below. Define the operator
\begin{equation}
P_1 x = \sum_{n \in N_1} \left\langle x, e_n \right\rangle_{L^2(D)} e_n,
\end{equation}
which is the operator of projection on the $\text{span} \, \{ e_n, n \in N_1 \}$ (that is $P_1 : L^2(D) \rightarrow \text{span} \, \{ e_n, n \in N_1 \}$). Furthermore define the operator $\beta : L^2(D) \rightarrow L^2(D)$ by $\beta = \left( -bA - a^2 I \right)^{\frac{1}{2}} P_1$, that is
\begin{equation}
\beta x = \sum_{n \in N_1} \sqrt{b \alpha_n - a^2} \left\langle x, e_n \right\rangle_{L^2(D)} e_n, \quad x \in \text{Dom}(\beta),
\end{equation}
where $\text{Dom}(\beta) = \{ x \in L^2(D), \sum_{n \in N_1} \left( b \alpha_n - a^2 \right) \left\langle x, e_n \right\rangle_{L^2(D)}^2 < \infty \} = \text{Dom}((-A)^{\frac{1}{2}})$.

Similarly define
\begin{align}
\beta^{-1} x &= \sum_{n \in N_1} \frac{1}{\sqrt{b \alpha_n - a^2}} \left\langle x, e_n \right\rangle_{L^2(D)} e_n, \label{betanaminusprvni} \\
\sin (\beta t) x &= \sum_{n \in N_1} \sin \left(\sqrt{b \alpha_n - a^2} \, t \right) \left\langle x, e_n \right\rangle_{L^2(D)} e_n, \\
\beta^{-1} \sin (\beta t) x &= \sum_{n \in N_1} \frac{\sin \left(\sqrt{b \alpha_n - a^2} \, t \right)}{\sqrt{b \alpha_n - a^2}} \left\langle x, e_n \right\rangle_{L^2(D)} e_n, \\
\cos (\beta t) x &= \sum_{n \in N_1} \cos \left(\sqrt{b \alpha_n - a^2} \, t \right) \left\langle x, e_n \right\rangle_{L^2(D)} e_n, \label{cosbetat}
\end{align}
where $x \in L^2(D)$.

Note that $\beta^{-1} = \left( -bA - a^2 I \right)^{-\frac{1}{2}}P_1$, so $\beta^{-1} \beta x = P_1 x$ for any $x \in \text{Dom}(\beta)$ and $\beta^{-1} \beta x = Ix$ for any $x \in \text{Dom}(\beta) \cap \text{span} \, \{ e_n, n \in N_1 \}$. Also note that the operator $\cos (\beta t)$ evaluated at time $t=0$ is $\cos (\beta t)|_{t=0} \, x = P_1 x$ for any $x \in L^2(D)$.

The form of the semigroup $(S_1(t), t \geq 0)$, for the coordinates from the set $N_1$, is described by the following Lemma.

\begin{lemma}\label{form of S1}
For all $x = (x_1, x_2)^{\top} \in V_1$ we have
$$
S_1(t) \left( \begin{array}{r}
x_1\\
x_2
\end{array} \right) = \left( \begin{array}{cc}
s_{11}(t)&s_{12}(t)\\
s_{21}(t)&s_{22}(t)
\end{array} \right) \left( \begin{array}{r}
x_1\\
x_2
\end{array} \right), \quad \forall t \geq 0,
$$
where
\begin{align}
s_{11}(t) &= e^{-at} \left( \cos (\beta t) + a \beta^{-1} \sin (\beta t) \right), \notag \\
s_{12}(t) &= e^{-at} \beta^{-1} \sin (\beta t), \notag \\
s_{21}(t) &= e^{-at} \left( - \beta - a^2 \beta^{-1} \right) \sin (\beta t), \notag \\
s_{22}(t) &= e^{-at} \left( \cos (\beta t) - a \beta^{-1} \sin (\beta t) \right) \notag.
\end{align}
\end{lemma}
\begin{proof}
It is sufficient to show that
\begin{itemize}
\item[(i)] $$
S_1(0) = \left( \begin{array}{cc}
I&0\\
0&I
\end{array} \right),
$$
\item[(ii)] $$
\frac{d}{dt} S_1(t) x = \mathcal A S_1(t)x, \quad \forall x \in \text{Dom}(\mathcal A) \cap V_1, \quad \forall t \geq 0.
$$
\end{itemize}

As for (i), it is easy to see that
$$
S_1(0) \left( \begin{array}{r}
x_1\\
x_2
\end{array} \right) = \left( \begin{array}{cc}
P_1&0\\
0&P_1
\end{array} \right) \left( \begin{array}{r}
x_1\\
x_2
\end{array} \right),
$$
which is the identity operator for $x_1 \in \text{span} \, \{f_n, n \in N_1 \}$, $x_2 \in \text{span}\,  \{e_n, n \in N_1 \}$.

(ii) may be verified by straightforward computation.

\end{proof}

The adjoint operator of $(S_1(t), t \geq 0)$ is introduced in Lemma \ref{form of S1 adjoint}.
\begin{lemma}\label{form of S1 adjoint}
For all $x = (x_1, x_2)^{\top} \in V_1$ we have
$$
S_1^*(t) \left( \begin{array}{r}
x_1\\
x_2
\end{array} \right) = \left( \begin{array}{cc}
r_{11}(t)&r_{12}(t)\\
r_{21}(t)&r_{22}(t)
\end{array} \right) \left( \begin{array}{r}
x_1\\
x_2
\end{array} \right), \quad \forall t \geq 0,
$$
where
\begin{align}
r_{11}(t) &= e^{-at} (-A)^{- \frac{1}{2}} \left( \cos (\beta t) + a \beta^{-1} \sin (\beta t) \right)(-A)^{\frac{1}{2}}, \notag \\
r_{12}(t) &= e^{-at} (-A)^{- \frac{1}{2}} \left( - \beta - a^2 \beta^{-1} \right) \sin (\beta t)(-A)^{- \frac{1}{2}}, \notag \\
r_{21}(t) &= e^{-at} (-A)^{\frac{1}{2}} \beta^{-1} \sin (\beta t)(-A)^{\frac{1}{2}}, \notag \\
r_{22}(t) &= e^{-at} \left( \cos (\beta t) - a \beta^{-1} \sin (\beta t) \right) \notag.
\end{align}
\end{lemma}
\begin{proof}
It is easy to verify that
$$
\left\langle S_1(t)x, y \right\rangle_V = \left\langle x, S_1^*(t)y \right\rangle_V, \quad \forall x, y \in V_1, \quad \forall t \geq 0.
$$
\end{proof}

Using Lemma \ref{form of S1 adjoint}, it is possible to compute the integrand in \eqref{qinfty} and consequently to obtain the exact formula for the covariance operator $Q_{\infty}^{(a,b)}$.

\begin{lemma}\label{form of Q infinity}
The covariance operator $Q_{\infty}^{(a,b)}$ takes the form
\begin{align}
Q_{\infty}^{(a,b)}\left( \begin{array}{c}
x_1\\
x_2
\end{array} \right) &= \sum_{n \in N_1} \sum_{k \in N_1} \frac{\left\langle Qe_n, e_k \right\rangle_{L^2(D)}}{b^2 (\alpha_n - \alpha_k)^2 + 8a^2 b (\alpha_n + \alpha_k)} \times \\
&\left( \begin{array}{c}
4a \alpha_n \left\langle x_1, e_n \right\rangle_{L^2(D)} e_k + b(\alpha_k - \alpha_n) \left\langle x_2, e_n \right\rangle_{L^2(D)} e_k\\
b \alpha_n(\alpha_n - \alpha_k) \left\langle x_1, e_n \right\rangle_{L^2(D)} e_k + 2ab (\alpha_n + \alpha_k) \left\langle x_2, e_n \right\rangle_{L^2(D)} e_k
\end{array} \right), \label{Q infinity for N1}
\end{align}
for any $(x_1, x_2)^{\top} \in V_1.$
\end{lemma}

\begin{proof}
The integrand in \eqref{qinfty} can be computed as follows
\begin{equation}
Q_{\infty}^{(a,b)} = \int_0^{\infty} S_1(t) \Phi \Phi^* S_1^*(t) \, dt = \int_0^{\infty} \left( \begin{array}{cc}
q_{11}(t)&q_{12}(t)\\
q_{21}(t)&q_{22}(t)
\end{array} \right) \, dt,
\end{equation}
where
\begin{align}
q_{11}(t) &= e^{-2at} \beta^{-1} \sin (\beta t) Q (-A)^{\frac{1}{2}} \beta^{-1} \sin (\beta t) (-A)^{\frac{1}{2}}, \\
q_{12}(t) &= e^{-2at} \beta^{-1} \sin (\beta t) Q \left( \cos (\beta t) - a \beta^{-1} \sin (\beta t) \right), \label{q12} \\
q_{21}(t) &= e^{-2at} \left( \cos (\beta t) - a \beta^{-1} \sin (\beta t) \right) Q (-A)^{\frac{1}{2}} \beta^{-1} \sin (\beta t) (-A)^{\frac{1}{2}}, \label{q21} \\
q_{22}(t) &= e^{-2at} \left( \cos (\beta t) - a \beta^{-1} \sin (\beta t) \right) Q \left( \cos (\beta t) - a \beta^{-1} \sin (\beta t) \right).
\end{align}

We need to evaluate the integrals of $q_{11}(t)$, $q_{12}(t)$, $q_{21}(t)$ and $q_{22}(t)$. For every $x = (x_1, x_2)^{\top} \in V_1$, we have that
\begin{align}
q_{11}(t) x_1 &= e^{-2at} \beta^{-1} \sin (\beta t) Q \sum_{n \in N_1} \alpha_n \frac{\sin \left(\sqrt{b \alpha_n - a^2} \, t \right)}{\sqrt{b \alpha_n - a^2}} \left\langle x_1, e_n \right\rangle_{L^2(D)} e_n \notag \\
&= e^{-2at} \beta^{-1} \sin (\beta t) \sum_{n \in N_1} \sum_{k=1}^{\infty} \alpha_n \frac{\sin \left(\sqrt{b \alpha_n - a^2} \, t \right)}{\sqrt{b \alpha_n - a^2}} \left\langle Qe_n, e_k \right\rangle_{L^2(D)} \left\langle x_1, e_n \right\rangle_{L^2(D)} e_k \notag \\
&= e^{-2at} \sum_{n \in N_1} \sum_{k \in N_1} \alpha_n \frac{\sin \left(\sqrt{b \alpha_n - a^2} \, t \right)}{\sqrt{b \alpha_n - a^2}} \frac{\sin \left(\sqrt{b \alpha_k - a^2} \, t \right)}{\sqrt{b \alpha_k - a^2}} \times \notag \\
&\phantom{=} \left\langle Qe_n, e_k \right\rangle_{L^2(D)} \left\langle x_1, e_n \right\rangle_{L^2(D)} e_k. \notag
\end{align}

Now we use the fact that
$$
\int_0^{\infty} e^{-2at} \sin \left(\sqrt{b \alpha_n - a^2} \, t \right) \sin \left(\sqrt{b \alpha_k - a^2} \, t \right) \, dt = \frac{4a \sqrt{b \alpha_n - a^2} \sqrt{b \alpha_k - a^2}}{b^2 (\alpha_n - \alpha_k)^2 + 8a^2 b (\alpha_n + \alpha_k)}.
$$

Hence by integrating the formula for $q_{11}(t) x_1$ over $t$ from zero to infinity, we will arrive at
$$
\left( \int_0^{\infty} q_{11}(t) \, dt \right) x_1 = \sum_{n \in N_1} \sum_{k \in N_1} \frac{4a \alpha_n \left\langle Qe_n, e_k \right\rangle_{L^2(D)}}{b^2 (\alpha_n - \alpha_k)^2 + 8a^2 b (\alpha_n + \alpha_k)} \left\langle x_1, e_n \right\rangle_{L^2(D)} e_k.
$$

As for $q_{12}(t)$,
\begin{align}
q_{12}(t) x_2 &= e^{-2at} \beta^{-1} \sin (\beta t) Q \sum_{n \in N_1} \left( \cos \left( \sqrt{b \alpha_n - a^2} \, t \right) - a \frac{\sin \left(\sqrt{b \alpha_n - a^2} \, t \right)}{\sqrt{b \alpha_n - a^2}} \right) \times \notag \\
&\phantom{=} \left\langle x_2, e_n \right\rangle_{L^2(D)} e_n \notag \\
&= e^{-2at} \beta^{-1} \sin (\beta t) \sum_{n \in N_1} \sum_{k=1}^{\infty} \left( \cos \left( \sqrt{b \alpha_n - a^2} \, t \right) - a \frac{\sin \left(\sqrt{b \alpha_n - a^2} \, t \right)}{\sqrt{b \alpha_n - a^2}} \right) \times \notag \\
&\phantom{=} \left\langle Qe_n, e_k \right\rangle_{L^2(D)} \left\langle x_2, e_n \right\rangle_{L^2(D)} e_k \notag \\
&= e^{-2at} \sum_{n \in N_1} \sum_{k \in N_1} \left( \cos \left( \sqrt{b \alpha_n - a^2} \, t \right) - a \frac{\sin \left(\sqrt{b \alpha_n - a^2} \, t \right)}{\sqrt{b \alpha_n - a^2}} \right) \times \notag \\
&\phantom{=} \frac{\sin \left(\sqrt{b \alpha_k - a^2} \, t \right)}{\sqrt{b \alpha_k - a^2}} \left\langle Qe_n, e_k \right\rangle_{L^2(D)} \left\langle x_2, e_n \right\rangle_{L^2(D)} e_k. \notag
\end{align}

Now we use the fact that
\begin{align}
&\int_0^{\infty} e^{-2at} \sin \left(\sqrt{b \alpha_k - a^2} \, t \right) \left( \cos \left( \sqrt{b \alpha_n - a^2} \, t \right) - a \frac{\sin \left(\sqrt{b \alpha_n - a^2} \, t \right)}{\sqrt{b \alpha_n - a^2}} \right) \, dt \notag \\
&\phantom{=}= \frac{b (\alpha_k - \alpha_n) \sqrt{b \alpha_k - a^2}}{b^2 (\alpha_n - \alpha_k)^2 + 8a^2 b (\alpha_n + \alpha_k)}. \notag
\end{align}

Hence by integrating the formula for $q_{12}(t) x_2$ over $t$ from zero to infinity, we obtain
$$
\left( \int_0^{\infty} q_{12}(t) \, dt \right) x_2 = \sum_{n \in N_1} \sum_{k \in N_1} \frac{b (\alpha_k - \alpha_n) \left\langle Qe_n, e_k \right\rangle_{L^2(D)}}{b^2 (\alpha_n - \alpha_k)^2 + 8a^2 b (\alpha_n + \alpha_k)} \left\langle x_2, e_n \right\rangle_{L^2(D)} e_k.
$$

The expression for $q_{21}(t)x_1$ is very similar to the previous one,
\begin{align}
q_{21}(t) x_1 &= e^{-2at} \left( \cos (\beta t) - a \beta^{-1} \sin (\beta t) \right) Q \sum_{n \in N_1} \alpha_n \frac{\sin \left(\sqrt{b \alpha_n - a^2} \, t \right)}{\sqrt{b \alpha_n - a^2}} \left\langle x_1, e_n \right\rangle_{L^2(D)} e_n \notag \\
&= e^{-2at} \left( \cos (\beta t) - a \beta^{-1} \sin (\beta t) \right) \sum_{n \in N_1} \sum_{k=1}^{\infty} \alpha_n \frac{\sin \left(\sqrt{b \alpha_n - a^2} \, t \right)}{\sqrt{b \alpha_n - a^2}} \times \notag \\
&\phantom{=} \left\langle Qe_n, e_k \right\rangle_{L^2(D)} \left\langle x_1, e_n \right\rangle_{L^2(D)} e_k \notag \\
&= e^{-2at} \sum_{n \in N_1} \sum_{k \in N_1} \alpha_n \left( \cos \left( \sqrt{b \alpha_k - a^2} \, t \right) - a \frac{\sin \left(\sqrt{b \alpha_k - a^2} \, t \right)}{\sqrt{b \alpha_k - a^2}} \right) \times \notag \\
&\phantom{=} \frac{\sin \left(\sqrt{b \alpha_n - a^2} \, t \right)}{\sqrt{b \alpha_n - a^2}} \left\langle Qe_n, e_k \right\rangle_{L^2(D)} \left\langle x_2, e_n \right\rangle_{L^2(D)} e_k. \notag
\end{align}

Here the integration over $t$ from zero to infinity yields the same result as before with indicies $n$ and $k$ reversed (note that the denominator in the resulting formula will remain the same). Hence we obtain that
$$
\left( \int_0^{\infty} q_{21}(t) \, dt \right) x_1 = \sum_{n \in N_1} \sum_{k \in N_1} \frac{b \alpha_n (\alpha_n - \alpha_k) \left\langle Qe_n, e_k \right\rangle_{L^2(D)}}{b^2 (\alpha_n - \alpha_k)^2 + 8a^2 b (\alpha_n + \alpha_k)} \left\langle x_1, e_n \right\rangle_{L^2(D)} e_k.
$$

In a similar manner, we have that
\begin{align}
q_{22}(t) x_2 &= e^{-2at} \sum_{n \in N_1} \sum_{k \in N_1} \left( \cos \left( \sqrt{b \alpha_k - a^2} \, t \right) - a \frac{\sin \left(\sqrt{b \alpha_k - a^2} \, t \right)}{\sqrt{b \alpha_k - a^2}} \right) \times \notag \\
&\phantom{=} \left( \cos \left( \sqrt{b \alpha_n - a^2} \, t \right) - a \frac{\sin \left(\sqrt{b \alpha_n - a^2} \, t \right)}{\sqrt{b \alpha_n - a^2}} \right) \left\langle Qe_n, e_k \right\rangle_{L^2(D)} \left\langle x_2, e_n \right\rangle_{L^2(D)} e_k \notag
\end{align}

and by evaluating the appropriate integral, we arrive at
$$
\left( \int_0^{\infty} q_{22}(t) \, dt \right) x_2 = \sum_{n \in N_1} \sum_{k \in N_1} \frac{2ab (\alpha_n + \alpha_k) \left\langle Qe_n, e_k \right\rangle_{L^2(D)}}{b^2 (\alpha_n - \alpha_k)^2 + 8a^2 b (\alpha_n + \alpha_k)} \left\langle x_2, e_n \right\rangle_{L^2(D)} e_k.
$$

These results may be summarized by the formula \eqref{Q infinity for N1}, which completes the proof.
\end{proof}

\subsection{Case $\alpha_n < \frac{a^2}{b}$}
In the case $\alpha_n < \frac{a^2}{b}$, the eigenvalues $\{ l_n^{1,2}, n \in \mathbb N \}$ of the operator $\mathcal A$ are
\begin{align}
l_n^1 &= -a + \sqrt{a^2 - b \alpha_n}, \\
l_n^2 &= -a - \sqrt{a^2 - b \alpha_n}
\end{align}
and the operator $\mathcal A$ generates a $C_0$-semigroup on $V$, which is also exponentially stable (the eigenvalues $l_n^1$ and $l_n^2$ are negative). The form of the semigroup $(S_2(t), t \geq 0)$ is given in Lemma \ref{form of S2}, but let us again introduce some operators, which will be needed further.

First define the operator $P_2$
\begin{equation}
P_2 x = \sum_{n \in N_2} \left\langle x, e_n \right\rangle_{L^2(D)} e_n,
\end{equation}
which is the operator of projection on the $\text{span} \, \{ e_n, n \in N_2 \}$ (that is $P_2 : L^2(D) \rightarrow \text{span} \, \{ e_n, n \in N_2 \}$). Furthermore define the operator $\gamma : L^2(D) \rightarrow L^2(D)$ by $\gamma = \left( a^2 I + bA \right)^{\frac{1}{2}} P_2$, that is
\begin{equation}
\gamma x = \sum_{n \in N_2} \sqrt{a^2 - b \alpha_n} \left\langle x, e_n \right\rangle_{L^2(D)} e_n,
\end{equation}
where $x \in L^2(D)$. (Since the sum over the set $N_2$ is finite, it is possible to define the operator $\gamma$ on the whole space $L^2(D)$.)

Similarly, define
\begin{align}
\gamma^{-1}x &= \sum_{n \in N_2} \frac{1}{\sqrt{a^2 - b \alpha_n}} \left\langle x, e_n \right\rangle_{L^2(D)} e_n, \\
L_1 x &= \left( -aP_2 + \gamma \right) x, \\
L_2 x &= \left( -aP_2 - \gamma \right) x, \\
e^{L_1 t} x &= \sum_{n \in N_2} e^{l_n^1 t} \left\langle x, e_n \right\rangle_{L^2(D)} e_n, \\
e^{L_2 t} x &= \sum_{n \in N_2} e^{l_n^2 t} \left\langle x, e_n \right\rangle_{L^2(D)} e_n,
\end{align}
where $x \in L^2(D)$.

Note that $\gamma^{-1} = \left( a^2 I + bA \right)^{-\frac{1}{2}}P_2$, so $\gamma^{-1} \gamma x = P_2 x$ for any $x \in L^2(D)$ and $\gamma^{-1} \gamma x = Ix$ for any $x \in \text{span} \, \{ e_n, n \in N_2 \}$. Also note that the following properties hold true
\begin{align}
L_1 - L_2 &= 2 \gamma, \\
L_1 L_2 &= - bAP_2 \, (= L_2 L_1), \label{L1 L2}
\end{align}
so the operators $L_1$ and $L_2$ commute. The last remark is that the operator $e^{L_1 t}$ evaluated at time $t=0$ is $e^{L_1 t}|_{t=0} \, x = P_2 x$ for any $x \in L^2(D)$. (The operator $e^{L_2 t}$ has indeed the same property.)

The form of the semigroup $(S_2(t), t \geq 0)$, for the coordinates from the set $N_2$, is described by the following Lemma.

\begin{lemma}\label{form of S2}
For all $x = (x_1, x_2)^{\top} \in V_2$ we have
$$
S_2(t) \left( \begin{array}{r}
x_1\\
x_2
\end{array} \right) = \left( \begin{array}{cc}
\frac{1}{2} \gamma^{-1} \left( - L_2 e^{L_1 t} + L_1 e^{L_2 t} \right) & \frac{1}{2} \gamma^{-1} \left( e^{L_1 t} - e^{L_2 t} \right) \\
\frac{1}{2} \gamma^{-1} \left( - L_1 L_2 e^{L_1 t} + L_1 L_2 e^{L_2 t} \right) & \frac{1}{2} \gamma^{-1} \left( L_1 e^{L_1 t} - L_2 e^{L_2 t} \right)
\end{array} \right) \left( \begin{array}{r}
x_1\\
x_2
\end{array} \right),
$$
for all $t \geq 0$.
\end{lemma}
\begin{proof}
Analogously to the proof of Lemma \ref{form of S1}, it is sufficient to show that 
\begin{itemize}
\item[(i)] $$
S_2(0) = \left( \begin{array}{cc}
I&0\\
0&I
\end{array} \right),
$$
\item[(ii)] $$
\frac{d}{dt} S_2(t)x = \mathcal A S_2(t)x, \quad \forall x \in V_2, \quad t \geq 0.
$$
\end{itemize}

As for (i), it is just matter of evaluating the operators at time $t=0$ and simplifying. For example the upper--left operator simplifies as follows
$$
\left. \frac{1}{2} \gamma^{-1} \left( - L_2 e^{L_1 t} + L_1 e^{L_2 t} \right) \right|_{t=0} = \frac{1}{2} \gamma^{-1} \left( L_1 - L_2 \right) P_2 = \gamma^{-1} \gamma P_2 = P_2.
$$

Consequently we arrive at
$$
S_2(0) \left( \begin{array}{r}
x_1\\
x_2
\end{array} \right) = \left( \begin{array}{cc}
P_2 & 0 \\
0 & P_2
\end{array} \right) \left( \begin{array}{r}
x_1\\
x_2
\end{array} \right),
$$
which is an identity operator for $x_1 \in \text{span} \, \{ f_n, n \in N_2 \}$, $x_2 \in \text{span} \, \{ e_n, n \in N_2 \}$.

(ii) may be verified by straightforward computation.
\end{proof}

The adjoint operator of $(S_2(t), t \geq 0)$ is introduced in Lemma \ref{form of S2 adjoint}.
\begin{lemma}\label{form of S2 adjoint}
For all $x = (x_1, x_2)^{\top} \in V_2$ we have
$$
S_2^*(t) \left( \begin{array}{r}
x_1\\
x_2
\end{array} \right) = \left( \begin{array}{cc}
r_{11}(t)&r_{12}(t)\\
r_{21}(t)&r_{22}(t)
\end{array} \right) \left( \begin{array}{r}
x_1\\
x_2
\end{array} \right), \quad \forall t \geq 0,
$$
where
\begin{align}
r_{11}(t) &= \frac{1}{2} \gamma^{-1} \left( - L_2 e^{L_1 t} + L_1 e^{L_2 t} \right), \notag \\
r_{12}(t) &= \frac{1}{2} (-A)^{- \frac{1}{2}} \gamma^{-1} \left( - L_1 L_2 e^{L_1 t} + L_1 L_2 e^{L_2 t} \right) (-A)^{- \frac{1}{2}}, \notag \\
r_{21}(t) &= \frac{1}{2} (-A)^{\frac{1}{2}} \gamma^{-1} \left( e^{L_1 t} - e^{L_2 t} \right) (-A)^{\frac{1}{2}}, \notag \\
r_{22}(t) &= \frac{1}{2} \gamma^{-1} \left( L_1 e^{L_1 t} - L_2 e^{L_2 t} \right) \notag.
\end{align}
\end{lemma}
\begin{proof}
It is possible to verify that
$$
\left\langle S_2(t)x, y \right\rangle_V = \left\langle x, S_2^*(t)y \right\rangle_V, \quad \forall x, y \in V_2, \quad \forall t \geq 0.
$$
\end{proof}

Using Lemma \ref{form of S2} and Lemma \ref{form of S2 adjoint}, it is possible to compute the integrand in \eqref{qinfty} and to obtain the formula for the covariance operator $Q_{\infty}^{(a,b)}$ for the case $\alpha_n < \frac{a^2}{b}$.

\begin{lemma}\label{form of Q infinity second case}
The covariance operator $Q_{\infty}^{(a,b)}$ takes the form
\begin{align}
Q_{\infty}^{(a,b)}\left( \begin{array}{c}
x_1\\
x_2
\end{array} \right) &= \sum_{n \in N_2} \sum_{k \in N_2} \frac{\left\langle Qe_n, e_k \right\rangle_{L^2(D)}}{b^2 (\alpha_n - \alpha_k)^2 + 8a^2 b (\alpha_n + \alpha_k)} \times \\
&\left( \begin{array}{c}
4a \alpha_n \left\langle x_1, e_n \right\rangle_{L^2(D)} e_k + b(\alpha_k - \alpha_n) \left\langle x_2, e_n \right\rangle_{L^2(D)} e_k\\
b \alpha_n(\alpha_n - \alpha_k) \left\langle x_1, e_n \right\rangle_{L^2(D)} e_k + 2ab (\alpha_n + \alpha_k) \left\langle x_2, e_n \right\rangle_{L^2(D)} e_k
\end{array} \right), \label{Q infinity for N2}
\end{align}
for any $(x_1, x_2)^{\top} \in V_2.$
\end{lemma}
\begin{proof}
According to \eqref{qinfty}, the covariance operator $Q_{\infty}^{(a,b)}$ may be expressed as
\begin{equation}
Q_{\infty}^{(a,b)} = \int_0^{\infty} S_2(t) \Phi \Phi^* S_2^*(t) \, dt = \int_0^{\infty} \left( \begin{array}{cc}
q_{11}(t)&q_{12}(t)\\
q_{21}(t)&q_{22}(t)
\end{array} \right) \, dt,
\end{equation}
where
\begin{align}
q_{11}(t) &= \frac{1}{4} \gamma^{-1} \left( e^{L_1 t} - e^{L_2 t} \right) Q (-A)^{\frac{1}{2}} \gamma^{-1} \left( e^{L_1 t} - e^{L_2 t} \right) (-A)^{\frac{1}{2}}, \\
q_{12}(t) &= \frac{1}{4} \gamma^{-1} \left( e^{L_1 t} - e^{L_2 t} \right) Q \gamma^{-1} \left( L_1 e^{L_1 t} - L_2 e^{L_2 t} \right), \\
q_{21}(t) &= \frac{1}{4} \gamma^{-1} \left( L_1 e^{L_1 t} - L_2 e^{L_2 t} \right) Q (-A)^{\frac{1}{2}} \gamma^{-1} \left( e^{L_1 t} - e^{L_2 t} \right) (-A)^{\frac{1}{2}}, \\
q_{22}(t) &= \frac{1}{4} \gamma^{-1} \left( L_1 e^{L_1 t} - L_2 e^{L_2 t} \right) Q \gamma^{-1} \left( L_1 e^{L_1 t} - L_2 e^{L_2 t} \right).
\end{align}

As in the proof of Lemma \ref{form of Q infinity}, we need to evaluate the integrals of $q_{11}(t)$, $q_{12}(t)$, $q_{21}(t)$ and $q_{22}(t)$. For every $x = (x_1, x_2)^{\top} \in V_2$, we have that
\begin{align}
q_{11}(t)x_1 &= \frac{1}{4} \gamma^{-1} \left( e^{L_1 t} - e^{L_2 t} \right) Q \sum_{n \in N_2} \alpha_n \frac{e^{l_n^1 t} - e^{l_n^2 t}}{\sqrt{a^2 - b \alpha_n}} \left\langle x_1, e_n \right\rangle_{L^2(D)} e_n \notag \\
&= \frac{1}{4} \gamma^{-1} \left( e^{L_1 t} - e^{L_2 t} \right) \sum_{n \in N_2} \sum_{k=1}^{\infty} \alpha_n \frac{e^{l_n^1 t} - e^{l_n^2 t}}{\sqrt{a^2 - b \alpha_n}} \left\langle Qe_n, e_k \right\rangle_{L^2(D)} \left\langle x_1, e_n \right\rangle_{L^2(D)} e_k \notag \\
&= \frac{1}{4} \sum_{n \in N_2} \sum_{k \in N_2} \alpha_n \frac{e^{l_n^1 t} - e^{l_n^2 t}}{\sqrt{a^2 - b \alpha_n}} \frac{e^{l_k^1 t} - e^{l_k^2 t}}{\sqrt{a^2 - b \alpha_k}} \left\langle Qe_n, e_k \right\rangle_{L^2(D)} \left\langle x_1, e_n \right\rangle_{L^2(D)} e_k. \notag
\end{align}

If we now use the fact that
$$
\int_0^{\infty} \left( e^{l_n^1 t} - e^{l_n^2 t} \right) \left( e^{l_k^1 t} - e^{l_k^2 t} \right) \, dt = \frac{16a \sqrt{a^2 - b \alpha_n} \sqrt{a^2 - b \alpha_k}}{b^2 (\alpha_n - \alpha_k)^2 + 8a^2 b (\alpha_n + \alpha_k)},
$$
we arrive at
$$
\left( \int_0^{\infty} q_{11}(t) \, dt \right) x_1 = \sum_{n \in N_2} \sum_{k \in N_2} \frac{4a \alpha_n \left\langle Qe_n, e_k \right\rangle_{L^2(D)}}{b^2 (\alpha_n - \alpha_k)^2 + 8a^2 b (\alpha_n + \alpha_k)} \left\langle x_1, e_n \right\rangle_{L^2(D)} e_k.
$$

As for the operator $q_{12}(t)$
\begin{align}
q_{12}(t) x_2 &= \frac{1}{4} \gamma^{-1} \left( e^{L_1 t} - e^{L_2 t} \right) Q \sum_{n \in N_2} \frac{l_n^1 e^{l_n^1 t} - l_n^2 e^{l_n^2 t}}{\sqrt{a^2 - b \alpha_n}} \left\langle x_2, e_n \right\rangle_{L^2(D)} e_n \notag \\
&= \frac{1}{4} \gamma^{-1} \left( e^{L_1 t} - e^{L_2 t} \right) \sum_{n \in N_2} \sum_{k=1}^{\infty} \frac{l_n^1 e^{l_n^1 t} - l_n^2 e^{l_n^2 t}}{\sqrt{a^2 - b \alpha_n}} \left\langle Qe_n, e_k \right\rangle_{L^2(D)} \left\langle x_2, e_n \right\rangle_{L^2(D)} e_k \notag \\
&= \frac{1}{4} \sum_{n \in N_2} \sum_{k \in N_2} \frac{l_n^1 e^{l_n^1 t} - l_n^2 e^{l_n^2 t}}{\sqrt{a^2 - b \alpha_n}} \frac{e^{l_k^1 t} - e^{l_k^2 t}}{\sqrt{a^2 - b \alpha_k}} \left\langle Qe_n, e_k \right\rangle_{L^2(D)} \left\langle x_2, e_n \right\rangle_{L^2(D)} e_k. \notag
\end{align}

Now we use the fact that
$$
\int_0^{\infty} \left( l_n^1 e^{l_n^1 t} - l_n^2 e^{l_n^2 t} \right) \left( e^{l_k^1 t} - e^{l_k^2 t} \right) \, dt = \frac{4b(\alpha_k - \alpha_n) \sqrt{a^2 - b \alpha_n} \sqrt{a^2 - b \alpha_k}}{b^2 (\alpha_n - \alpha_k)^2 + 8a^2 b (\alpha_n + \alpha_k)}.
$$

Hence by integrating the formula for $q_{12}(t)x_2$ over $t$ from zero to infinity, we obtain
$$
\left( \int_0^{\infty} q_{12}(t) \, dt \right) x_2 = \sum_{n \in N_2} \sum_{k \in N_2} \frac{b (\alpha_k - \alpha_n) \left\langle Qe_n, e_k \right\rangle_{L^2(D)}}{b^2 (\alpha_n - \alpha_k)^2 + 8a^2 b (\alpha_n + \alpha_k)} \left\langle x_2, e_n \right\rangle_{L^2(D)} e_k.
$$

The expression for $q_{21}(t)x_1$ is similar to the previous one,
\begin{align}
q_{21}(t)x_1 &= \frac{1}{4} \gamma^{-1} \left( L_1 e^{L_1 t} - L_2 e^{L_2 t} \right) Q \sum_{n \in N_2} \alpha_n \frac{e^{l_n^1 t} - e^{l_n^2 t}}{\sqrt{a^2 - b \alpha_n}} \left\langle x_1, e_n \right\rangle_{L^2(D)} e_n \notag \\
&= \frac{1}{4} \gamma^{-1} \left( L_1 e^{L_1 t} - L_2 e^{L_2 t} \right) \sum_{n \in N_2} \sum_{k=1}^{\infty} \alpha_n \frac{e^{l_n^1 t} - e^{l_n^2 t}}{\sqrt{a^2 - b \alpha_n}} \left\langle Qe_n, e_k \right\rangle_{L^2(D)} \left\langle x_1, e_n \right\rangle_{L^2(D)} e_k \notag \\
&= \frac{1}{4} \sum_{n \in N_2} \sum_{k \in N_2} \frac{e^{l_n^1 t} - e^{l_n^2 t}}{\sqrt{a^2 - b \alpha_n}} \frac{l_k^1 e^{l_k^1 t} - l_k^2 e^{l_k^2 t}}{\sqrt{a^2 - b \alpha_k}} \left\langle Qe_n, e_k \right\rangle_{L^2(D)} \left\langle x_1, e_n \right\rangle_{L^2(D)} e_k. \notag
\end{align}

The integration over $t$ from zero to infinity yields the same result as before with indicies $n$ and $k$ reversed. Hence we obtain that
$$
\left( \int_0^{\infty} q_{21}(t) \, dt \right) x_1 = \sum_{n \in N_2} \sum_{k \in N_2} \frac{b \alpha_n (\alpha_n - \alpha_k) \left\langle Qe_n, e_k \right\rangle_{L^2(D)}}{b^2 (\alpha_n - \alpha_k)^2 + 8a^2 b (\alpha_n + \alpha_k)} \left\langle x_1, e_n \right\rangle_{L^2(D)} e_k.
$$

In a similar manner, we have that
$$
q_{22}(t) x_2 = \frac{1}{4} \sum_{n \in N_2} \sum_{k \in N_2} \frac{l_n^1 e^{l_n^1 t} - l_n^2 e^{l_n^2 t}}{\sqrt{a^2 - b \alpha_n}} \frac{l_k^1 e^{l_k^1 t} - l_k^2 e^{l_k^2 t}}{\sqrt{a^2 - b \alpha_k}} \left\langle Qe_n, e_k \right\rangle_{L^2(D)} \left\langle x_2, e_n \right\rangle_{L^2(D)} e_k
$$

and by evaluating the appropriate integral, we arrive at
$$
\left( \int_0^{\infty} q_{22}(t) \, dt \right) x_2 = \sum_{n \in N_2} \sum_{k \in N_2} \frac{2ab (\alpha_n + \alpha_k) \left\langle Qe_n, e_k \right\rangle_{L^2(D)}}{b^2 (\alpha_n - \alpha_k)^2 + 8a^2 b (\alpha_n + \alpha_k)} \left\langle x_2, e_n \right\rangle_{L^2(D)} e_k.
$$

These results may be summarized by the formula \eqref{Q infinity for N2}, which completes the proof.
\end{proof}

\subsection{Case $\alpha_n = \frac{a^2}{b}$}
In the case $\alpha_n = \frac{a^2}{b}$, the situation is much easier. The eigenvalue of the operator $\mathcal A$ is $-a$, so the operator $\mathcal A$ generates $C_0$--semigroup on $V$, which is also exponentially stable.

Define the operator $P_3$ in a similar fashion as $P_1$ and $P_2$ above
\begin{equation}
P_3 x = \sum_{n \in N_3} \left\langle x, e_n \right\rangle_{L^2(D)} e_n.
\end{equation}

That is the operator of projection on the $\text{span} \, \{ e_n, n \in N_3 \}$ (that is $P_3 : L^2(D) \rightarrow \text{span} \, \{ e_n, n \in N_3 \}$). The form of semigroup $(S_3(t), t \geq 0)$ is given by the following Lemma.

\begin{lemma}\label{form of S3}
For all $x = (x_1, x_2)^{\top} \in V_3$ we have
$$
S_3(t) \left( \begin{array}{r}
x_1\\
x_2
\end{array} \right) = \left( \begin{array}{cc}
(1 + at) e^{-at} P_3 & te^{-at} P_3 \\
-a^2 t e^{-at} P_3 & (1 - at) e^{-at} P_3
\end{array} \right) \left( \begin{array}{r}
x_1\\
x_2
\end{array} \right), \quad \forall t \geq 0.
$$
\end{lemma}
\begin{proof}
If we evaluate the above operator $S_3(t)$ at time $t=0$, we obtain
$$
S_3(0) \left( \begin{array}{r}
x_1\\
x_2
\end{array} \right) = \left( \begin{array}{cc}
P_3 & 0 \\
0 & P_3
\end{array} \right) \left( \begin{array}{r}
x_1\\
x_2
\end{array} \right),
$$
which is an identity operator for $x_1 \in \text{span} \, \{ f_n, n \in N_3 \}$, $x_2 \in \text{span} \, \{ e_n, n \in N_3 \}$. The property
$$
\frac{d}{dt} S_3(t)x = \mathcal A S_3(t)x, \quad \forall x \in V_3, \quad \forall t \geq 0,
$$
may be verified by straightforward computation.
%
%The right--hand side equals to
%$$
%\frac{d}{dt} S_3(t) = \left( \begin{array}{cc}
%-a^2 t e^{-at} P_3 & (1 - at) e^{-at} P_3 \\
%(a^3 t - a^2) e^{-at} P_3 & (a^2 t - 2a) e^{-at} P_3
%\end{array} \right).
%$$
%The left--hand side equals to
%\begin{align}
%\mathcal A S_3(t) &= \left( \begin{array}{cc}
%0 & I \\
%bA & -2aI
%\end{array} \right) \left( \begin{array}{cc}
%(1 + at) e^{-at} P_3 & te^{-at} P_3 \\
%-a^2 t e^{-at} P_3 & (1 - at) e^{-at} P_3
%\end{array} \right) \notag \\
%&= \left( \begin{array}{cc}
%-a^2 t e^{-at} P_3 & (1 - at) e^{-at} P_3 \\
%(1 + at) e^{-at} bAP_3 + 2a^3 t e^{-at} P_3 & t e^{-at} bAP_3 - 2a(1 - at) e^{-at}P_3
%\end{array} \right), \notag
%\end{align}
%which coincides with the right--hand side. During simplifications, it is used the fact that in the case $\alpha_n = \frac{a^2}{b}$, it is true that $bAP_3x = -a^2 P_3 x$ for any $x \in L^2(D)$.
\end{proof}

The adjoint operator of $(S_3(t), t \geq 0)$ is introduced in the following Lemma.
\begin{lemma}\label{form of S3 adjoint}
For all $x = (x_1, x_2)^{\top} \in V_3$ we have
$$
S_3^*(t) \left( \begin{array}{r}
x_1\\
x_2
\end{array} \right) = \left( \begin{array}{cc}
(1 + at) e^{-at} P_3 & -bt e^{-at} P_3 \\
\frac{a^2}{b} t e^{-at} P_3 & (1 - at) e^{-at} P_3
\end{array} \right) \left( \begin{array}{r}
x_1\\
x_2
\end{array} \right), \quad \forall t \geq 0.
$$
\end{lemma}
\begin{proof}
It is possible to verify that
$$
\left\langle S_3(t)x, y \right\rangle_V = \left\langle x, S_3^*(t)y \right\rangle_V, \quad \forall x, y \in V_3, \quad \forall t \geq 0.
$$
\end{proof}

As in the two previous cases, we may use Lemma \ref{form of S3} and Lemma \ref{form of S3 adjoint} to compute the covariance operator $Q_{\infty}^{(a,b)}$.
\begin{lemma}\label{form of Q infinity third case}
The covariance operator $Q_{\infty}^{(a,b)}$ takes the form
\begin{equation}
Q_{\infty}^{(a,b)} \left( \begin{array}{c}
x_1\\
x_2
\end{array} \right) = \left( \begin{array}{cc}
\frac{1}{4ab} P_3 Q & 0 \\
0 & \frac{1}{4a} P_3 Q
\end{array} \right) \left( \begin{array}{c}
x_1\\
x_2
\end{array} \right), \quad \forall \left( \begin{array}{r}
x_1\\
x_2
\end{array} \right) \in V_3. \label{Q infinity for N3}
\end{equation}
\end{lemma}
\begin{proof}
According to \eqref{qinfty}, the operator $Q_{\infty}^{(a,b)}$ may be expressed as
\begin{align}
Q_{\infty}^{(a,b)} &= \int_0^{\infty} S_3(t) \Phi \Phi^* S_3^*(t) \, dt \notag \\
&= \int_0^{\infty} \left( \begin{array}{cc}
\frac{a^2}{b} t^2 e^{-2at} P_3 Q P_3 & (1 - at) t e^{-2at} P_3 Q P_3 \\
\frac{a^2}{b} t (1 - at) e^{-2at} P_3 Q P_3 & (1 - at)^2 e^{-2at} P_3 Q P_3
\end{array} \right) \, dt \notag
\end{align}
and the result is just straightforward integration. Since we consinder only $(x_1, x_2)^{\top}$ from the space $V_3$, the first (right--hand side) projection $P_3$ may be omitted.
\end{proof}

The formula \eqref{Q infinity for N3} may be also written in the form like \eqref{Q infinity for N1} in Lemma \ref{form of Q infinity} or \eqref{Q infinity for N2} in Lemma \ref{form of Q infinity second case}
$$
Q_{\infty}^{(a,b)} \left( \begin{array}{c}
x_1\\
x_2
\end{array} \right) = \sum_{n \in N_3} \sum_{k \in N_3} \left\langle Qe_n, e_k \right\rangle_{L^2(D)} \left( \begin{array}{c}
\frac{1}{4ab} \left\langle x_1, e_n \right\rangle_{L^2(D)} e_k\\
\frac{1}{4a} \left\langle x_2, e_n \right\rangle_{L^2(D)} e_k
\end{array} \right),
$$
for any $(x_1, x_2)^{\top} \in V_3$, which is in fact the same formula as \eqref{Q infinity for N1} (or \eqref{Q infinity for N2}), with $\alpha_n = \frac{a^2}{b} = \alpha_k$ and sums over the set $N_3$. It is indeed some kind of consistency of these formulae \eqref{Q infinity for N1}, \eqref{Q infinity for N2}, \eqref{Q infinity for N3}.

\subsection{Summary} \label{subsection summary}
We have computed the semigroups $(S_1(t), t \geq 0)$, $(S_2(t), t \geq 0)$ and $(S_3(t), t \geq 0)$ for the coordinates from the sets $N_1$, $N_2$ and $N_3$. The semigroup $(S(t), t \geq 0)$ (with the infinitesimal generator $\mathcal A$) is in fact combination of all of them and its form is stated in the following Theorem.

\begin{theorem} \label{form of S}
The operator $\mathcal A$ is the infinitesimal operator of the strongly continuous semigroup $(S(t), t \geq 0)$ on $V$, which takes the following form
\begin{align}
S(t) \left( \begin{array}{r}
x_1\\
x_2
\end{array} \right) &= S_1(t) \left( \begin{array}{cc}
P_1 & 0 \\
0 & P_1
\end{array} \right) \left( \begin{array}{r}
x_1\\
x_2
\end{array} \right) + S_2(t) \left( \begin{array}{cc}
P_2 & 0 \\
0 & P_2
\end{array} \right) \left( \begin{array}{r}
x_1\\
x_2
\end{array} \right) \notag \\
& \phantom{=} + S_3(t) \left( \begin{array}{cc}
P_3 & 0 \\
0 & P_3
\end{array} \right) \left( \begin{array}{r}
x_1\\
x_2
\end{array} \right), \quad \forall x = \left( \begin{array}{r}
x_1\\
x_2
\end{array} \right) \in V, \quad \forall t \geq 0.
\end{align}

Moreover, the semigroup $(S(t), t \geq 0)$ is exponentially stable.
\end{theorem}
\begin{proof}
For every $x \in V$, its projections to the space $V_i, \, i = 1, 2, 3$ are taken and then the appropriate semigroup to the appropriate coordinates is applied. From the proofs of Lemmas \ref{form of S1}, \ref{form of S2} and \ref{form of S3}, it is also clear that
\begin{itemize}
\item[(i)] $$
S(0) = \left( \begin{array}{cc}
I&0\\
0&I
\end{array} \right),
$$
\item[(ii)] $$
\frac{d}{dt} S(t)x = \mathcal A S(t)x, \quad \forall x \in \text{Dom}(\mathcal A), \quad \forall t \geq 0,
$$
\end{itemize}
which means, that this is the form of semigroup $(S(t), t \geq 0)$ with infinitesimal generator $\mathcal A$. Exponential stability is implied by exponential stability of semigroups $(S_1(t), t \geq 0)$, $(S_2(t), t \geq 0)$ and $(S_3(t), t \geq 0)$.
\end{proof}

The covariance operator $Q_{\infty}^{(a,b)}$ is in fact combined in the same way (we could have used marks $Q_{\infty, \, 1}^{(a,b)}, Q_{\infty, \, 2}^{(a,b)}$ and $Q_{\infty, \, 3}^{(a,b)}$ in the previous cases), but since \eqref{Q infinity for N1}, \eqref{Q infinity for N2} and \eqref{Q infinity for N3} coincide, the resulting formula is rather simple and is given by the following Theorem.

\begin{theorem} \label{form of Q infinity - theorem}
There is a unique invariant measure $\mu_{\infty}^{(a,b)} = N \left( 0, Q_{\infty}^{(a,b)} \right)$ for the equation \eqref{linear equation with parameters} and
$$
w^* - \lim_{t \rightarrow \infty} \mu_t^{x_0} = \mu_{\infty}^{(a,b)}
$$
for each initial condition $x_0 \in V$. The covariance operator $Q_{\infty}^{(a,b)}$ takes the form
\begin{align}
Q_{\infty}^{(a,b)} \left( \begin{array}{c}
x_1\\
x_2
\end{array} \right) &= \sum_{n=1}^{\infty} \sum_{k=1}^{\infty} \frac{\left\langle Qe_n, e_k \right\rangle_{L^2(D)}}{b^2 (\alpha_n - \alpha_k)^2 + 8a^2 b (\alpha_n + \alpha_k)} \times \\
&\left( \begin{array}{c}
4a \alpha_n \left\langle x_1, e_n \right\rangle_{L^2(D)} e_k + b(\alpha_k - \alpha_n) \left\langle x_2, e_n \right\rangle_{L^2(D)} e_k\\
b \alpha_n(\alpha_n - \alpha_k) \left\langle x_1, e_n \right\rangle_{L^2(D)} e_k + 2ab (\alpha_n + \alpha_k) \left\langle x_2, e_n \right\rangle_{L^2(D)} e_k
\end{array} \right), \label{form of Q infinity general case}
\end{align}
for any $(x_1, x_2)^{\top} \in V.$
\end{theorem}
\begin{proof}
The existence of invariant measure $\mu_{\infty}^{(a,b)}$ is given by Proposition \ref{existence of invariant measure}. The formula for the covariance operator $Q_{\infty}^{(a,b)}$ follows from Lemmas \ref{form of Q infinity}, \ref{form of Q infinity second case} and \ref{form of Q infinity third case}.
\end{proof}

\section{Parameter estimation} \label{parameter estimation}
Consider the stochastic differential equation \eqref{linear equation with parameters} with the parameters $a > 0$, $b > 0$ unknown. Our goal is to propose strongly consistent estimators of these parameters based on observation of the trajectory of the process $(X^{x_0}(t), 0 \leq t \leq T)$ up to time $T$.

Since the linear differential equation \eqref{linear equation with parameters} has unique invariant measure $\mu_{\infty}^{(a,b)}$, we may use the following ergodic theorem for arbitrary solution (see \cite{maslowskipospisil}, Theorem 4.9.).

\begin{theorem} \label{ergodic theorem}
Let $(X^{x_0}, t \geq 0)$ be a solution to \eqref{linear equation with parameters} with $\Phi \in \mathcal L_2(U,V)$. Let $\varrho : V \rightarrow \mathbb R$ be a functional satisfying the following local Lipschitz condition: let there exist real constants $K > 0$ and $m \geq 0$ such that
\begin{equation}
|\varrho (x) - \varrho (y)| \leq K \| x - y \|_V \left( 1 + \| x \|_V^m + \| y \|_V^m \right)
\end{equation}
for all $x, y \in V$. Then
\begin{equation}
\lim_{T \rightarrow \infty} \frac{1}{T} \int_0^T \varrho \left( X^{x_0}(t) \right) \, dt = \int_V \varrho (y) \, \mu_{\infty} (dy), \quad \mathbb P-a.s.
\end{equation}
for all $x_0 \in V$.
\end{theorem}

We will be specifically interested in a functional $\varrho : V \rightarrow \mathbb R$, $\varrho (y) = \| y \|_V^2$, $y \in V$. Then all the conditions of above Theorem are satisfied with $m=1$ and
\begin{align}
\lim_{T \rightarrow \infty} \frac{1}{T} \int_0^T \varrho \left( X^{x_0}(t) \right) \, dt &= \lim_{T \rightarrow \infty} \frac{1}{T} \int_0^T \| X^{x_0}(t) \|_V^2 \, dt \notag \\
&= \int_V \| y \|_V^2 \, \mu_{\infty}^{(a,b)} (dy) \notag \\
&= \text{Tr} \, Q_{\infty}^{(a,b)}, \label{ergodic property}
\end{align}
where $\text{Tr } (\cdot)$ denotes the trace of the (nuclear) operator. Hence we first introduce the trace of the operator $Q_{\infty}^{(a,b)}$.
\begin{lemma} \label{trace of Q infinity}
Trace of the nuclear operator $Q_{\infty}^{(a,b)}$ takes the form
\begin{align}
\mathrm{Tr} \, Q_{\infty}^{(a,b)} &= \frac{1}{4ab} \sum_{n=1}^{\infty} \lambda_n + \frac{1}{4a} \sum_{n=1}^{\infty} \lambda_n \label{form of Q infinity splitted} \\
&= \frac{b+1}{4ab} \, \mathrm{Tr} \, Q. \label{form of Q infinity result}
\end{align}
\end{lemma}
\begin{proof}
According to the definition of the trace
$$
\text{Tr} \, Q_{\infty}^{(a,b)} = \sum_{j=1}^{\infty} \left\langle Q_{\infty}^{(a,b)} \left( \begin{array}{c} f_j \\ 0 \end{array} \right), \left( \begin{array}{c} f_j \\ 0 \end{array} \right) \right\rangle_V + \sum_{j=1}^{\infty} \left\langle Q_{\infty}^{(a,b)} \left( \begin{array}{c} 0 \\ e_j \end{array} \right), \left( \begin{array}{c} 0 \\ e_j \end{array} \right) \right\rangle_V.
$$

With \eqref{form of Q infinity general case} in mind, we start with the summand of the first sum
\begin{align}
&\left\langle Q_{\infty}^{(a,b)} \left( \begin{array}{c} f_j \\ 0 \end{array} \right), \left( \begin{array}{c} f_j \\ 0 \end{array} \right) \right\rangle_V = \notag \\
&= \left\langle \sum_{n=1}^{\infty} \sum_{k=1}^{\infty} \frac{\left\langle Qe_n, e_k \right\rangle_{L^2(D)}}{b^2 (\alpha_n - \alpha_k)^2 + 8a^2 b (\alpha_n + \alpha_k)} \left( \begin{array}{c} 4a \alpha_n \left\langle f_j, e_n \right\rangle_{L^2(D)} e_k \\ b \alpha_n(\alpha_n - \alpha_k) \left\langle f_j, e_n \right\rangle_{L^2(D)} e_k \end{array} \right), \left( \begin{array}{c} f_j \\ 0 \end{array} \right) \right\rangle_V \notag \\
&= \sum_{n=1}^{\infty} \sum_{k=1}^{\infty} \frac{4a \alpha_n \left\langle Qe_n, e_k \right\rangle_{L^2(D)}}{b^2 (\alpha_n - \alpha_k)^2 + 8a^2 b (\alpha_n + \alpha_k)} \left\langle f_j, e_n \right\rangle_{L^2(D)} \left\langle f_j, e_k \right\rangle_{\text{Dom} (-A)^{\frac{1}{2}}}. \notag
\end{align}

Since
\begin{align}
\left\langle f_j, e_n \right\rangle_{L^2(D)} &= \frac{1}{\sqrt{\alpha_j}} \delta_{j,n}, \notag \\
\left\langle f_j, e_k \right\rangle_{\text{Dom} (-A)^{\frac{1}{2}}} &= \sqrt{\alpha_k} \delta_{j,k}, \notag
\end{align}
where $\delta$ stands for the Kronecker's delta, there is only one nonzero summand, which corresponds to $n=k=j$, so we arrive at
$$
\frac{1}{4ab} \left\langle Qe_j, e_j \right\rangle_{L^2(D)}.
$$

If we sum up these terms over $j$, we will obtain the first term on the right--hand side of \eqref{form of Q infinity splitted}, that is
$$
\frac{1}{4ab} \sum_{j=1}^{\infty} \left\langle Qe_j, e_j \right\rangle_{L^2(D)} = \frac{1}{4ab} \sum_{j=1}^{\infty} \lambda_j.
$$

Note that
$$
\text{Tr} \, Q = \sum_{j=1}^{\infty} \left\langle Qe^{\prime}_j, e^{\prime}_j \right\rangle_{L^2(D)} = \sum_{j=1}^{\infty} \lambda_j = \sum_{j=1}^{\infty} \left\langle Qe_j, e_j \right\rangle_{L^2(D)},
$$
where the last equality follows from the fact that the definition of the trace does not depend on the choice of orthonormal basis of $L^2(D)$.

In a similar fashion, we compute the summand of the second sum
\begin{align}
&\left\langle Q_{\infty}^{(a,b)} \left( \begin{array}{c} 0 \\ e_j \end{array} \right), \left( \begin{array}{c} 0 \\ e_j \end{array} \right) \right\rangle_V \notag = \\
&= \left\langle \sum_{n=1}^{\infty} \sum_{k=1}^{\infty} \frac{\left\langle Qe_n, e_k \right\rangle_{L^2(D)}}{b^2 (\alpha_n - \alpha_k)^2 + 8a^2 b (\alpha_n + \alpha_k)} \left( \begin{array}{c} b(\alpha_k - \alpha_n) \left\langle e_j, e_n \right\rangle_{L^2(D)} e_k \\ 2ab (\alpha_n + \alpha_k) \left\langle e_j, e_n \right\rangle_{L^2(D)} e_k \end{array} \right), \left( \begin{array}{c} 0 \\ e_j \end{array} \right) \right\rangle_V \notag \\
&= \sum_{n=1}^{\infty} \sum_{k=1}^{\infty} \frac{2ab (\alpha_n + \alpha_k) \left\langle Qe_n, e_k \right\rangle_{L^2(D)}}{b^2 (\alpha_n - \alpha_k)^2 + 8a^2 b (\alpha_n + \alpha_k)} \left\langle e_j, e_n \right\rangle_{L^2(D)} \left\langle e_j, e_k \right\rangle_{L^2(D)} \notag \\
&= \frac{1}{4a} \left\langle Qe_j, e_j \right\rangle_{L^2(D)}. \notag
\end{align}

If we sum up these terms over $j$, we will obtain the second term on the right--hand side of \eqref{form of Q infinity splitted} of the trace, that is
$$
\frac{1}{4a} \sum_{j=1}^{\infty} \left\langle Qe_j, e_j \right\rangle_{L^2(D)} = \frac{1}{4a} \sum_{j=1}^{\infty} \lambda_j.
$$
\end{proof}

Based on above Lemma and Theorem \ref{ergodic theorem}, strongly consistent estimators of parameters $a$ and $b$ may be proposed now.
\begin{theorem} \label{strong consistency of a hat and b hat}
If we set
\begin{equation} \label{definition of IT}
I_T = \frac{1}{T} \int_0^T \| X^{x_0}(t) \|_V^2 \, dt,
\end{equation}
then the processes
\begin{align}
\hat{a}_T &= \frac{b+1}{4b I_T} \, \mathrm{Tr} \, Q, \label{a hat} \\
\hat{b}_T &= \frac{\mathrm{Tr} \, Q}{4a I_T - \mathrm{Tr} \, Q} \label{b hat}
\end{align}
are strongly constistent estimators of the parameters $a$ and $b$, respectively, that is $\hat{a}_T \rightarrow a$, $\hat{b}_T \rightarrow b$, $\mathbb P-a.s.$ as $T \rightarrow \infty$.
\end{theorem}
\begin{proof}
From \eqref{ergodic property} and \eqref{form of Q infinity result}, it follows that
$$
\lim_{T \rightarrow \infty} I_T = \frac{b+1}{4ab} \, \text{Tr} \, Q, \quad \mathbb P-a.s.
$$

Hence we obtain the desired limits $\hat{a}_T \rightarrow a$, $\hat{b}_T \rightarrow b$, $\mathbb P-a.s.$ as $T \rightarrow \infty$.
\end{proof}

\begin{remark}
The estimators $\hat{a}_T$ and $\hat{b}_T$ may be easily implemented, but they have one major disadvantage: We need to know the true value of the other parameter. In order to compute the estimator $\hat{a}_T$, we need to know not only the quantity $I_T$ (which can be computed from the observation of the trajectory of the process $(X^{x_0}(t), 0 \leq t \leq T)$), the trace of the operator $Q$ (which is supposed to be given by the model), but we also need to know the true value of the parameter $b$. (And similarly for the estimator $\hat{b}_T$.) Nevertheless, we believe these estimators may be useful in the situations, when we are estimating only one of the parameters and the other is known.
\end{remark}

However another family of estimators $(\tilde{a}_T, \tilde{b}_T)$ is proposed now, which does not possess this disadvantage. Since
$$
\| x \|_V^2 = \| x_1 \|_{\text{Dom}(-A)^{\frac{1}{2}}}^2 + \| x_2 \|_{L^2(D)}^2, \quad \forall x = (x_1, x_2)^{\top} \in V,
$$
the integral in \eqref{definition of IT} may be split into two parts
\begin{align}
I_T &= \frac{1}{T} \int_0^T \| X^{x_0}(t) \|_V^2 \, dt \notag \\
&= \frac{1}{T} \int_0^T \| X_1^{x_0}(t) \|_{\text{Dom}(-A)^{\frac{1}{2}}}^2 \, dt + \frac{1}{T} \int_0^T \| X_2^{x_0}(t) \|_{L^2(D)}^2 \, dt \notag \\
&=: Y_T + H_T, \label{definition of YT and HT}
\end{align}
where $X^{x_0}(t) = ( X_1^{x_0}(t), X_2^{x_0}(t))^{\top} \in V$ is the solution to the equation \eqref{linear equation with parameters}.

From the proof of Lemma \ref{trace of Q infinity} (and also from the formula \eqref{form of Q infinity splitted}), it is easy to see, that the $\text{Tr} \, Q_{\infty}^{(a,b)}$ may be also split into two parts. In the following Theorem, we show that these parts converge individually to their corresponding limits and based on this convergence, we may introduce new family of estimators $\tilde{a}_T$ and $\tilde{b}_T$.

\begin{theorem} \label{strong consistency of a tilde and b tilde}
The estimators
\begin{align}
\tilde{a}_T &= \frac{\mathrm{Tr} \, Q}{4 H_T}, \label {a tilde} \\
\tilde{b}_T &= \frac{H_T}{Y_T} \label{b tilde}
\end{align}
are strongly constistent estimators of the parameters $a$ and $b$, respectively.
\end{theorem}
\begin{proof}
Consider the functional $\varrho_1 : V \rightarrow \mathbb R$, $\varrho_1 (y) = \| y_1 \|_{\text{Dom}(-A)^{\frac{1}{2}}}^2$, $y = (y_1, y_2)^{\top} \in V$. Then all the conditions of Theorem \ref{ergodic theorem} are satisfied with $m=1$ and
\begin{align}
\lim_{T \rightarrow \infty} \frac{1}{T} \int_0^T \varrho_1 \left( X^{x_0}(t) \right) \, dt &= \lim_{T \rightarrow \infty} \frac{1}{T} \int_0^T \| X_1^{x_0}(t) \|_{\text{Dom}(-A)^{\frac{1}{2}}}^2 \, dt \notag \\
&= \int_{\text{Dom}(-A)^{\frac{1}{2}}} \| y_1 \|_{\text{Dom}(-A)^{\frac{1}{2}}}^2 \, \mu_{\infty, 1}^{(a,b)} (dy_1) \notag \\
&= \frac{1}{4ab} \, \text{Tr} \, Q, \notag 
\end{align}
where $\mu_{\infty, 1}^{(a,b)}$ is the Gaussian measure with zero mean and covariance operator
$$
\sum_{n=1}^{\infty} \sum_{k=1}^{\infty} \frac{4a \alpha_n \left\langle Qe_n, e_k \right\rangle_{L^2(D)}}{b^2 (\alpha_n - \alpha_k)^2 + 8a^2 b (\alpha_n + \alpha_k)} \left\langle x, e_n \right\rangle_{L^2(D)} e_k, \quad \forall x \in \text{Dom}((-A)^{\frac{1}{2}}),
$$
so $\mu_{\infty, 1}^{(a,b)}$ is "the first marginal" of the measure $\mu_{\infty}^{(a,b)}$.

Hence $Y_T \rightarrow \frac{1}{4ab} \, \text{Tr} \, Q$ for $T \rightarrow \infty$.

Similarly consider the functional $\varrho_2 : V \rightarrow \mathbb R$, $\varrho_2 (y) = \| y_2 \|_{L^2(D)}^2$, $y = (y_1, y_2)^{\top} \in V$. Then all the conditions of Theorem \ref{ergodic theorem} are satisfied with $m=1$ and
\begin{align}
\lim_{T \rightarrow \infty} \frac{1}{T} \int_0^T \varrho_2 \left( X^{x_0}(t) \right) \, dt &= \lim_{T \rightarrow \infty} \frac{1}{T} \int_0^T \| X_2^{x_0}(t) \|_{L^2(D)}^2 \, dt \notag \\
&= \int_{L^2(D)} \| y_2 \|_{L^2(D)}^2 \, \mu_{\infty, 2}^{(a,b)} (dy_2) \notag \\
&= \frac{1}{4a} \, \text{Tr} \, Q, \notag 
\end{align}
where $\mu_{\infty, 2}^{(a,b)}$ is the Gaussian measure with zero mean and covariance operator
$$
\sum_{n=1}^{\infty} \sum_{k=1}^{\infty} \frac{2ab (\alpha_n + \alpha_k) \left\langle Qe_n, e_k \right\rangle_{L^2(D)}}{b^2 (\alpha_n - \alpha_k)^2 + 8a^2 b (\alpha_n + \alpha_k)} \left\langle x, e_n \right\rangle_{L^2(D)} e_k, \quad \forall x \in L^2(D),
$$
so it is "the second marginal" of the measure $\mu_{\infty}^{(a,b)}$.

Hence $H_T \rightarrow \frac{1}{4a} \, \text{Tr} \, Q$ for $T \rightarrow \infty$ and the convergence of $\tilde{a}_T$ to the true value of parameter $a$ follows. Similarly
$$
\tilde{b}_T = \frac{H_T}{Y_T} \rightarrow \frac{\frac{\text{Tr} \, Q}{4a}}{\frac{\text{Tr} \, Q}{4ab}} = b, \quad T \rightarrow \infty, \quad \mathbb P-a.s.
$$
\end{proof}

\section{Asymptotic normality of the estimators} \label{section: AN}
\subsection{Asymptotic normality of the estimators $\hat{a}_T$, $\hat{b}_T$}
In this section, we show asymptotic normality of estimators \eqref{a hat} and \eqref{b hat}, that is the weak convergences of $\mathrm{Law} \left( \sqrt{T} \left( \hat{a}_T - a \right) \right)$ and $\mathrm{Law} \left(\sqrt{T} \left( \hat{b}_T - b \right) \right)$ to Gaussian distributions. To this aim, define an operator $R: V \rightarrow V$ by
$$
Rx = R \left( \begin{array}{r}
x_1\\
x_2
\end{array} \right) = \left( \begin{array}{cc}
bI - \frac{4a^2}{b+1} A^{-1} & - \frac{2a}{b+1} A^{-1}\\
\frac{2a}{b+1}I & I
\end{array} \right) \left( \begin{array}{r}
x_1\\
x_2
\end{array} \right), \quad \forall x = \left( \begin{array}{r}
x_1\\
x_2
\end{array} \right) \in V.
$$

The properties of $R$ needed in the sequel are summarized in the following Lemma.

\begin{lemma}\label{properties of R}
The operator $R$ is a self--adjoint linear isomorphism of $V$. Moreover,
\begin{equation}\label{RxAx}
\left\langle Rx, \mathcal A x \right\rangle_V = - \frac{2ab}{b+1} \| x \|_V^2, \quad \forall x = \left( \begin{array}{r}
x_1\\
x_2
\end{array} \right) \in \mathrm{Dom}(\mathcal A).
\end{equation}
\end{lemma}
\begin{proof}
It is evident that $R \in \mathcal L(V)$ and for $x = (x_1, x_2)^{\top} \in V$ and $y = (y_1, y_2)^{\top} \in V$ we have
\begin{align}
\left\langle Rx, y \right\rangle_V &= \left\langle \left( \begin{array}{c}
bx_1 - \frac{4a^2}{b+1} A^{-1} x_1 - \frac{2a}{b+1} A^{-1} x_2 \\
\frac{2a}{b+1} x_1 + x_2
\end{array} \right), \left( \begin{array}{r}
y_1\\
y_2
\end{array} \right) \right\rangle_V \notag \\
&= b \left\langle (-A)^{\frac{1}{2}} x_1, (-A)^{\frac{1}{2}} y_1 \right\rangle_{L^2(D)} - \frac{4a^2}{b+1} \left\langle (-A)^{\frac{1}{2}} A^{-1} x_1, (-A)^{\frac{1}{2}} y_1 \right\rangle_{L^2(D)} \notag \\
&\phantom{=}- \frac{2a}{b+1} \left\langle (-A)^{\frac{1}{2}} A^{-1} x_2, (-A)^{\frac{1}{2}} y_1 \right\rangle_{L^2(D)} + \frac{2a}{b+1} \left\langle x_1, y_2 \right\rangle_{L^2(D)} + \left\langle x_2, y_2 \right\rangle_{L^2(D)} \notag \\
&= \left\langle x, Ry \right\rangle_V, \notag
\end{align}
so $R = R^*$. The equation \eqref{RxAx} can be derived by similar computation. Indeed, for every $x = (x_1, x_2)^{\top} \in \text{Dom}(\mathcal A)$ we have
\begin{align}
\left\langle Rx, \mathcal A x \right\rangle_V &%= \left\langle \left( \begin{array}{cc}
%bI - \frac{4a^2}{b+1} A^{-1} & - \frac{2a}{b+1} A^{-1}\\
%\frac{2a}{b+1}I & I
%\end{array} \right) \left( \begin{array}{r}
%x_1\\
%x_2
%\end{array} \right), \left( \begin{array}{cc}
%0&I\\
%bA&-2aI
%\end{array} \right) \left( \begin{array}{r}
%x_1\\
%x_2
%\end{array} \right) \right\rangle_V \notag \\
= \left\langle \left( \begin{array}{c}
bx_1 - \frac{4a^2}{b+1} A^{-1} x_1 - \frac{2a}{b+1} A^{-1} x_2 \\
\frac{2a}{b+1} x_1 + x_2
\end{array} \right), \left( \begin{array}{c}
x_2\\
bAx_1 - 2ax_2
\end{array} \right) \right\rangle_V \notag \\
&= b \left\langle (-A)^{\frac{1}{2}} x_1, (-A)^{\frac{1}{2}} x_2 \right\rangle_{L^2(D)} - \frac{4a^2}{b+1} \left\langle (-A)^{\frac{1}{2}} A^{-1} x_1, (-A)^{\frac{1}{2}} x_2  \right\rangle_{L^2(D)} \notag \\
&\phantom{=}- \frac{2a}{b+1} \left\langle (-A)^{\frac{1}{2}} A^{-1} x_2, (-A)^{\frac{1}{2}} x_2 \right\rangle_{L^2(D)} + \frac{2ab}{b+1} \left\langle x_1, Ax_1 \right\rangle_{L^2(D)} \notag \\
&\phantom{=}- \frac{4a^2}{b+1} \left\langle x_1, x_2 \right\rangle_{L^2(D)} + b \left\langle x_2, Ax_1 \right\rangle_{L^2(D)} - 2a \left\langle x_2, x_2 \right\rangle_{L^2(D)} \notag \\
%&= - \frac{2ab}{b+1} \left\langle (-A)^{\frac{1}{2}} x_1, (-A)^{\frac{1}{2}} x_1 \right\rangle_{L^2(D)} + \left( \frac{2a}{b+1} - 2a \right) \left\langle x_2, x_2 \right\rangle_{L^2(D)} \notag \\
&= - \frac{2ab}{b+1} \left\langle (-A)^{\frac{1}{2}} x_1, (-A)^{\frac{1}{2}} x_1 \right\rangle_{L^2(D)} - \frac{2ab}{b+1} \left\langle x_2, x_2 \right\rangle_{L^2(D)} \notag \\
%&= - \frac{2ab}{b+1} \left( \| x_1 \|_{\text{Dom}(-A)^{\frac{1}{2}}}^2 + \| x_2 \|_{L^2(D)}^2 \right) \notag \\
&= - \frac{2ab}{b+1} \| x \|_V^2. \notag
\end{align}
\end{proof}

In the proof of Theorem \ref{asymptotic normality of a hat and b hat}, we will also need the alternative formula for the process $I_T$, which was defined by \eqref{definition of IT}.

\begin{lemma} \label{representation of IT}
The process $I_T$ admits the following representation
\begin{align}
I_T = \frac{1}{T} \int_0^T \| X^{x_0}(t) \|_V^2 \, dt &= - \frac{b+1}{4abT} \left( \left\langle RX^{x_0}(T), X^{x_0}(T) \right\rangle_V - \left\langle Rx_0, x_0 \right\rangle_V \right) \notag \\
&\phantom{=}+ \frac{b+1}{2abT} \int_0^T \left\langle RX^{x_0}(t), \Phi dB(t) \right\rangle_V + \frac{b+1}{4ab} \, \mathrm{Tr} \, Q. \label{IT via Ito}
\end{align}
\end{lemma}
\begin{proof}
Define the function $g: V \rightarrow \mathbb R$ by
\begin{equation}
g(x) = \left\langle Rx, x \right\rangle_V, \quad \forall x \in V.
\end{equation}

The It\^ o's formula (see e. g. \cite{dapratozabczyk}, Theorem 4.17.) is not applicable to the process $g(X^{x_0}(t))$ directly, because $(X^{x_0}(t), t \geq 0)$ is not a strong solution to the equation \eqref{linear equation with parameters}. We apply it to suitable finite--dimensional projections.

Let $\{ h_n, n \in \mathbb N \}$ be an orthonormal basis in $V$ and let $P_N$ be the operator of projection on the $\text{span} \, \{ h_n, n = 1, \ldots N \}$, that is
$$
P_N x = \sum_{n=1}^N \left\langle x, h_n \right\rangle_V h_n, \quad \forall x \in V, \quad \forall N \in \mathbb N.
$$

Fix $N \in \mathbb N$ and set
$$
X^{x_0, N}(t) := P_N X^{x_0}(t), \quad \forall t \geq 0.
$$

The expansion for the $X^{x_0, N}(t)$ is finite, so $X_1^{x_0, N}(t) \in \text{Dom}(A)$, $X_2^{x_0, N}(t) \in \text{Dom}((-A)^{\frac{1}{2}})$ and consequently $X^{x_0, N}(t) \in \text{Dom}(\mathcal A)$ for all $t \geq 0$. Now we apply It\^ o's formula to the function $g(X^{x_0, N}(t))$, which yields
\begin{equation} \label{ito}
dg(X^{x_0, N}(t)) = 2 \left\langle RX^{x_0, N}(t), dX^{x_0, N}(t) \right\rangle_V + \frac{1}{2} \, \text{Tr} \left( 2R \Phi \Phi^* \right) \, dt.
\end{equation}

The second term may be simplified via following calculation
$$
\frac{1}{2} \, \text{Tr} \left( 2R \Phi \Phi^* \right) = \text{Tr} \left( \begin{array}{cc}
0 & - \frac{2a}{b+1} A^{-1}Q \\
0 & Q
\end{array} \right) = \text{Tr} \, Q.
$$

Using that fact and Lemma \ref{properties of R}, the expression \eqref{ito} implies
\begin{align}
dg(X^{x_0, N}(t)) &= 2 \left\langle RX^{x_0, N}(t), \mathcal A X^{x_0, N}(t) \right\rangle_V \, dt + 2 \left\langle RX^{x_0, N}(t), \Phi dB(t) \right\rangle_V + \text{Tr} \, Q \, dt \notag \\
&= - \frac{4ab}{b+1} \left\| X^{x_0, N}(t) \right\|_V^2 \, dt + 2 \left\langle RX^{x_0, N}(t), \Phi dB(t) \right\rangle_V + \text{Tr} \, Q \, dt. \notag
\end{align}

After integrating previous formula over the interval $(0,T)$, we arrive at
\begin{align}
\frac{1}{T} \int_0^T \left\| X^{x_0, N}(t) \right\|_V^2 \, dt &= - \frac{b+1}{4abT} \left( \left\langle RX^{x_0, N}(T), X^{x_0, N}(T) \right\rangle_V - \left\langle Rx_0^N, x_0^N \right\rangle_V \right) \notag \\
&\phantom{=}+ \frac{b+1}{2abT} \int_0^T \left\langle RX^{x_0, N}(t), \Phi dB(t) \right\rangle_V + \frac{b+1}{4ab} \, \mathrm{Tr} \, Q. \label{IT N via Ito}
\end{align}

Since
$$
\left\| X^{x_0, N}(t) \right\|_V \leq \| X^{x_0}(t) \|_V, \quad \forall t \geq 0, \quad \forall N \in \mathbb N,
$$
the function $\| X^{x_0}(t) \|_V^2$ is an integrable majorant for the integral on the left--hand side. Also
$$
\int_0^T \left\langle RX^{x_0, N}(t), \Phi \, dB(t) \right\rangle_V \rightarrow \int_0^T \left\langle RX^{x_0}(t), \Phi \, dB(t) \right\rangle_V, \quad N \rightarrow \infty \text{ in } L^2(\Omega),
$$
because
$$
\mathbb E \left| \int_0^T \left\langle R \left( X^{x_0, N}(t) - X^{x_0}(t) \right), \Phi \, dB(t) \right\rangle_V \right|^2 \leq C \int_0^T \mathbb E \left\| X^{x_0, N}(t) - X^{x_0}(t) \right\|_V^2 \, dt,
$$
which tends to $0$ as $N \rightarrow \infty$, since
$$
X^{x_0, N}(t) \rightarrow X^{x_0}(t), \quad \forall t \geq 0, \quad N \rightarrow \infty \text{ in } L^2(\Omega).
$$

Hence we obtain \eqref{IT via Ito} by passing $N$ to infinity in \eqref{IT N via Ito}.
\end{proof}

We will also need the following Lemma.

\begin{lemma} \label{convergence to zero}
Let $(X^{x_0}(t), t \geq 0)$ be a solution to the linear equation \eqref{linear equation with parameters} and $R \in \mathcal L(V)$. Then
$$
\frac{1}{\sqrt{t}} \left\langle RX^{x_0}(t), X^{x_0}(t) \right\rangle_V \rightarrow 0
$$
in $L^1(\Omega)$ as $t \rightarrow \infty$.
\end{lemma}
\begin{proof}
\begin{align}
\mathbb E \left| \frac{\left\langle RX^{x_0}(t), X^{x_0}(t) \right\rangle_V}{\sqrt{t}} \right| &\leq \frac{C}{\sqrt{t}} \, \mathbb E \| X^{x_0}(t) \|_V^2 \notag \\
&\leq \frac{2C}{\sqrt{t}} \, \mathbb E \| S(t) x_0 \|_V^2 + \frac{2C}{\sqrt{t}} \, \mathbb E \| Z(t) \|_V^2 \notag \\
&\leq \frac{C_1}{\sqrt{t}} e^{-2 \rho t} \, \mathbb E \| x_0 \|_V^2 + \frac{2C}{\sqrt{t}} \, \text{Tr} \, Q_t. \notag
\end{align}

Since
$$
\sup_{t \geq 0} \, \text{Tr} \, Q_t < \infty,
$$
(which is equivalent to the existence of an invariant measure, see \cite{dapratozabczyk}, Theorem 11.7.), both terms tend to $0$ as $t$ tends to infinity.

\end{proof}

Finally, define the operator $\tilde{R}: V \rightarrow L^2(D)$ by
\begin{equation}
\tilde{R} x = \left( \begin{array}{cc} \frac{2a}{b+1}I & I \end{array} \right) \left( \begin{array}{r} x_1 \\ x_2 \end{array} \right) = \frac{2a}{b+1} x_1 + x_2, \quad \forall x = \left( \begin{array}{r} x_1 \\ x_2 \end{array} \right) \in V.
\end{equation}

Note, that the adjoint operator of $\tilde{R}$ has the following form
\begin{equation}
\tilde{R}^* : L^2(D) \rightarrow V, \, \tilde{R}^* x = \left( \begin{array}{c} - \frac{2a}{b+1} A^{-1} \\ I \end{array} \right) x = \left( \begin{array}{c} - \frac{2a}{b+1} A^{-1} x \\ x \end{array} \right), \quad \forall x \in L^2(D).
\end{equation}

Asymptotic normality of the estimators $\hat{a}_T$ and $\hat{b}_T$ is formulated in the following Theorem.

\begin{theorem} \label{asymptotic normality of a hat and b hat}
1) The estimator $\hat{a}_T$ is asymptotically normal, that is $\mathrm{Law} \left( \sqrt{T} \left( \hat{a}_T - a \right) \right)$ converges weakly to the centered Gaussian distribution with variance \\ $\frac{4a^2}{(\mathrm{Tr} \, Q)^2} \, \mathrm{Tr} \left( Q \tilde{R} Q_{\infty}^{(a,b)} \tilde{R}^* \right)$, that is
\begin{equation}
\mathrm{Law} \left( \sqrt{T} \left( \hat{a}_T - a \right) \right) \stackrel{w^*}{\longrightarrow} N \left( 0, \, \frac{4a^2}{(\mathrm{Tr} \, Q)^2} \, \mathrm{Tr} \left( Q \tilde{R} Q_{\infty}^{(a,b)} \tilde{R}^* \right) \right), \quad T \rightarrow \infty.
\end{equation}
2) The estimator $\hat{b}_T$ is asymptotically normal, that is
\begin{equation}
\mathrm{Law} \left(\sqrt{T} \left( \hat{b}_T - b \right) \right) \stackrel{w^*}{\longrightarrow} N \left( 0, \, \frac{4b^2(b+1)^2}{(\mathrm{Tr} \, Q)^2} \, \mathrm{Tr} \left( Q \tilde{R} Q_{\infty}^{(a,b)} \tilde{R}^* \right) \right), \quad T \rightarrow \infty.
\end{equation}
\end{theorem}
\begin{proof}
Using formula \eqref{a hat} for the estimator $\hat{a}_T$ and Lemma \ref{representation of IT} for the representation of $I_T$, it is possible to compute the following
\begin{align}
\sqrt{T} \left( \hat{a}_T - a \right) &= \sqrt{T} \left( \frac{b+1}{4bI_T} \, \text{Tr} \, Q - a \right) = \sqrt{T} \frac{(b+1) \, \text{Tr} \, Q - 4abI_T}{4bI_T} \notag \\
&= \frac{\sqrt{T}}{4bI_T} \left( \frac{b+1}{T} \left( \left\langle RX^{x_0}(T), X^{x_0}(T) \right\rangle_V - \left\langle Rx_0, x_0 \right\rangle_V \right) \right. \notag \\
&\phantom{=} \left. - \frac{2(b+1)}{T} \int_0^T \left\langle RX^{x_0}(t), \Phi dB(t) \right\rangle_V \right) \notag \\
&= \frac{b+1}{4bI_T} \frac{1}{\sqrt{T}} \left( \left\langle RX^{x_0}(T), X^{x_0}(T) \right\rangle_V - \left\langle Rx_0, x_0 \right\rangle_V \right) \notag \\
&\phantom{=} - \frac{b+1}{2bI_T} \frac{1}{\sqrt{T}} \int_0^T \left\langle RX^{x_0}(t), \Phi dB(t) \right\rangle_V. \label{asymptotics for a}
\end{align}

The first term $\frac{1}{\sqrt{T}} \left( \left\langle RX^{x_0}(T), X^{x_0}(T) \right\rangle_V - \left\langle Rx_0, x_0 \right\rangle_V \right) \rightarrow 0$ in probability as $T \rightarrow \infty$ by Lemma \ref{convergence to zero}. Also define
\begin{align}
q(T) &= \frac{1}{\sqrt{T}} \int_0^T \left\langle RX^{x_0}(t), \Phi dB(t) \right\rangle_V \notag \\
&= \frac{1}{\sqrt{T}} \int_0^T \sum_{n=1}^{\infty} \sqrt{\lambda_n} \left\langle \tilde{R} X^{x_0}(t), e^{\prime}_n \right\rangle_{L^2(D)} \, d\beta_n (t), \notag
\end{align}
where we have used the representation of $V$--valued Brownian motion $B(t)$.

By the central limit theorem for martingales (see e. g. \cite{kutoyants}, Proposition 1.22.), $\mathrm{Law} \left( q(T) \right)$ converges weakly to the Gaussian distribution with a zero mean and variance given by
\begin{align}
\lim_{T \rightarrow \infty} \frac{1}{T} \int_0^T \sum_{n=1}^{\infty} \lambda_n \left\langle \tilde{R} X^{x_0}(t), e^{\prime}_n \right\rangle_{L^2(D)}^2 \, dt &= \lim_{T \rightarrow \infty} \frac{1}{T} \int_0^T \sum_{n=1}^{\infty} \left\langle Q^{\frac{1}{2}} \tilde{R} X^{x_0}(t), e^{\prime}_n \right\rangle_{L^2(D)}^2 \, dt \notag \\
&= \lim_{T \rightarrow \infty} \frac{1}{T} \int_0^T \| Q^{\frac{1}{2}} \tilde{R} X^{x_0}(t) \|_{L^2(D)}^2 \, dt \notag \\
&= \mathbb E \| Q^{\frac{1}{2}} \tilde{R} X(\infty) \|_{L^2(D)}^2 \notag \\
&= \text{Tr} \left( Q \tilde{R} Q_{\infty}^{(a,b)} \tilde{R}^* \right), \notag
\end{align}
where $X(\infty)$ is a $V$--valued Gaussian random variable with zero mean and covariance operator $Q_{\infty}^{(a,b)}$ (that is $\text{Law} \left(X(\infty) \right) = \mu_{\infty}^{(a,b)}$).

Since the multiplicative factor $- \frac{b+1}{2bI_T}$ of $q(T)$ in \eqref{asymptotics for a} converges to $- \frac{2a}{\text{Tr} \, Q}$ as $T \rightarrow \infty$, we arrive at
\begin{align}
\text{Law} (q(T)) &\stackrel{w^*}{\longrightarrow} N \left( 0, \, \text{Tr} \left( Q \tilde{R} Q_{\infty}^{(a,b)} \tilde{R}^* \right) \right), \quad T \rightarrow \infty, \label{asymptotics for qT} \\
\text{Law} \left( \sqrt{T} \left( \hat{a}_T - a \right) \right) &\stackrel{w^*}{\longrightarrow} N \left( 0, \, \frac{4a^2}{(\mathrm{Tr} \, Q)^2} \, \mathrm{Tr} \left( Q \tilde{R} Q_{\infty}^{(a,b)} \tilde{R}^* \right) \right), \quad T \rightarrow \infty. \notag
\end{align}

In a similar fashion, using formula \eqref{b hat} for the estimator $\hat{b}_T$ and Lemma \ref{representation of IT}, it is possible to compute the following
\begin{align}
\sqrt{T} \left( \hat{b}_T - b \right) &= \sqrt{T} \left( \frac{\text{Tr} \, Q}{4a I_T - \text{Tr} \, Q} - b \right) = \frac{\sqrt{T}}{4a I_T - \text{Tr} \, Q} \left( \text{Tr} \, Q - 4abI_T + b \, \text{Tr} \, Q \right) \notag \\
&= \frac{\sqrt{T}}{4a I_T - \text{Tr} \, Q} \left( \frac{b+1}{T} \left( \left\langle RX^{x_0}(T), X^{x_0}(T) \right\rangle_V - \left\langle Rx_0, x_0 \right\rangle_V \right) \right. \notag \\
&\phantom{=} \left. - \frac{2(b+1)}{T} \int_0^T \left\langle RX^{x_0}(t), \Phi dB(t) \right\rangle_V \right) \notag \\
&= \frac{b+1}{4a I_T - \text{Tr} \, Q} \frac{1}{\sqrt{T}} \left( \left\langle RX^{x_0}(T), X^{x_0}(T) \right\rangle_V - \left\langle Rx_0, x_0 \right\rangle_V \right) \notag \\
&\phantom{=}- \frac{2(b+1)}{4a I_T - \text{Tr} \, Q} \frac{1}{\sqrt{T}} \int_0^T \left\langle RX^{x_0}(t), \Phi dB(t) \right\rangle_V. \label{asypmtotics for b}
\end{align}

As above, the term $\frac{1}{\sqrt{T}} \left( \left\langle RX^{x_0}(T), X^{x_0}(T) \right\rangle_V - \left\langle Rx_0, x_0 \right\rangle_V \right) \rightarrow 0$ in pro\-ba\-bi\-li\-ty as $T \rightarrow \infty$ and the multiplicative factor $- \frac{2(b+1)}{4a I_T - \text{Tr} \, Q}$ of $q(T)$ in \eqref{asypmtotics for b} converges to $- \frac{2b(b+1)}{\text{Tr} \, Q}$ as $T \rightarrow \infty$. Hence we obtain the result
$$
\text{Law} \left( \sqrt{T} \left( \hat{b}_T - b \right) \right) \stackrel{w^*}{\longrightarrow} N \left( 0, \, \frac{4b^2(b+1)^2}{(\mathrm{Tr} \, Q)^2} \, \mathrm{Tr} \left( Q \tilde{R} Q_{\infty}^{(a,b)} \tilde{R}^* \right) \right), \quad T \rightarrow \infty.
$$
\end{proof}

\begin{remark} \label{form of limiting variance}
We specify the variance of the limiting Gaussian distribution in \eqref{asymptotics for qT}. By Theorem \ref{form of Q infinity - theorem}, we obtain
\begin{equation} \label{trace QRQR}
\mathrm{Tr} \left( Q \tilde{R} Q_{\infty}^{(a,b)} \tilde{R}^* \right) = \sum_{n=1}^{\infty} \sum_{k=1}^{\infty} \frac{16a^3 + 2ab(b+1)^2(\alpha_n + \alpha_k)}{b^2 (\alpha_n - \alpha_k)^2 + 8a^2 b(\alpha_n + \alpha_k)} \frac{1}{(b+1)^2} \left\langle Qe_n, e_k \right\rangle_{L^2(D)}^2.
\end{equation}
\end{remark}
%\begin{proof}
%If we start with the definition of the trace
%$$
%\text{Tr} \left( Q \tilde{R} Q_{\infty}^{(a,b)} \tilde{R}^* \right) = \sum_{j=1}^{\infty} \left\langle Q \tilde{R} Q_{\infty}^{(a,b)} \tilde{R}^* e_j, e_j \right\rangle_{L^2(D)}
%$$
%and use the definitions of $\tilde{R}$, $\tilde{R}^*$ and formula for $Q_{\infty}^{(a,b)}$ from Theorem \ref{form of Q infinity - theorem}, then the computation is rather straightforward, although quite long. After executing $Q \tilde{R} Q_{\infty}^{(a,b)} \tilde{R}^* e_j$ and simplifying, it is possible to verify that \eqref{trace QRQR} holds true.
%\end{proof}

\subsection{Asymptotic normality of the estimators $\tilde{a}_T$, $\tilde{b}_T$}
The family of estimators $\tilde{a}_T$, $\tilde{b}_T$ is also asymptotically normal, which will be shown in Theorem \ref{asymptotic normality of a tilde and b tilde}. The proof uses the same method as proof of Theorem \ref{asymptotic normality of a hat and b hat}, so the setup and auxiliary Lemmas will be very similar to those in previous subsection.

We start with the definition of operators $R_1: V \rightarrow V$ and $R_2: V \rightarrow V$:
$$
R_1 x = R_1 \left( \begin{array}{r}
x_1\\
x_2
\end{array} \right) = \left( \begin{array}{cc}
bI & 0 \\
0 & I
\end{array} \right) \left( \begin{array}{r}
x_1\\
x_2
\end{array} \right), \quad \forall x = \left( \begin{array}{r}
x_1\\
x_2
\end{array} \right) \in V,
$$
$$
R_2 x = R_2 \left( \begin{array}{r}
x_1\\
x_2
\end{array} \right) = \left( \begin{array}{cc}
bI - 4a^2 A^{-1} & - 2a A^{-1}\\
2a I & I
\end{array} \right) \left( \begin{array}{r}
x_1\\
x_2
\end{array} \right), \quad \forall x = \left( \begin{array}{r}
x_1\\
x_2
\end{array} \right) \in V.
$$

The properties of these two operators are summarized in the following Lemma.

\begin{lemma}\label{properties of R1 and R2}
The operators $R_1$ and $R_2$ are self--adjoint linear isomorphisms of $V$. Moreover,
\begin{align}
\left\langle R_1 x, \mathcal A x \right\rangle_V &= - 2a \| x_2 \|_{L^2(D)}^2, \quad \forall x = \left( \begin{array}{r}
x_1\\
x_2
\end{array} \right) \in \mathrm{Dom}(\mathcal A), \label{R1xAx} \\
\left\langle R_2 x, \mathcal A x \right\rangle_V &= - 2ab \| x_1 \|_{\mathrm{Dom}(-A)^{\frac{1}{2}}}^2, \quad \forall x = \left( \begin{array}{r}
x_1\\
x_2
\end{array} \right) \in \mathrm{Dom}(\mathcal A). \label{R2xAx}
\end{align}
\end{lemma}
\begin{proof}
It is evident that $R_1, R_2 \in \mathcal L(V)$ and for $x = (x_1, x_2)^{\top} \in V$ and $y = (y_1, y_2)^{\top} \in V$ we have
\begin{align}
\left\langle R_1 x, y \right\rangle_V &= \left\langle \left( \begin{array}{c}
bx_1 \\
x_2
\end{array} \right), \left( \begin{array}{r}
y_1\\
y_2
\end{array} \right) \right\rangle_V = b \left\langle (-A)^{\frac{1}{2}} x_1, (-A)^{\frac{1}{2}} y_1 \right\rangle_{L^2(D)} + \left\langle x_2, y_2 \right\rangle_{L^2(D)} \notag \\
&= \left\langle x, R_1 y \right\rangle_V, \notag
\end{align}
and
\begin{align}
\left\langle R_2 x, y \right\rangle_V &= \left\langle \left( \begin{array}{c}
bx_1 - 4a^2 A^{-1} x_1 - 2a A^{-1} x_2 \\
2a x_1 + x_2
\end{array} \right), \left( \begin{array}{r}
y_1\\
y_2
\end{array} \right) \right\rangle_V \notag \\
&= b \left\langle (-A)^{\frac{1}{2}} x_1, (-A)^{\frac{1}{2}} y_1 \right\rangle_{L^2(D)} - 4a^2 \left\langle (-A)^{\frac{1}{2}} A^{-1} x_1, (-A)^{\frac{1}{2}} y_1 \right\rangle_{L^2(D)} \notag \\
&\phantom{=}- 2a \left\langle (-A)^{\frac{1}{2}} A^{-1} x_2, (-A)^{\frac{1}{2}} y_1 \right\rangle_{L^2(D)} + 2a \left\langle x_1, y_2 \right\rangle_{L^2(D)} + \left\langle x_2, y_2 \right\rangle_{L^2(D)} \notag \\
&= \left\langle x, R_2 y \right\rangle_V, \notag
\end{align}
so $R_1 = R_1^*$ and $R_2 = R_2^*$. The equation \eqref{R1xAx} can be derived by simple computation. For every $x = (x_1, x_2)^{\top} \in \text{Dom}(\mathcal A)$ we have
\begin{align}
\left\langle R_1 x, \mathcal A x \right\rangle_V &%= \left\langle \left( \begin{array}{cc}
%bI & 0 \\
%0 & I
%\end{array} \right) \left( \begin{array}{r}
%x_1\\
%x_2
%\end{array} \right), \left( \begin{array}{cc}
%0&I\\
%bA&-2aI
%\end{array} \right) \left( \begin{array}{r}
%x_1\\
%x_2
%\end{array} \right) \right\rangle_V \notag \\
= \left\langle \left( \begin{array}{c}
bx_1 \\
x_2
\end{array} \right), \left( \begin{array}{c}
x_2\\
bAx_1 - 2ax_2
\end{array} \right) \right\rangle_V \notag \\
&= b \left\langle (-A)^{\frac{1}{2}} x_1, (-A)^{\frac{1}{2}} x_2 \right\rangle_{L^2(D)} + b \left\langle x_2, A x_1 \right\rangle_{L^2(D)} - 2a \left\langle x_2, x_2 \right\rangle_{L^2(D)} \notag \\
&= -2a \| x_2 \|_{L^2(D)}^2. \notag
\end{align}

Similar computation yields \eqref{R2xAx}:
\begin{align}
\left\langle R_2 x, \mathcal A x \right\rangle_V &%= \left\langle \left( \begin{array}{cc}
%bI - 4a^2 A^{-1} & - 2a A^{-1} \\
%2a I & I
%\end{array} \right) \left( \begin{array}{r}
%x_1\\
%x_2
%\end{array} \right), \left( \begin{array}{cc}
%0&I\\
%bA&-2aI
%\end{array} \right) \left( \begin{array}{r}
%x_1\\
%x_2
%\end{array} \right) \right\rangle_V \notag \\
= \left\langle \left( \begin{array}{c}
bx_1 - 4a^2 A^{-1} x_1 - 2a A^{-1} x_2\\
2a x_1 + x_2
\end{array} \right), \left( \begin{array}{c}
x_2\\
bAx_1 - 2ax_2
\end{array} \right) \right\rangle_V \notag \\
&= b \left\langle (-A)^{\frac{1}{2}} x_1, (-A)^{\frac{1}{2}} x_2 \right\rangle_{L^2(D)} - 4a^2 \left\langle (-A)^{\frac{1}{2}} A^{-1} x_1, (-A)^{\frac{1}{2}} x_2  \right\rangle_{L^2(D)} \notag \\
&\phantom{=}- 2a \left\langle (-A)^{\frac{1}{2}} A^{-1} x_2, (-A)^{\frac{1}{2}} x_2 \right\rangle_{L^2(D)} + 2ab \left\langle x_1, Ax_1 \right\rangle_{L^2(D)} \notag \\
&\phantom{=}- 4a^2 \left\langle x_1, x_2 \right\rangle_{L^2(D)} + b \left\langle x_2, Ax_1 \right\rangle_{L^2(D)} - 2a \left\langle x_2, x_2 \right\rangle_{L^2(D)} \notag \\
&= -2ab \| x_1 \|_{\text{Dom}(-A)^{\frac{1}{2}}}^2. \notag
\end{align}
\end{proof}

We will also need the alternative formulae for processes $Y_T$ and $H_T$, which were defined by \eqref{definition of YT and HT}.

\begin{lemma} \label{representation of YT and HT}
The process $Y_T$ admits the following representation
\begin{align}
Y_T &= \frac{1}{T} \int_0^T \| X_1^{x_0}(t) \|_{\mathrm{Dom}(-A)^{\frac{1}{2}}}^2 \, dt \notag \\
&= - \frac{1}{4abT} \left( \left\langle R_2 X^{x_0}(T), X^{x_0}(T) \right\rangle_V - \left\langle R_2 x_0, x_0 \right\rangle_V \right) \notag \\
&\phantom{=}+ \frac{1}{2abT} \int_0^T \left\langle R_2 X^{x_0}(t), \Phi dB(t) \right\rangle_V + \frac{1}{4ab} \, \mathrm{Tr} \, Q. \label{YT via Ito}
\end{align}

The process $H_T$ admits the following representation
\begin{align}
H_T &= \frac{1}{T} \int_0^T \| X_2^{x_0}(t) \|_{L^2(D)}^2 \, dt \notag \\
&= - \frac{1}{4aT} \left( \left\langle R_1 X^{x_0}(T), X^{x_0}(T) \right\rangle_V - \left\langle R_1 x_0, x_0 \right\rangle_V \right) \notag \\
&\phantom{=}+ \frac{1}{2aT} \int_0^T \left\langle R_1 X^{x_0}(t), \Phi dB(t) \right\rangle_V + \frac{1}{4a} \, \mathrm{Tr} \, Q. \label{HT via Ito}
\end{align}
\end{lemma}
\begin{proof}
Define the function $g_1: V \rightarrow \mathbb R$ by
\begin{equation}
g_1(x) = \left\langle R_1 x, x \right\rangle_V, \quad \forall x \in V.
\end{equation}

The application of It\^ o's formula to the function $g_1(X^{x_0, N}(t))$ (we also have to use suitable projections, see proof of Lemma \ref{representation of IT}), yields
\begin{equation} \label{ito g1}
dg_1(X^{x_0, N}(t)) = 2 \left\langle R_1 X^{x_0, N}(t), dX^{x_0, N}(t) \right\rangle_V + \frac{1}{2} \, \text{Tr} \left( 2R_1 \Phi \Phi^* \right) \, dt.
\end{equation}

Since the second term equals to
$$
\frac{1}{2} \text{ Tr} \left( 2R_1 \Phi \Phi^* \right) = \text{Tr} \left( \begin{array}{cc}
0 & 0 \\
0 & Q
\end{array} \right) = \text{Tr } Q,
$$
the expression \eqref{ito g1} and Lemma \ref{properties of R1 and R2} imply
\begin{align}
dg_1(X^{x_0, N}(t)) &= 2 \left\langle R_1 X^{x_0, N}(t), \mathcal A X^{x_0, N}(t) \right\rangle_V \, dt + 2 \left\langle R_1 X^{x_0, N}(t), \Phi dB(t) \right\rangle_V + \text{Tr} \, Q \, dt \notag \\
&= - 4a \left\| X_2^{x_0, N}(t) \right\|_{L^2(D)}^2 \, dt + 2 \left\langle R_1 X^{x_0, N}(t), \Phi dB(t) \right\rangle_V + \text{Tr} \, Q \, dt. \notag
\end{align}

After integrating previous formula over the interval $(0,T)$ and passing $N$ to infinity, we will arrive at \eqref{HT via Ito}.

Similarly, if we define the function $g_2: V \rightarrow \mathbb R$ by
\begin{equation}
g_2(x) = \left\langle R_2 x, x \right\rangle_V, \quad \forall x \in V,
\end{equation}
and apply It\^ o's formula to the function $g_2(X^{x_0, N}(t))$, we will obtain
\begin{equation} \label{ito g2}
dg_2(X^{x_0, N}(t)) = 2 \left\langle R_2 X^{x_0, N}(t), dX^{x_0, N}(t) \right\rangle_V + \frac{1}{2} \, \text{Tr} \left( 2R_2 \Phi \Phi^* \right) \, dt.
\end{equation}

Since the second term equals to
$$
\frac{1}{2} \text{ Tr} \left( 2R_2 \Phi \Phi^* \right) = \text{Tr} \left( \begin{array}{cc}
0 & -2a A^{-1} Q \\
0 & Q
\end{array} \right) = \text{Tr } Q,
$$
the expression \eqref{ito g2} and Lemma \ref{properties of R1 and R2} imply
\begin{align}
dg_2(X^{x_0, N}(t)) &= 2 \left\langle R_2 X^{x_0, N}(t), \mathcal A X^{x_0, N}(t) \right\rangle_V \, dt + 2 \left\langle R_2 X^{x_0, N}(t), \Phi dB(t) \right\rangle_V + \text{Tr} \, Q \, dt \notag \\
&= - 4ab \left\| X_1^{x_0, N}(t) \right\|_{\text{Dom}(-A)^{\frac{1}{2}}}^2 \, dt + 2 \left\langle R_2 X^{x_0, N}(t), \Phi dB(t) \right\rangle_V + \text{Tr} \, Q \, dt. \notag
\end{align}

After integrating previous formula over the interval $(0,T)$ and passing $N$ to infinity, we will arrive at \eqref{YT via Ito}.
\end{proof}

Also define the operator $\tilde{R}_1: V \rightarrow L^2(D)$ by
\begin{equation}
\tilde{R}_1 x = \left( \begin{array}{cc} 0 & I \end{array} \right) \left( \begin{array}{r} x_1 \\ x_2 \end{array} \right) = x_2, \quad \forall x = \left( \begin{array}{r} x_1 \\ x_2 \end{array} \right) \in V,
\end{equation}

%with its adjoint $\tilde{R}_1^*$
%\begin{equation}
%\tilde{R}_1^* : L^2(D) \rightarrow V, \, \tilde{R}^* x = \left( \begin{array}{c} 0 \\ I \end{array} \right) x = \left( \begin{array}{c} 0 \\ x \end{array} \right), \quad \forall x \in L^2(D),
%\end{equation}

and the operator $\tilde{R}_2: V \rightarrow L^2(D)$ by
\begin{equation}
\tilde{R}_2 x = \left( \begin{array}{cc} 2a I & 0 \end{array} \right) \left( \begin{array}{r} x_1 \\ x_2 \end{array} \right) = 2a x_1, \quad \forall x = \left( \begin{array}{r} x_1 \\ x_2 \end{array} \right) \in V.
\end{equation}

Note that
\begin{equation}
\tilde{R}_2^* : L^2(D) \rightarrow V, \, \tilde{R}^* x = \left( \begin{array}{c} -2a A^{-1} \\ 0 \end{array} \right) x = \left( \begin{array}{c} -2a A^{-1} x \\ 0 \end{array} \right), \quad \forall x \in L^2(D).
\end{equation}

Asymptotic normality of the estimators $\tilde{a}_T$ and $\tilde{b}_T$ is formulated in the following Theorem.

\begin{theorem} \label{asymptotic normality of a tilde and b tilde}
1) The estimator $\tilde{a}_T$ is asymptotically normal, that is
\begin{equation}
\mathrm{Law} \left (\sqrt{T} \left( \tilde{a}_T - a \right) \right) \stackrel{w^*}{\longrightarrow} N \left( 0, \, \frac{4a^2}{(\mathrm{Tr} \, Q)^2} \, \mathrm{Tr} \left( Q \tilde{R}_1 Q_{\infty}^{(a,b)} \tilde{R}_1^* \right) \right), \quad T \rightarrow \infty.
\end{equation}
2) The estimator $\tilde{b}_T$ is asymptotically normal, that is
\begin{equation}
\mathrm{Law} \left(\sqrt{T} \left( \tilde{b}_T - b \right) \right) \stackrel{w^*}{\longrightarrow} N \left( 0, \, \frac{4b^2}{(\mathrm{Tr} \, Q)^2} \, \mathrm{Tr} \left( Q \tilde{R}_2 Q_{\infty}^{(a,b)} \tilde{R}_2^* \right) \right), \quad T \rightarrow \infty.
\end{equation}
\end{theorem}
\begin{proof}
If we use formula \eqref{a tilde} for the estimator $\tilde{a}_T$ and representation \eqref{HT via Ito} for $H_T$ from Lemma \ref{representation of YT and HT}, we obtain
\begin{align}
\sqrt{T} \left( \tilde{a}_T - a \right) &= \sqrt{T} \left( \frac{1}{4H_T} \, \text{Tr} \, Q - a \right) = \frac{\sqrt{T}}{4H_T} \left( \text{Tr} \, Q - 4aH_T \right) \notag \\
&= \frac{\sqrt{T}}{4H_T} \left( \frac{1}{T} \left( \left\langle R_1 X^{x_0}(T), X^{x_0}(T) \right\rangle_V - \left\langle R_1 x_0, x_0 \right\rangle_V \right) \right. \notag \\
&\phantom{=} \left. - \frac{2}{T} \int_0^T \left\langle R_1 X^{x_0}(t), \Phi dB(t) \right\rangle_V \right) \notag \\
&= \frac{1}{4H_T} \frac{1}{\sqrt{T}} \left( \left\langle R_1 X^{x_0}(T), X^{x_0}(T) \right\rangle_V - \left\langle R_1 x_0, x_0 \right\rangle_V \right) \notag \\
&\phantom{=} - \frac{1}{2H_T} \frac{1}{\sqrt{T}} \int_0^T \left\langle R_1 X^{x_0}(t), \Phi dB(t) \right\rangle_V. \label{asymptotics for a tilde}
\end{align}

The first term $\frac{1}{\sqrt{T}} \left( \left\langle R_1 X^{x_0}(T), X^{x_0}(T) \right\rangle_V - \left\langle R_1 x_0, x_0 \right\rangle_V \right) \rightarrow 0$ in probability as $T \rightarrow \infty$ by Lemma \ref{convergence to zero}. Also define	
\begin{align}
q_1(T) &= \frac{1}{\sqrt{T}} \int_0^T \left\langle R_1 X^{x_0}(t), \Phi dB(t) \right\rangle_V \notag \\
&= \frac{1}{\sqrt{T}} \int_0^T \sum_{n=1}^{\infty} \sqrt{\lambda_n} \left\langle \tilde{R}_1 X^{x_0}(t), e^{\prime}_n \right\rangle_{L^2(D)} \, d\beta_n (t), \notag
\end{align}
where we have used the representation of $V$--valued Brownian motion $B(t)$.

By the central limit theorem for martingales, $\mathrm{Law} \left( q_1(T) \right)$ to the Gaussian distribution with a zero mean and variance given by
\begin{align}
\lim_{T \rightarrow \infty} \frac{1}{T} \int_0^T \sum_{n=1}^{\infty} \lambda_n \left\langle \tilde{R}_1 X^{x_0}(t), e^{\prime}_n \right\rangle_{L^2(D)}^2 \, dt %&= \lim_{T \rightarrow \infty} \frac{1}{T} \int_0^T \sum_{n=1}^{\infty} \left\langle Q^{\frac{1}{2}} \tilde{R}_1 X^{x_0}(t), e^{\prime}_n \right\rangle_{L^2(D)}^2 \, dt \notag \\
&= \lim_{T \rightarrow \infty} \frac{1}{T} \int_0^T \| Q^{\frac{1}{2}} \tilde{R}_1 X^{x_0}(t) \|_{L^2(D)}^2 \, dt \notag \\
%&= \mathbb E \| Q^{\frac{1}{2}} \tilde{R}_1 X(\infty) \|_{L^2(D)}^2 \notag \\
&= \text{Tr} \left( Q \tilde{R}_1 Q_{\infty}^{(a,b)} \tilde{R}_1^* \right). \notag
\end{align}
%where $X(\infty)$ is a $V$--valued Gaussian random variable with zero mean and covariance operator $Q_{\infty}^{(a,b)}$ (i. e. $\text{Law} \left(X(\infty) \right) = \mu_{\infty}^{(a,b)}$).

Since the multiplicative factor $- \frac{1}{2H_T}$ of $q_1(T)$ in \eqref{asymptotics for a tilde} converges to $- \frac{2a}{\text{Tr} \, Q}$ as $T \rightarrow \infty$, we arrive at
\begin{align}
\text{Law} (q_1(T)) &\stackrel{w^*}{\longrightarrow} N \left( 0, \, \text{Tr} \left( Q \tilde{R}_1 Q_{\infty}^{(a,b)} \tilde{R}_1^* \right) \right), \quad T \rightarrow \infty, \\
\text{Law} \left(\sqrt{T} \left( \tilde{a}_T - a \right) \right) &\stackrel{w^*}{\longrightarrow} N \left( 0, \, \frac{4a^2}{(\mathrm{Tr} \, Q)^2} \, \mathrm{Tr} \left( Q \tilde{R}_1 Q_{\infty}^{(a,b)} \tilde{R}_1^* \right) \right), \quad T \rightarrow \infty. \notag
\end{align}

Similarly, using formula \eqref{b tilde} for the estimator $\tilde{b}_T$ and Lemma \ref{representation of YT and HT} for representation of $Y_T$ and $H_T$, we may compute the following
\begin{align}
\sqrt{T} \left( \tilde{b}_T - b \right) &= \sqrt{T} \left( \frac{H_T}{Y_T} - b \right) = \frac{\sqrt{T}}{Y_T} \left( H_T - b Y_T \right) \notag \\
&= \frac{\sqrt{T}}{Y_T} \left( - \frac{1}{4aT} \left( \left\langle R_1 X^{x_0}(T), X^{x_0}(T) \right\rangle_V - \left\langle R_1 x_0, x_0 \right\rangle_V \right) \right. \notag \\
&\phantom{=} + \frac{1}{2aT} \int_0^T \left\langle R_1 X^{x_0}(t), \Phi dB(t) \right\rangle_V \notag \\
&\phantom{=}+ \frac{1}{4aT} \left( \left\langle R_2 X^{x_0}(T), X^{x_0}(T) \right\rangle_V - \left\langle R_2 x_0, x_0 \right\rangle_V \right) \notag \\
&\phantom{=} \left. - \frac{1}{2aT} \int_0^T \left\langle R_2 X^{x_0}(t), \Phi dB(t) \right\rangle_V \right) \notag \\
&= \frac{1}{4aY_T} \frac{1}{\sqrt{T}} \left( \left\langle (R_2 - R_1) X^{x_0}(T), X^{x_0}(T) \right\rangle_V - \left\langle (R_2 - R_1) x_0, x_0 \right\rangle_V \right) \notag \\
&\phantom{=} - \frac{1}{2aY_T} \frac{1}{\sqrt{T}} \int_0^T \left\langle (R_2 - R_1) X^{x_0}(t), \Phi dB(t) \right\rangle_V. \label{asymptotics for b tilde}
\end{align}

As above, the term
$$
\frac{1}{\sqrt{T}} \left( \left\langle (R_2 - R_1) X^{x_0}(T), X^{x_0}(T) \right\rangle_V - \left\langle (R_2 - R_1) x_0, x_0 \right\rangle_V \right) \rightarrow 0
$$
in probability as $T \rightarrow \infty$. If we denote
\begin{align}
q_2(T) &= \frac{1}{\sqrt{T}} \int_0^T \left\langle (R_2 - R_1) X^{x_0}(t), \Phi dB(t) \right\rangle_V \notag \\
&= \frac{1}{\sqrt{T}} \int_0^T \sum_{n=1}^{\infty} \sqrt{\lambda_n} \left\langle \tilde{R}_2 X^{x_0}(t), e^{\prime}_n \right\rangle_{L^2(D)} \, d\beta_n (t), \notag
\end{align}
then $\mathrm{Law} \left( q_2(T) \right)$ converges weakly to the Gaussian distribution with a zero mean and variance given by $\text{Tr} \left( Q \tilde{R}_2 Q_{\infty}^{(a,b)} \tilde{R}_2^* \right)$. Since the multiplicative factor $- \frac{1}{2aY_T}$ of $q_2(T)$ in \eqref{asymptotics for b tilde} converges to $- \frac{2b}{\text{Tr} \, Q}$ as $T \rightarrow \infty$, we obtain the result
$$
\text{Law} \left(\sqrt{T} \left( \tilde{b}_T - b \right) \right) \stackrel{w^*}{\longrightarrow} N \left( 0, \, \frac{4b^2}{(\mathrm{Tr} \, Q)^2} \, \mathrm{Tr} \left( Q \tilde{R}_2 Q_{\infty}^{(a,b)} \tilde{R}_2^* \right) \right), \quad T \rightarrow \infty.
$$
\end{proof}

\begin{remark} \label{forms of limiting variances}
It is also possible to specify the variance of the limiting Gaussian distribution of $q_1(T)$ and $q_2(T)$ as
\begin{equation} \label{trace QR1QR1}
\mathrm{Tr} \left( Q \tilde{R}_1 Q_{\infty}^{(a,b)} \tilde{R}_1^* \right) = \sum_{n=1}^{\infty} \sum_{k=1}^{\infty} \frac{2ab(\alpha_n + \alpha_k)}{b^2 (\alpha_n - \alpha_k)^2 + 8a^2 b(\alpha_n + \alpha_k)} \left\langle Qe_n, e_k \right\rangle_{L^2(D)}^2,
\end{equation}
\begin{equation} \label{trace QR2QR2}
\mathrm{Tr} \left( Q \tilde{R}_2 Q_{\infty}^{(a,b)} \tilde{R}_2^* \right) = \sum_{n=1}^{\infty} \sum_{k=1}^{\infty} \frac{16a^3}{b^2 (\alpha_n - \alpha_k)^2 + 8a^2 b(\alpha_n + \alpha_k)} \left\langle Qe_n, e_k \right\rangle_{L^2(D)}^2.
\end{equation}
\end{remark}
%\begin{proof}
%Using definitions of $\tilde{R}_1$, $\tilde{R}_1^*$, $\tilde{R}_2$, $\tilde{R}_2^*$ and formula for $Q_{\infty}^{(a,b)}$ from Theorem \ref{form of Q infinity - %theorem}, it is possible to verify that \eqref{trace QR1QR1} and \eqref{trace QR2QR2} holds true.
%\end{proof}

The family of estimators $(\tilde{a}_T, \tilde{b}_T)$ may be viewed as better than the family of estimators $(\hat{a}_T, \hat{b}_T)$, because their respective limiting variances are smaller, which is stated in the following Theorem.

\begin{theorem} \label{is smaller}
1) The limiting variance of $\sqrt{T} \left( \tilde{a}_T - a \right)$ is smaller than the li\-mi\-ting variance of $\sqrt{T} \left( \hat{a}_T - a \right)$, that is
\begin{equation} \label{comparing of a tilde and a hat}
\frac{4a^2}{\mathrm{Tr} \, Q^2} \, \mathrm{Tr} \left( Q \tilde{R}_1 Q_{\infty}^{(a,b)} \tilde{R}_1^* \right) < \frac{4a^2}{\mathrm{Tr} \, Q^2} \, \mathrm{Tr} \left( Q \tilde{R} Q_{\infty}^{(a,b)} \tilde{R}^* \right).
\end{equation}

2) The limiting variance of $\sqrt{T} \left( \tilde{b}_T - b \right)$ is smaller than the limiting variance of $\sqrt{T} \left( \hat{b}_T - b \right)$, that is
\begin{equation} \label{comparing of b tilde and b hat}
\frac{4b^2}{\mathrm{Tr} \, Q^2} \, \mathrm{Tr} \left( Q \tilde{R}_2 Q_{\infty}^{(a,b)} \tilde{R}_2^* \right) < \frac{4b^2 (b+1)^2}{\mathrm{Tr} \, Q^2} \, \mathrm{Tr} \left( Q \tilde{R} Q_{\infty}^{(a,b)} \tilde{R}^* \right).
\end{equation}
\end{theorem}
\begin{proof}
By Remarks \ref{form of limiting variance} and \ref{forms of limiting variances}, $\text{Tr} \left( Q \tilde{R} Q_{\infty}^{(a,b)} \tilde{R}^* \right)$ equals to
\begin{align}
&\sum_{n=1}^{\infty} \sum_{k=1}^{\infty} \frac{16a^3 + 2ab(b+1)^2(\alpha_n + \alpha_k)}{b^2 (\alpha_n - \alpha_k)^2 + 8a^2 b(\alpha_n + \alpha_k)} \frac{1}{(b+1)^2} \left\langle Qe_n, e_k \right\rangle_{L^2(D)}^2 \notag \\
&= \sum_{n=1}^{\infty} \sum_{k=1}^{\infty} \frac{16a^3}{b^2 (\alpha_n - \alpha_k)^2 + 8a^2 b(\alpha_n + \alpha_k)} \frac{1}{(b+1)^2} \left\langle Qe_n, e_k \right\rangle_{L^2(D)}^2 \notag \\
&\phantom{=} + \sum_{n=1}^{\infty} \sum_{k=1}^{\infty} \frac{2ab(\alpha_n + \alpha_k)}{b^2 (\alpha_n - \alpha_k)^2 + 8a^2 b(\alpha_n + \alpha_k)} \left\langle Qe_n, e_k \right\rangle_{L^2(D)}^2 \notag \\
&= \frac{1}{(b+1)^2} \, \text{Tr} \left( Q \tilde{R}_2 Q_{\infty}^{(a,b)} \tilde{R}_2^* \right) + \text{Tr} \left( Q \tilde{R}_1 Q_{\infty}^{(a,b)} \tilde{R}_1^* \right). \notag
\end{align}

Since both above terms are positive, \eqref{comparing of a tilde and a hat} and \eqref{comparing of b tilde and b hat} follow.
\end{proof}

\begin{remark} \label{diagonal case remark}
If we consider so--called "diagonal case", that is $Qe_n = \lambda_n e_n$ for orthonormal basis $\{ e_n, n \in \mathbb N \}$ in $L^2(D)$, many of the previous formulae may be considerably simplified. The covariance operator $Q_{\infty}^{(a,b)}$ from Theorem \ref{form of Q infinity - theorem} will take the form
\begin{equation}
Q_{\infty}^{(a,b)} = \left( \begin{array}{cc} \frac{1}{4ab} \, Q & 0 \\ 0 & \frac{1}{4a} \, Q \end{array} \right),
\end{equation}
with the same trace given by Lemma \ref{trace of Q infinity},
$$
\text{Tr} \, Q_{\infty}^{(a,b)} = \frac{1}{4ab} \, \text{Tr} \, Q + \frac{1}{4a} \, \text{Tr} \, Q = \frac{b+1}{4ab} \, \text{Tr} \, Q.
$$

Also the limiting variances of Gaussian distributions in Theorems \ref{asymptotic normality of a hat and b hat} and \ref{asymptotic normality of a tilde and b tilde} may be further specified as
\begin{align}
\text{Law} \left( \sqrt{T} \left( \hat{a}_T - a \right) \right) &\stackrel{w^*}{\longrightarrow} N \left( 0, \, \frac{1}{(\text{Tr} \, Q)^2} \left( \frac{4a^3}{b(b+1)^2} \, \text{Tr} \left( Q^2 (-A)^{-1} \right) + a \, \text{Tr} \, Q^2 \right) \right), \notag \\
\text{Law} \left( \sqrt{T} \left( \hat{b}_T - b \right) \right) &\stackrel{w^*}{\longrightarrow} N \left( 0, \, \frac{1}{(\text{Tr} \, Q)^2} \left( 4ab \, \text{Tr} \left( Q^2 (-A)^{-1} \right) + \frac{b^2 (b+1)^2}{a} \, \text{Tr} \, Q^2 \right) \right), \notag \\
\text{Law} \left( \sqrt{T} \left( \tilde{a}_T - a \right) \right) &\stackrel{w^*}{\longrightarrow} N \left( 0, \, a \, \frac{\text{Tr} \, Q^2}{(\text{Tr} \, Q)^2} \right), \notag \\
\text{Law} \left( \sqrt{T} \left( \tilde{b}_T - b \right) \right) &\stackrel{w^*}{\longrightarrow} N \left( 0, \, 4ab \, \frac{\text{Tr} \left( Q^2 (-A)^{-1} \right)}{(\text{Tr} \, Q)^2} \right), \notag
\end{align}
for $T \rightarrow \infty$.
\end{remark}

\section{Examples} \label{section: examples}
\begin{example} \label{wave equation}
Consider the wave equation with Dirichlet boundary conditions
\begin{align}
\frac{\partial^2 u}{\partial t^2} (t, \xi) &= b\Delta u(t, \xi) - 2a \frac{\partial u}{\partial t}(t, \xi) + \eta (t, \xi), \quad (t, \xi) \in \mathbb R_+ \times D, \label{example 1} \\
u(0, \xi) &= u_1(\xi), \quad \xi \in D, \notag \\
\frac{\partial u}{\partial t} (0, \xi) &= u_2(\xi), \quad \xi \in D, \notag \\
u(t, \xi) &= 0, \quad (t, \xi) \in \mathbb R_+ \times \partial D, \notag
\end{align}
where $D \subset \mathbb R^d$ is a bounded domain with a smooth boundary, $\eta$ is a noise process that is the formal time derivative of a space dependent Brownian motion and $a > 0$, $b > 0$ are unknown parameters.

We rewrite the hyperbolic system \eqref{example 1} as an infinite dimensional stochastic di\-ffer\-en\-tial equation \eqref{linear equation with parameters}
\begin{align}
dX(t) &= \mathcal A X(t) \, dt + \Phi \, dB(t), \notag \\
X(0) &= x_0 = \left( \begin{array}{c}
u_1\\
u_2
\end{array} \right) \notag
\end{align}
for $t \geq 0$, setting $A = \Delta|_{\text{Dom}(A)}$, $\text{Dom}(A) = H^2(D) \cap H^1_0 (D)$, $\text{Dom}(\mathcal A) = \text{Dom}(A) \times \text{Dom}( (-A)^{\frac{1}{2}})$ and
$$
\mathcal A = \left( \begin{array}{cc}
0&I\\
bA&-2aI
\end{array} \right).
$$

The operator $\mathcal A$ generates strongly continuous semigroup in the space $V =$ \\ $\text{Dom}((-A)^{\frac{1}{2}}) \times L^2(D)$. The driving process may take a form $B(t) = (0, \tilde{B}(t))^{\top}$, where $(\tilde{B}(t), t \geq 0)$ is a standard cylindrical Brownian motion on $L^2(D)$. The noise $\eta$ is modelled as the formal derivative $\Phi_1 \, \frac{d\tilde{B}(t)}{dt}$, $\Phi_1 \in \mathcal L_2(L^2(D))$ and $\Phi \in \mathcal L_2(V)$ is given by
$$
\Phi = \left( \begin{array}{cc}
0&0\\
0&\Phi_1
\end{array} \right).
$$

With this setup, all assumptions of Section \ref{main results} are fulfilled, so Theorems \ref{strong consistency of a hat and b hat} and \ref{strong consistency of a tilde and b tilde} may be used for estimation of parameters. Theorems \ref{asymptotic normality of a hat and b hat} and \ref{asymptotic normality of a tilde and b tilde}, which show asymptotic normality of these estimators, may be applied as well.

The operator $Q = \Phi_1 \Phi_1^*$ which appears in the formulae for estimators established in these Theorems may be interpreted as the "covariance in space" of the driving process $(\tilde{B}(t), t \geq 0)$, that is
$$
\mathbb E \left\langle \tilde{B}(t, \cdot), x \right\rangle_{L^2(D)} \left\langle \tilde{B}(t, \cdot), y \right\rangle_{L^2(D)} = t \left\langle Qx, y \right\rangle_{L^2(D)}, \quad t \geq 0,
$$
for $x,y \in L^2(D)$ (cf. \cite{dapratozabczyk}).
\end{example}

\begin{example} \label{plate equation}
Consider the plate equation with Dirichlet boundary conditions
\begin{align}
\frac{\partial^2 u}{\partial t^2} (t, \xi) &= -b\Delta^2 u(t, \xi) - 2a \frac{\partial u}{\partial t}(t, \xi) + \eta (t, \xi), \quad (t, \xi) \in \mathbb R_+ \times D, \label{example 2} \\
u(0, \xi) &= u_1(\xi), \quad \xi \in D, \notag \\
\frac{\partial u}{\partial t} (0, \xi) &= u_2(\xi), \quad \xi \in D, \notag \\
u(t, \xi) &= 0, \quad (t, \xi) \in \mathbb R_+ \times \partial D, \notag
\end{align}
where $D$, $\eta$, $a$ and $b$ satisfy the conditions in Example \ref{wave equation}.

We rewrite the hyperbolic system \eqref{example 2} as an infinite dimensional stochastic di\-ffer\-en\-tial equation \eqref{linear equation with parameters}, setting $A = \Delta|_{\text{Dom}(A)}$, $\text{Dom}(A) = H^2(D) \cap H^1_0 (D)$, $\text{Dom}(\mathcal A) = \text{Dom}(A^2) \times \text{Dom}(A)$ and
$$
\mathcal A = \left( \begin{array}{cc}
0&I\\
-b A^2&-2aI
\end{array} \right).
$$

The operator $\mathcal A$ generates strongly continuous semigroup in the space $V = $ \\ $\text{Dom}(A) \times L^2(D)$. The driving process may take a form $B(t) = (0, \tilde{B}(t))^{\top}$, where $(\tilde{B}(t), t \geq 0)$ is a standard cylindrical Brownian motion on $L^2(D)$. The noise $\eta$ is modelled as the formal derivative $\Phi_1 \, \frac{d\tilde{B}(t)}{dt}$, $\Phi_1 \in \mathcal L_2(L^2(D))$ and $\Phi \in \mathcal L_2(V)$ is given by
$$
\Phi = \left( \begin{array}{cc}
0&0\\
0&\Phi_1
\end{array} \right).
$$

The interpretation of the noise term is the same as in Example \ref{wave equation}.

In this case, all assumptions made in Section \ref{main results} are satisfied.
\end{example}

\section{Implementation and statistical evidence} \label{section: implementation}
We have generated a trajectory of the solution to the stochastic differential equation \eqref{example 1} from Example \ref{wave equation} in the program $\mathsf R$ by Euler's method (see \cite{iacus}). The setup of Example \ref{wave equation} is specified as follows:

\begin{itemize}
\item $D = (0,1)$ -- We consider the wave equation for the oscillating rod modeled as a function from the space $L^2((0,1))$.
\item The choice of the orthonormal basis of the space $L^2((0,1))$ is
$$
\{e_n (\xi) = \sqrt{2} \sin (n \pi \xi), n = 1, \ldots, N\},
$$
the elements of which satisfy the boundary condition $u(t, 0) = 0 = u(t, 1)$, for any $t > 0$.
\item $N = 10$ -- Due to possible memory limitations, we have restricted the expansion of the previous basis only to $N = 10$ functions. The accuracy of our results may suffer due to this limitation, nevertheless we will show that our results are sufficiently satisfactory.
\item $T = 100$ -- The length of the time interval.
\item $\Delta t = 0.001$ -- The mesh of the partition of the time interval $[0, T]$.
\item The intial functions $u_1$ and $u_2$ have the following form
$$
u_1 (\xi) = \sqrt{2} \sum_{n=1}^{N} \sin (n \pi \xi) = u_2 (\xi).
$$
This means that $\left\langle u_1, e_n \right\rangle_{L^2(D)} = 1 = \left\langle u_2, e_n \right\rangle_{L^2(D)}$ for any $n = 1, \ldots, N$, so the initial conditions are the same in all $N$ dimensions.
\item $a = 1$, $b = 0.2$ -- The values of the parameters that are to be estimated.
\item $- \alpha_n = - n^2 \pi^2$ -- The eigenvalues of the operator $A$. With this setup the operator $A$ is the Laplacian operator $A = \Delta |_{\mathrm{Dom}(A)}$ with $\mathrm{Dom}(A) = H^2((0,1)) \cap H_0^1((0,1))$.
\item $\lambda_n = \frac{1000}{n^2}$ -- The eigenvalues of the operator $Q$. (The eigenvalues of the ope\-ra\-tor $\Phi_1$ equal to $\sqrt{\lambda_n}$ for any $n = 1, \ldots, N$.) The eigenvalues are chosen in the way that the sum $\sum_{n=1}^{\infty} \lambda_n$ is convergent. The multiplication factor is chosen in order to increase the values of the $\lambda_n$. Otherwise the noise would be in "higher" dimensions so small that it would be practically vanishing.
\item We consider the "diagonal case", i.e., the eigenvectors of the operators $A$ and $Q$ coincide and form the basis $\{ e_n (\cdot), n = 1, \ldots, N \}$.
\end{itemize}

From the generated trajectory, we obtained the following results: The value of the statistic $I_T$ (on which the estimators $\hat{a}_T$ and $\hat{b}_T$ are based on (see Theorem \ref{strong consistency of a hat and b hat})) is $I_T = 2740.959$, while the trace of the operator $Q_{\infty}^{(a,b)}$ equals to $\mathrm{Tr} \, Q_{\infty}^{(a,b)} = \frac{b+1}{4ab} \sum_{n=1}^N \lambda_n = 2324.652$ (since we have restricted ourselves to just $N=10$ dimensions, we use only the sum of the $N$ eigenvalues to compute $\mathrm{Tr} \, Q$). The estimators of $a$ and $b$ are $\hat{a}_T = 0.8481$ and $\hat{b}_T = 0.1646$ and their time evolution is shown in Figure \ref{Figure 1}.

\begin{figure}[h!]
\centering
\subfigure[The estimator $\hat{a}_t$]{\epsfig{file=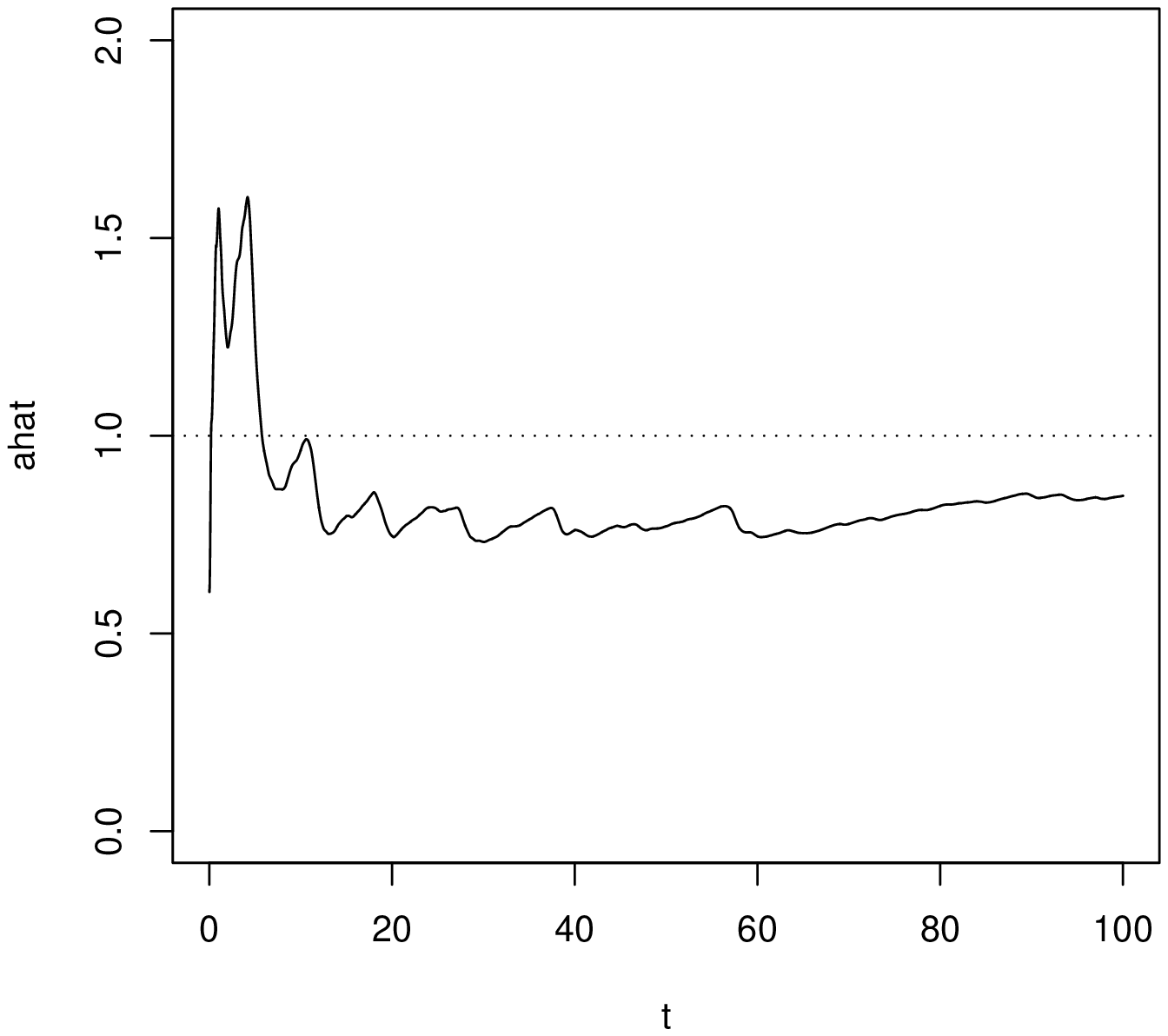,height=0.4\textwidth}}
\hspace{50pt}
\subfigure[The estimator $\hat{b}_t$]{\epsfig{file=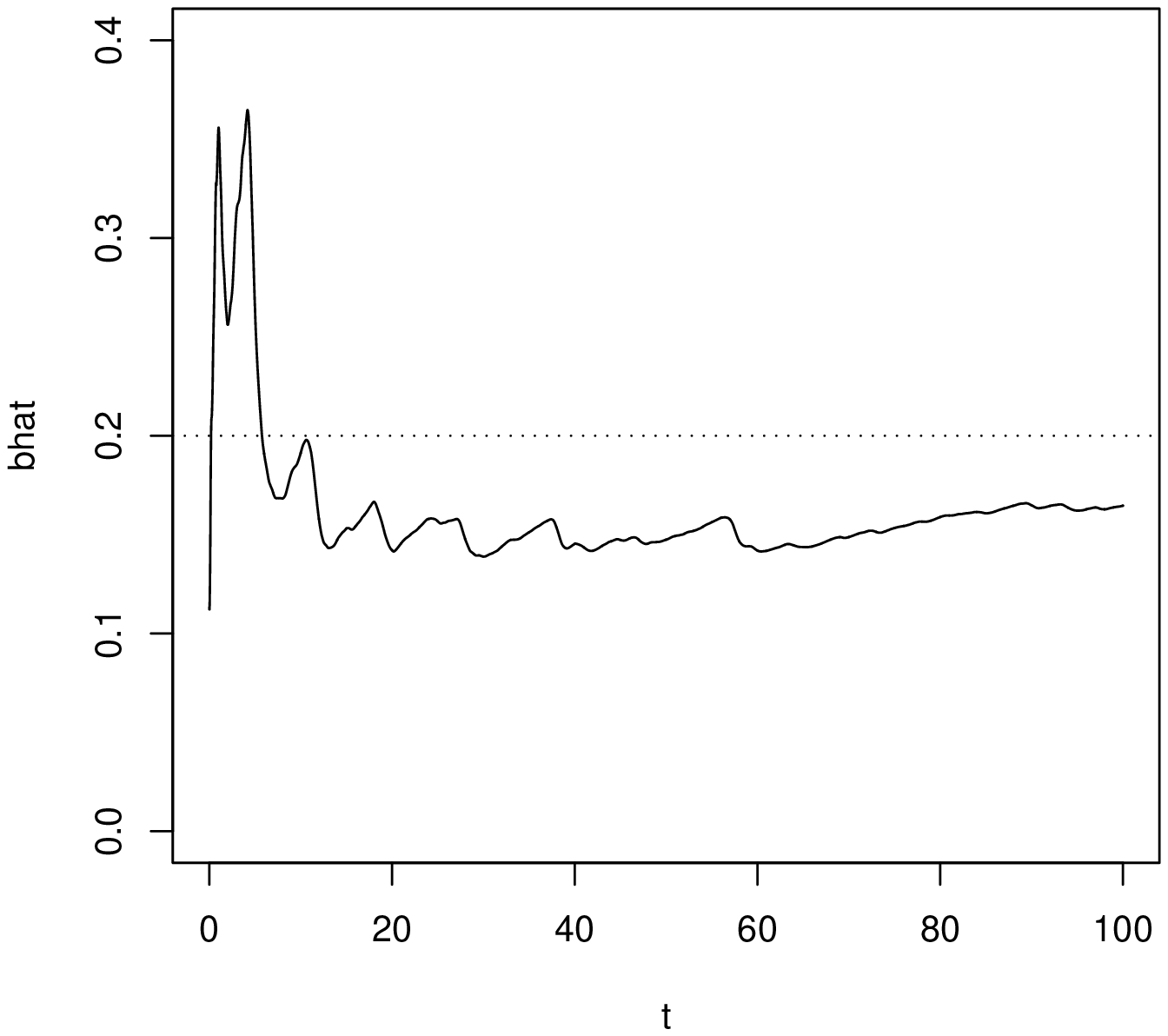,height=0.4\textwidth}}
\caption{The time evolution of the estimators $\hat{a}_t$ and $\hat{b}_t$} \label{Figure 1}
\end{figure}

Let us compute the estimators $\tilde{a}_T$ and $\tilde{b}_T$ from Theorem \ref{strong consistency of a tilde and b tilde}. The results are the following
$$
Y_T = 2330.218, \quad \frac{1}{4ab} \sum_{n=1}^N \lambda_n = 1937.210,
$$
$$
H_T = 410.741, \quad \frac{1}{4a} \sum_{n=1}^N \lambda_n = 387.442,
$$
$$
\tilde{a}_T = 0.9433, \quad \tilde{b}_T = 0.1763.
$$

Time evolution of the estimators $\tilde{a}_t$ and $\tilde{b}_t$ is shown in the Figure \ref{Figure 2}.

\begin{figure}[h!]
\centering
\subfigure[The estimator $\tilde{a}_t$]{\epsfig{file=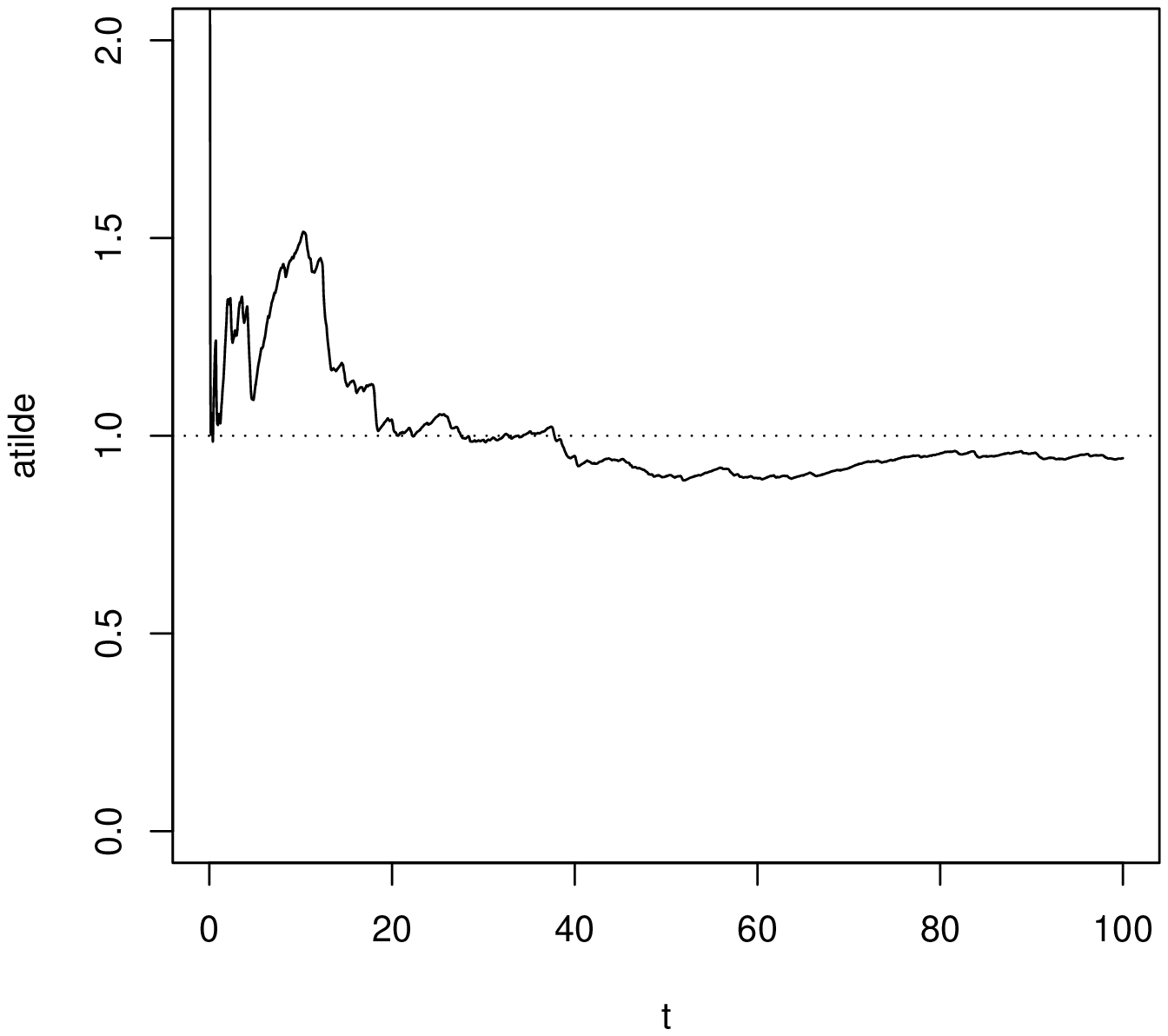,height=0.4\textwidth}}
\hspace{50pt}
\subfigure[The estimator $\tilde{b}_t$]{\epsfig{file=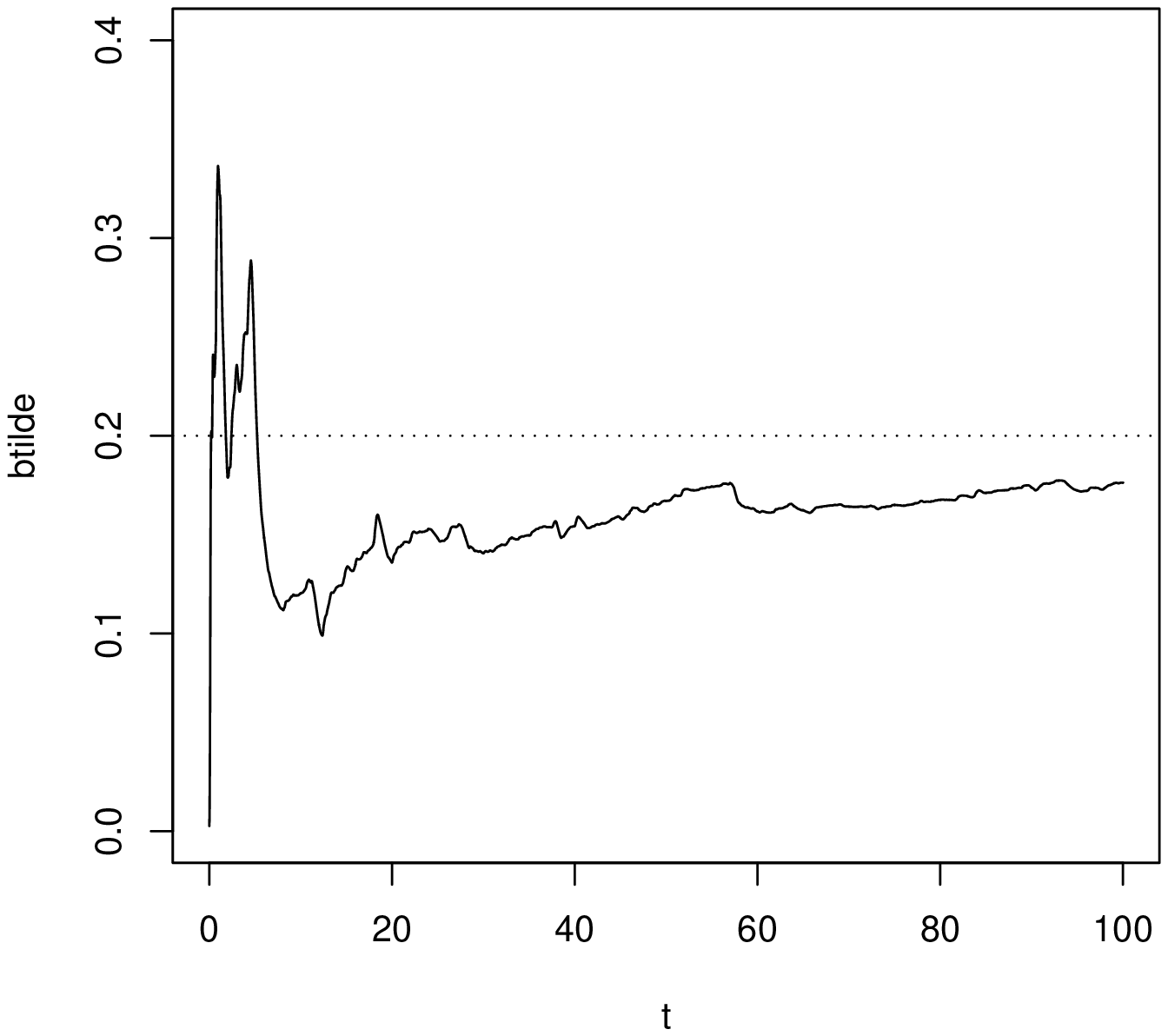,height=0.4\textwidth}}
\caption{The time evolution of the estimators $\tilde{a}_t$ and $\tilde{b}_t$} \label{Figure 2}
\end{figure}

From the figures (and also from the results) it seems that the family of estimators $(\tilde{a}_T, \tilde{b}_T)$ was better than the family $(\hat{a}_T, \hat{b}_T)$, nevertheless we have made $100$ more simulations in a similar manner. The values of the estimators $\hat{a}_T$ and $\hat{b}_T$ are depicted in Figure \ref{Figure 3} and the values of the estimators $\tilde{a}_T$ and $\tilde{b}_T$ are depicted in Figure \ref{Figure 4}. The overall statistics can be found in Table \ref{Table 1}.

\begin{figure}[h!]
\centering
\subfigure[The values of $\hat{a}_T$ -- Overall]{\epsfig{file=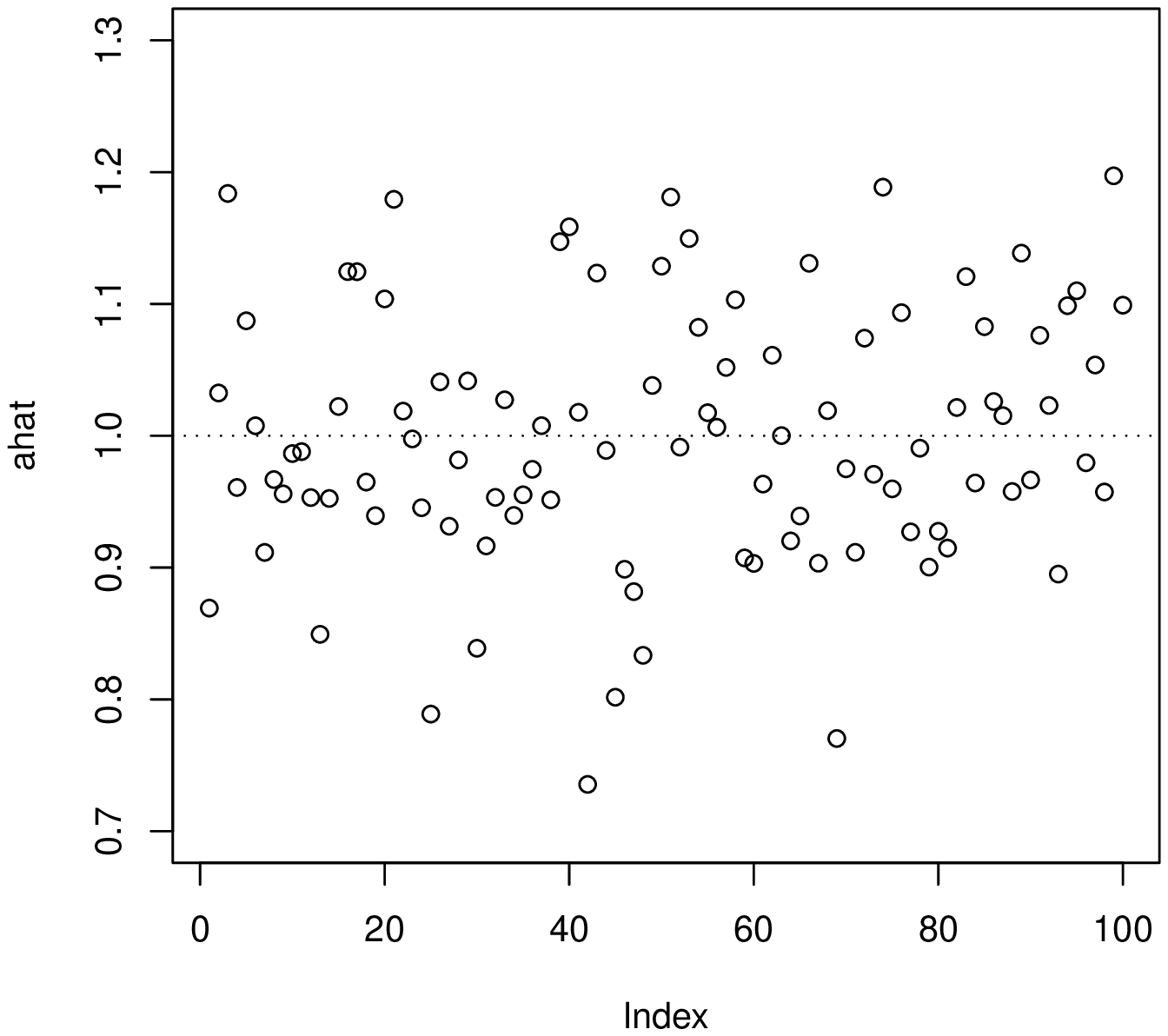,height=0.4\textwidth}}
\hspace{50pt}
\subfigure[The values of $\hat{b}_T$ -- Overall]{\epsfig{file=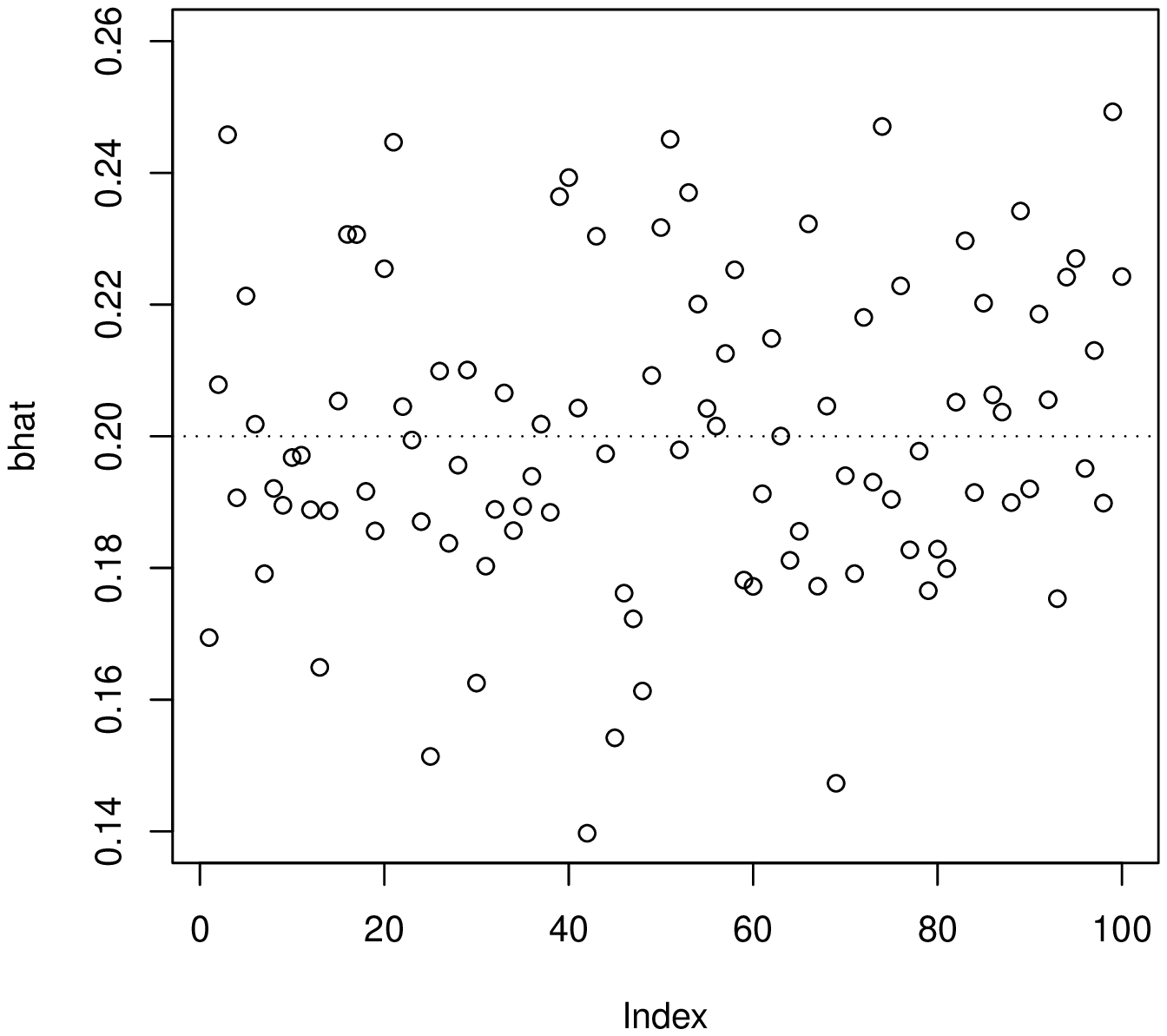,height=0.4\textwidth}}
\caption{The estimators $\hat{a}_T$ and $\hat{b}_T$ based on larger sample} \label{Figure 3}
\end{figure}

\begin{figure}[h!]
\centering
\subfigure[The values of $\tilde{a}_T$ -- Overall]{\epsfig{file=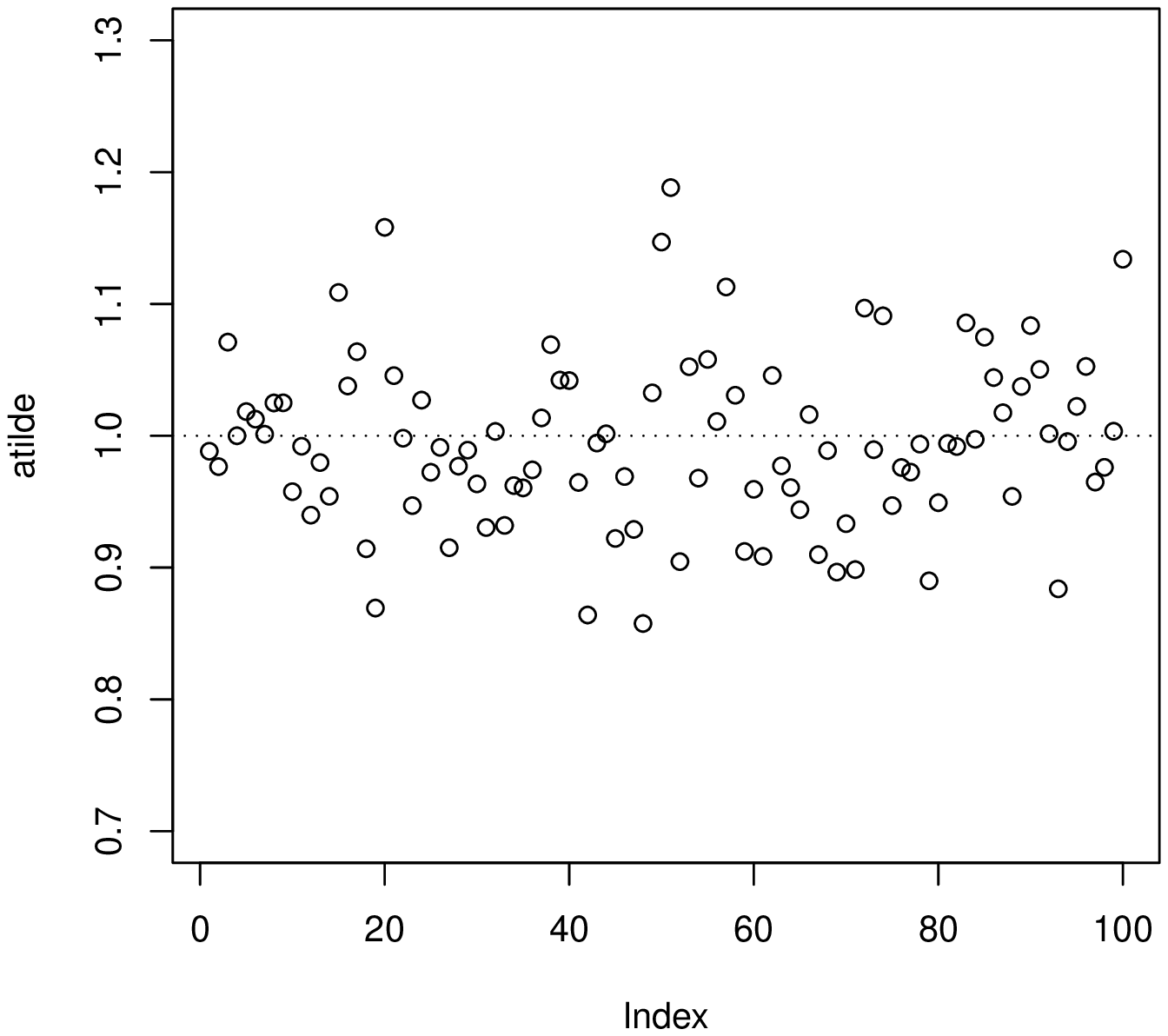,height=0.4\textwidth}}
\hspace{50pt}
\subfigure[The values of $\tilde{b}_T$ -- Overall]{\epsfig{file=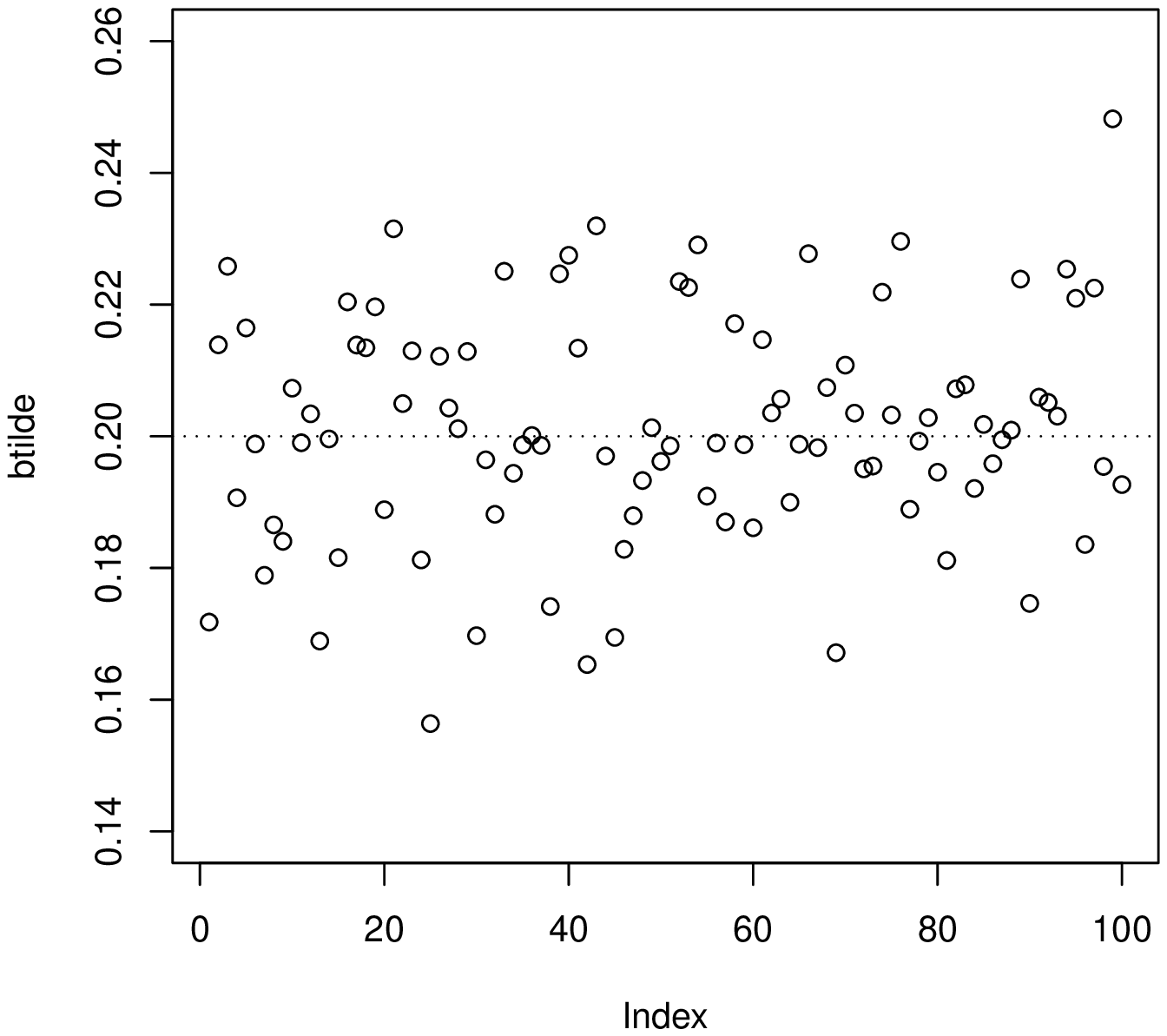,height=0.4\textwidth}}
\caption{The estimators $\tilde{a}_T$ and $\tilde{b}_T$ based on larger sample} \label{Figure 4}
\end{figure}

\begin{table}[h!]
\centering
\begin{tabular}{|l||c|c|c|c|} \hline
& $\hat{a}_T$ & $\hat{b}_T$ & $\tilde{a}_T$ & $\tilde{b}_T$ \\ \hline \hline
Mean & 0.9994 & 0.2003 & 0.9948 & 0.2013 \\ \hline
Var & 0.9473 & 0.0545 & 0.4218 & 0.0298 \\ \hline
Var -- Theoretical & 1.0466 & 0.0776 & 0.4505 & 0.0343 \\ \hline
Relative error -- Maximal & 26 \% & 30 \% & 20 \% & 24 \% \\ \hline
Relative error -- Typical & $\leq$ 10 \% & $\leq$ 10 \% & $\leq$ 5 \% & $\leq$ 7 \% \\ \hline
$p$--value & 0.2746 & 0.2728 & 0.3790 & 0.5800 \\ \hline
\end{tabular}
\bigskip
\caption{The results of the simulation} \label{Table 1}
\end{table}

The row "Var" stands for the variance of $\sqrt{T}(\hat{a} - a)$ (and its analogues in the following columns). The actual variances of the estimators are $100$ times smaller. The theoretical values of the limiting variances (see formulae in Remark \ref{diagonal case remark}) can be found in the row "Var -- Theoretical".

Since the absolute errors of the estimators can be viewed in Figures \ref{Figure 3} and \ref{Figure 4}, we mention only relative errors: maximal (which is the relative error of the worst estimator) and typical (that is the level below which $75$ \% of the errors belong).

The $p$--values of the Wilk--Shapiro test of normality can be found in the last row. Since they are greater than $0.05$, we do not reject the hypothesis of normality on $5 \%$--significance level. The Q--Q plots of the centered and rescaled estimators are shown in Figures \ref{Figure 5} and \ref{Figure 6}.

\begin{figure}[h!]
\centering
\subfigure[Q--Q plot of $\sqrt{T}(\hat{a}_T - a)$]{\epsfig{file=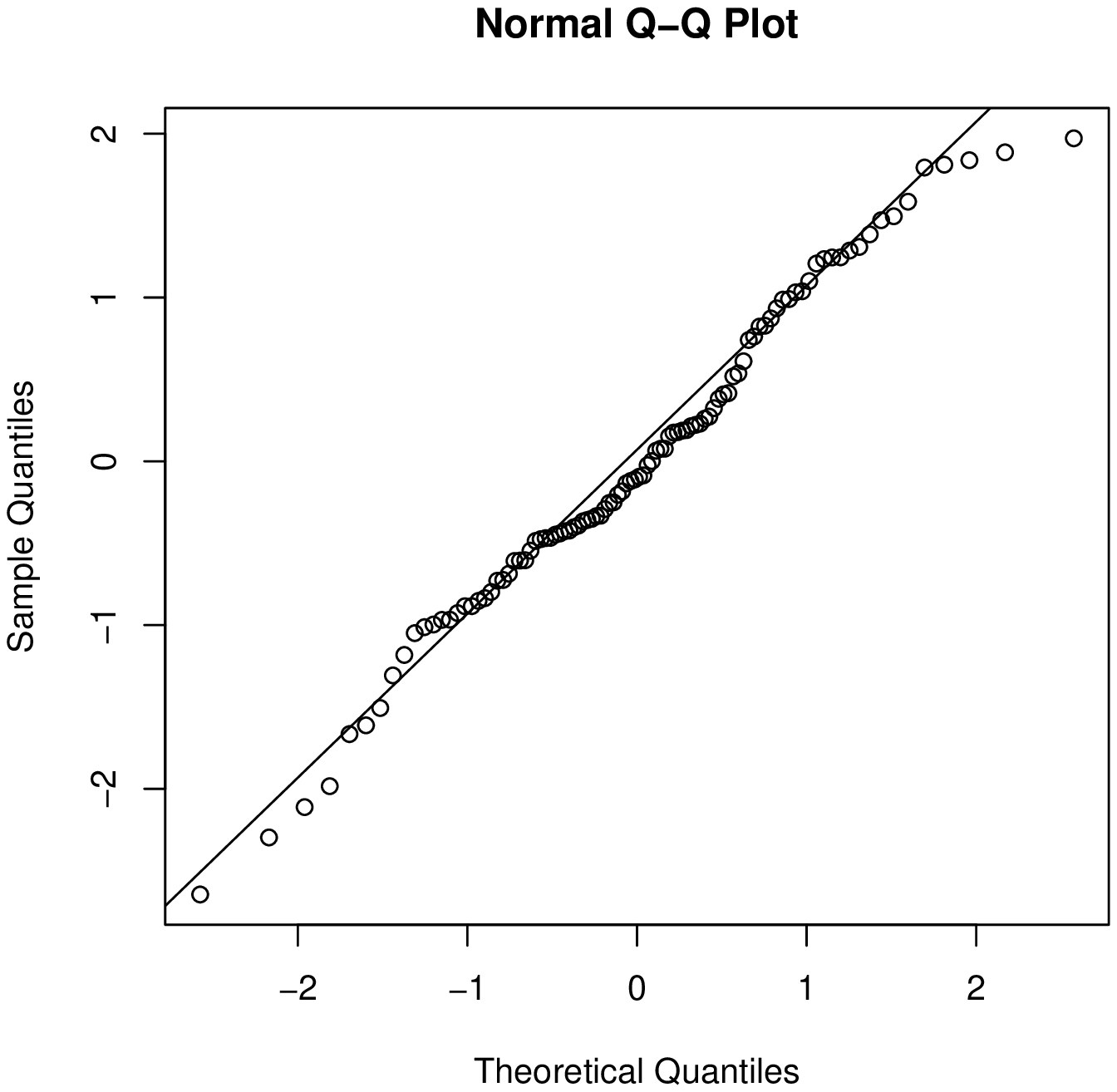,height=0.4\textwidth}}
\hspace{50pt}
\subfigure[Q--Q plot of $\sqrt{T}(\hat{b}_T - b)$]{\epsfig{file=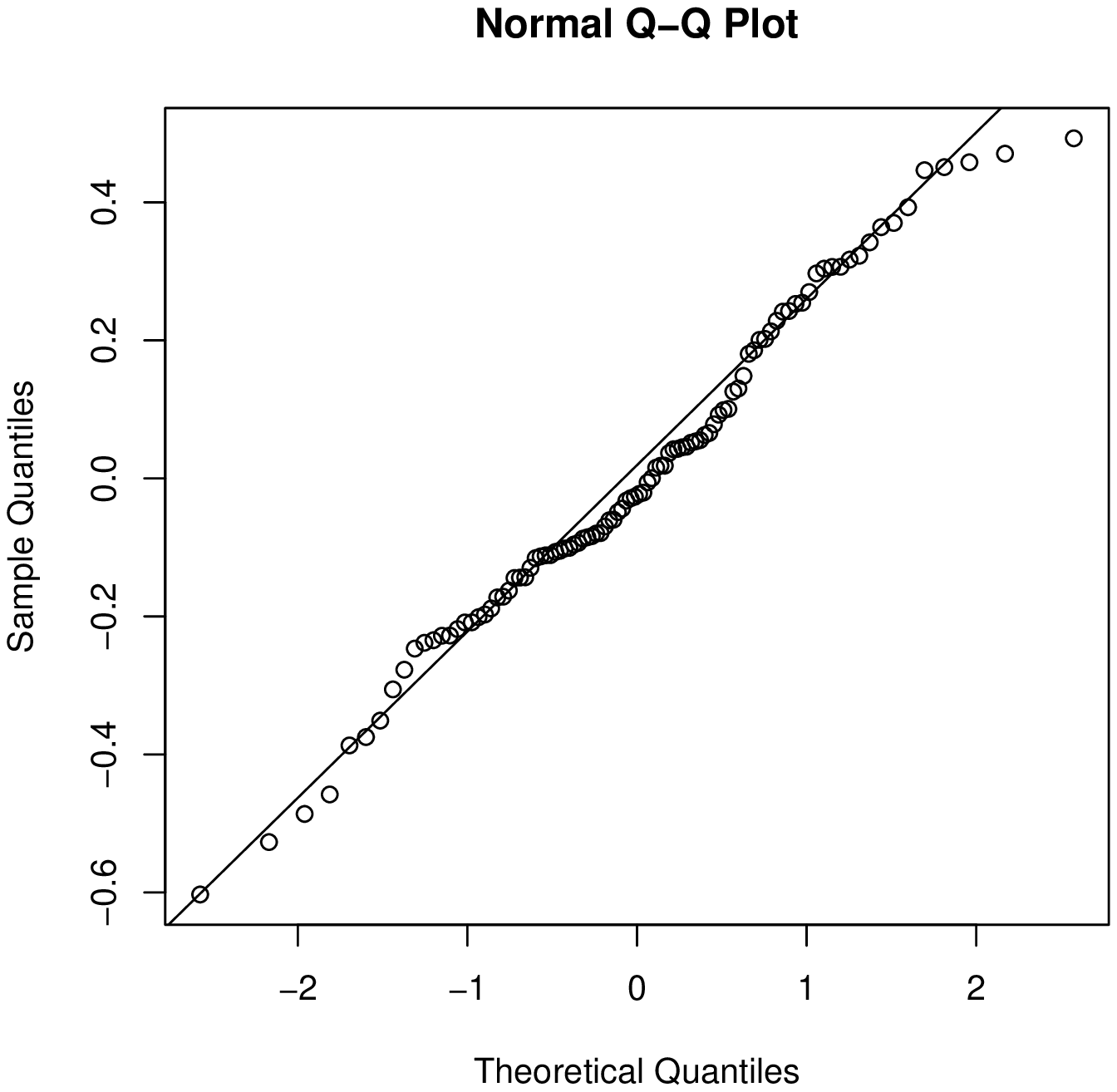,height=0.4\textwidth}}
\caption{Asymptotic normality of $\hat{a}_T$ and $\hat{b}_T$} \label{Figure 5}
\end{figure}

\begin{figure}[h!]
\centering
\subfigure[Q--Q plot of $\sqrt{T}(\tilde{a}_T - a)$]{\epsfig{file=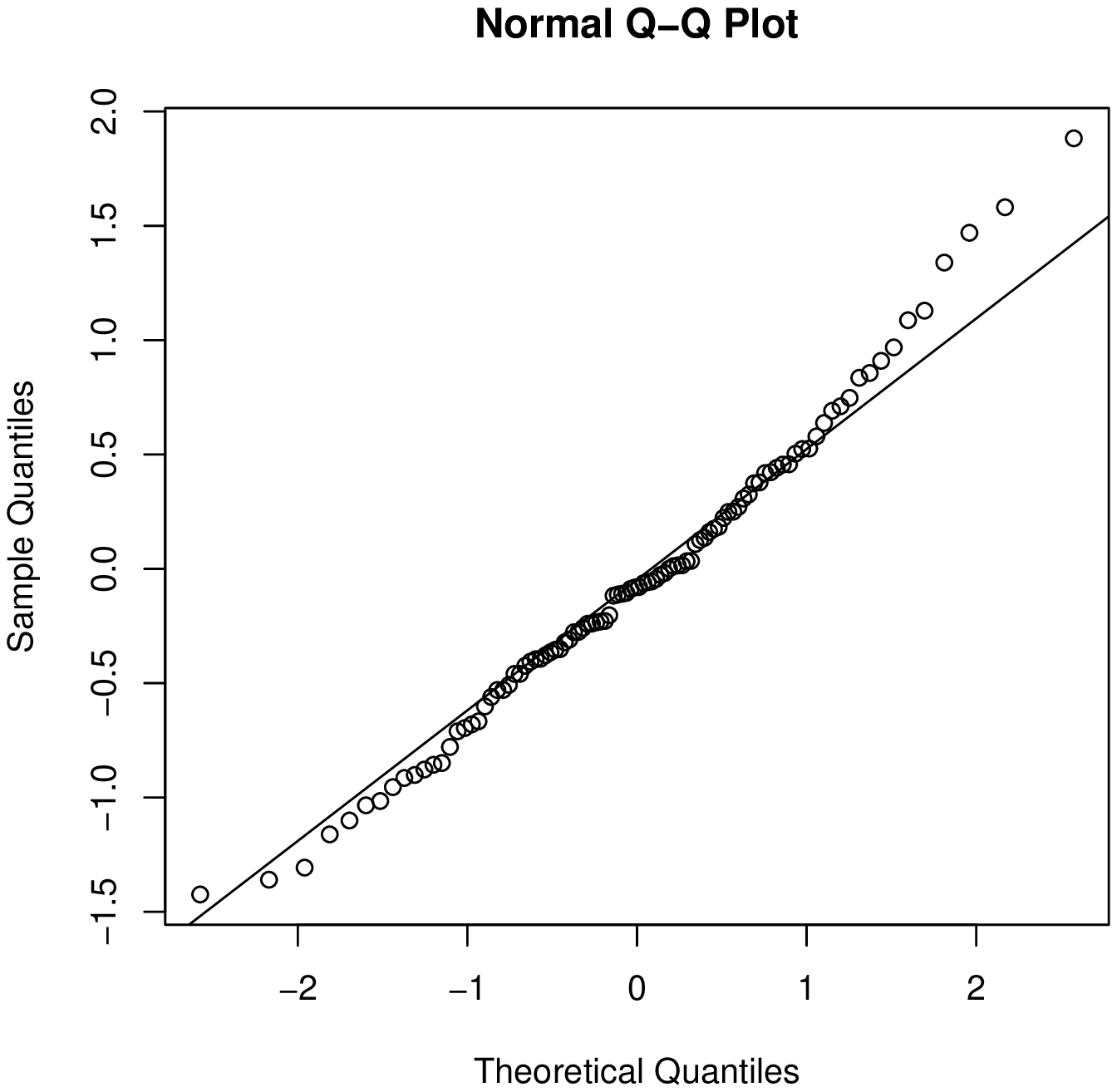,height=0.4\textwidth}}
\hspace{50pt}
\subfigure[Q--Q plot of $\sqrt{T}(\tilde{b}_T - b)$]{\epsfig{file=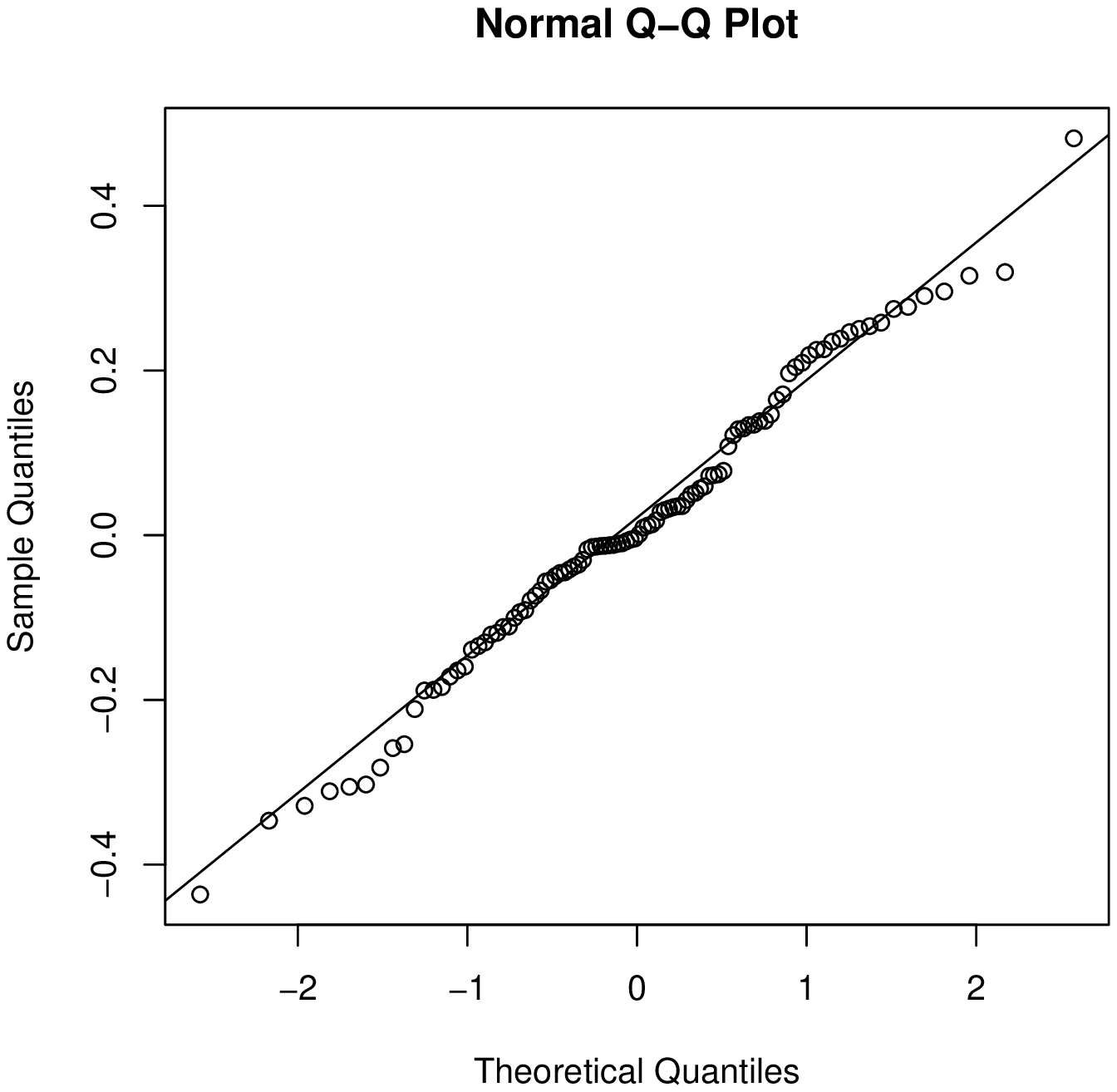,height=0.4\textwidth}}
\caption{Asymptotic normality of $\tilde{a}_T$ and $\tilde{b}_T$} \label{Figure 6}
\end{figure}

From the previous simulations the main three observations follow:

\begin{itemize}
\item The family of the estimators $(\tilde{a}_T, \tilde{b}_T)$ has similar mean as the family $(\hat{a}_T, \hat{b}_T)$, but in addition it has smaller variances and smaller relative errors. That behaviour is the consequence of Theorem \ref{is smaller}.
\item From the comparing of the rows "Var" and "Var -- Theoretical" it seems that the limiting variances from Remark \ref{diagonal case remark} are accurate.
\item From the Figures \ref{Figure 5}, \ref{Figure 6} and from the results of the Wilk--Shapiro tests it seems that the estimators are asymptotically normally distributed as prescribed.
\end{itemize}

Although these results for time $T = 100$ are satisfactory enough, we have also made simulations for time $T = 1000$. The results from one particular trajectory are the following
$$
I_T = 2360.458, \quad \frac{b+1}{4ab} \sum_{n=1}^N \lambda_n = 2324.652,
$$
$$
\hat{a}_T = 0.9848, \quad \hat{b}_T = 0.1964,
$$
$$
Y_T = 1975.777, \quad \frac{1}{4ab} \sum_{n=1}^N \lambda_n = 1937.210,
$$
$$
H_T = 384.681, \quad \frac{1}{4a} \sum_{n=1}^N \lambda_n = 387.442,
$$
$$
\tilde{a}_T = 1.0072, \quad \tilde{b}_T = 0.1947.
$$

Time evolution of the estimators $(\hat{a}_T, \hat{b}_T)$ is shown in Figure \ref{Figure 7} and time evolution of the estimators $(\tilde{a}_T, \tilde{b}_T)$ can be seen in Figure \ref{Figure 8}.

\begin{figure}[h!]
\centering
\subfigure[The estimator $\hat{a}_t$]{\epsfig{file=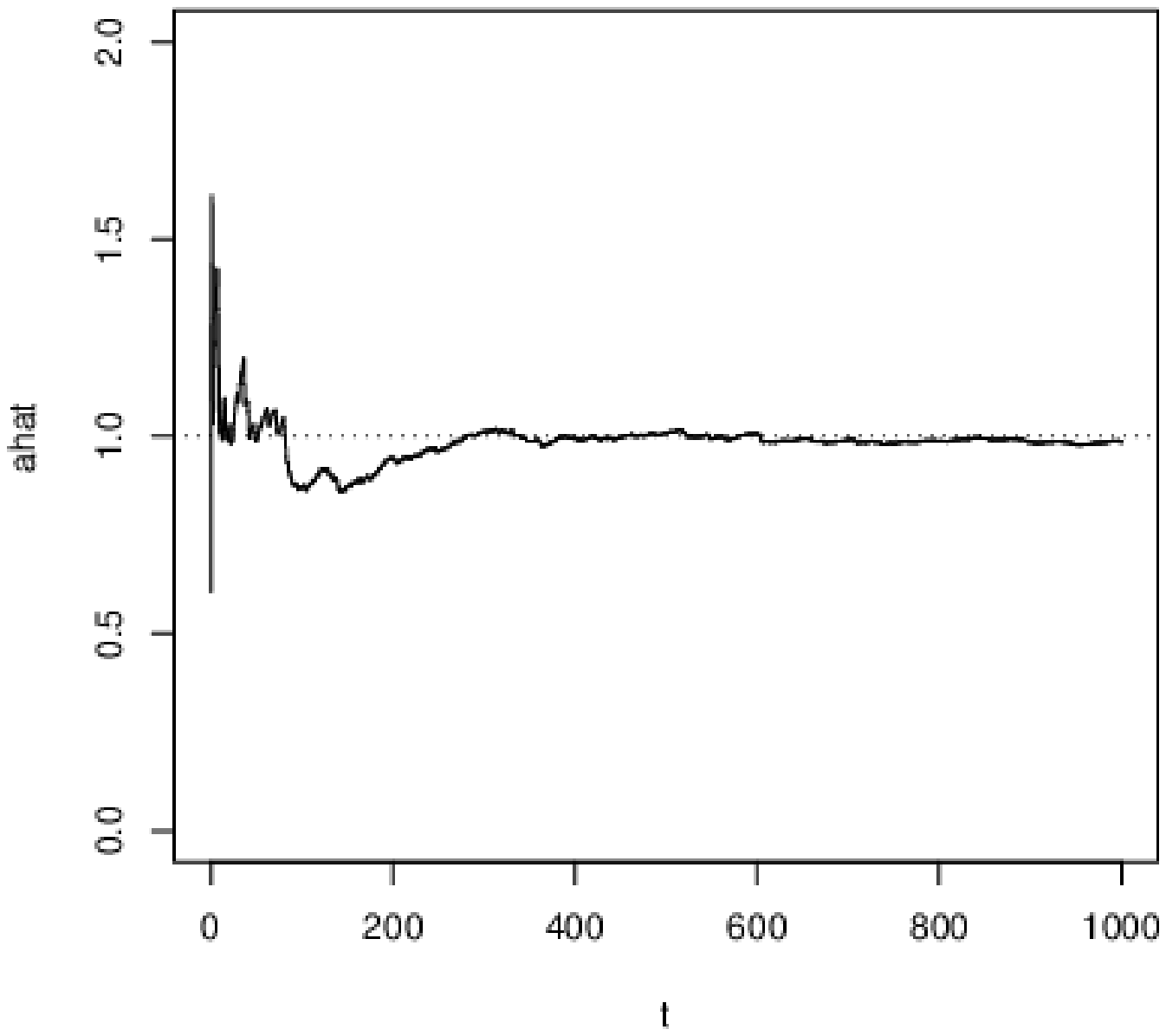,height=0.4\textwidth}}
\hspace{50pt}
\subfigure[The estimator $\hat{b}_t$]{\epsfig{file=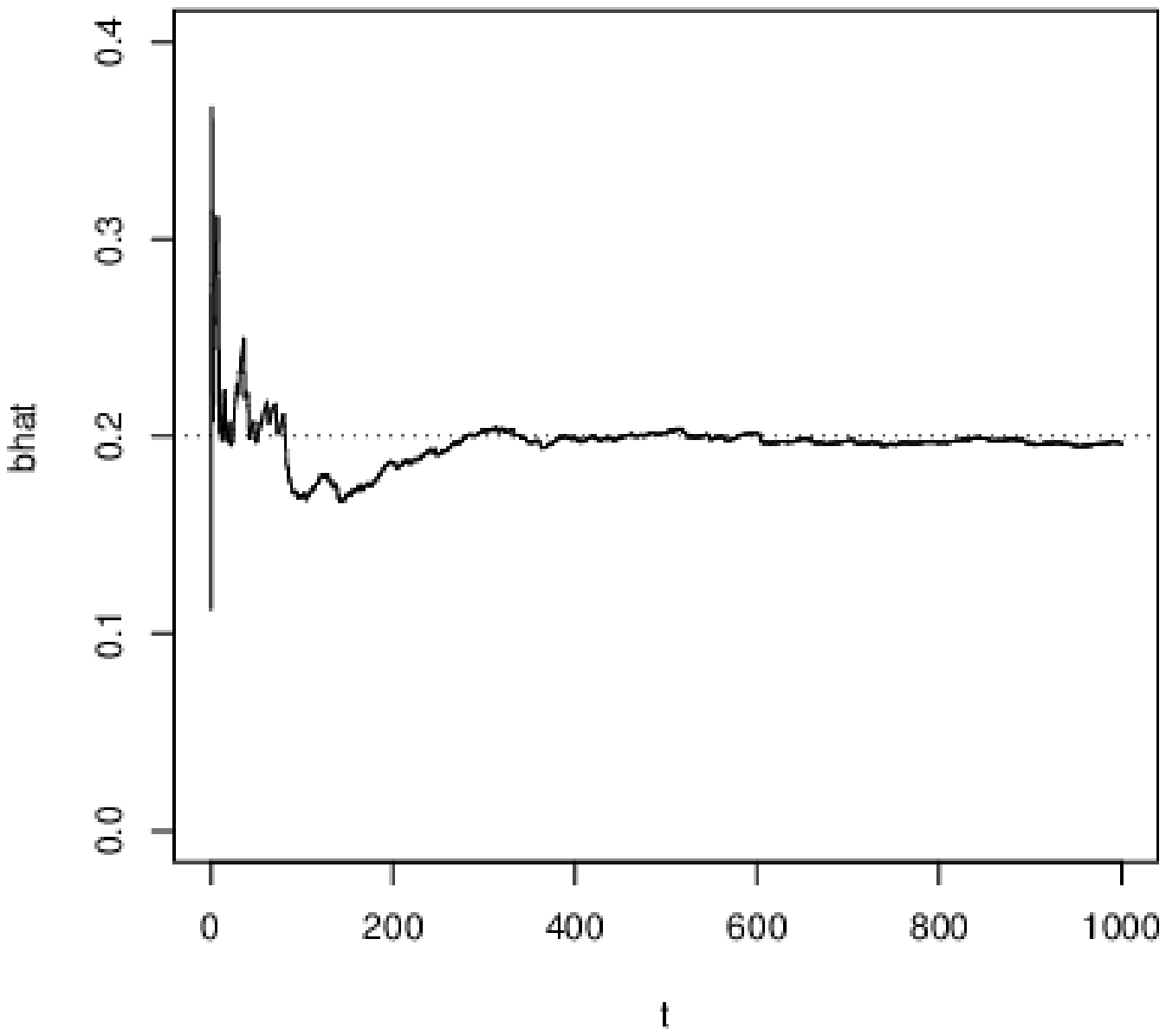,height=0.4\textwidth}}
\caption{The time evolution of the estimators $\hat{a}_t$ and $\hat{b}_t$, $T = 1000$} \label{Figure 7}
\end{figure}

\begin{figure}[h!]
\centering
\subfigure[The estimator $\tilde{a}_t$]{\epsfig{file=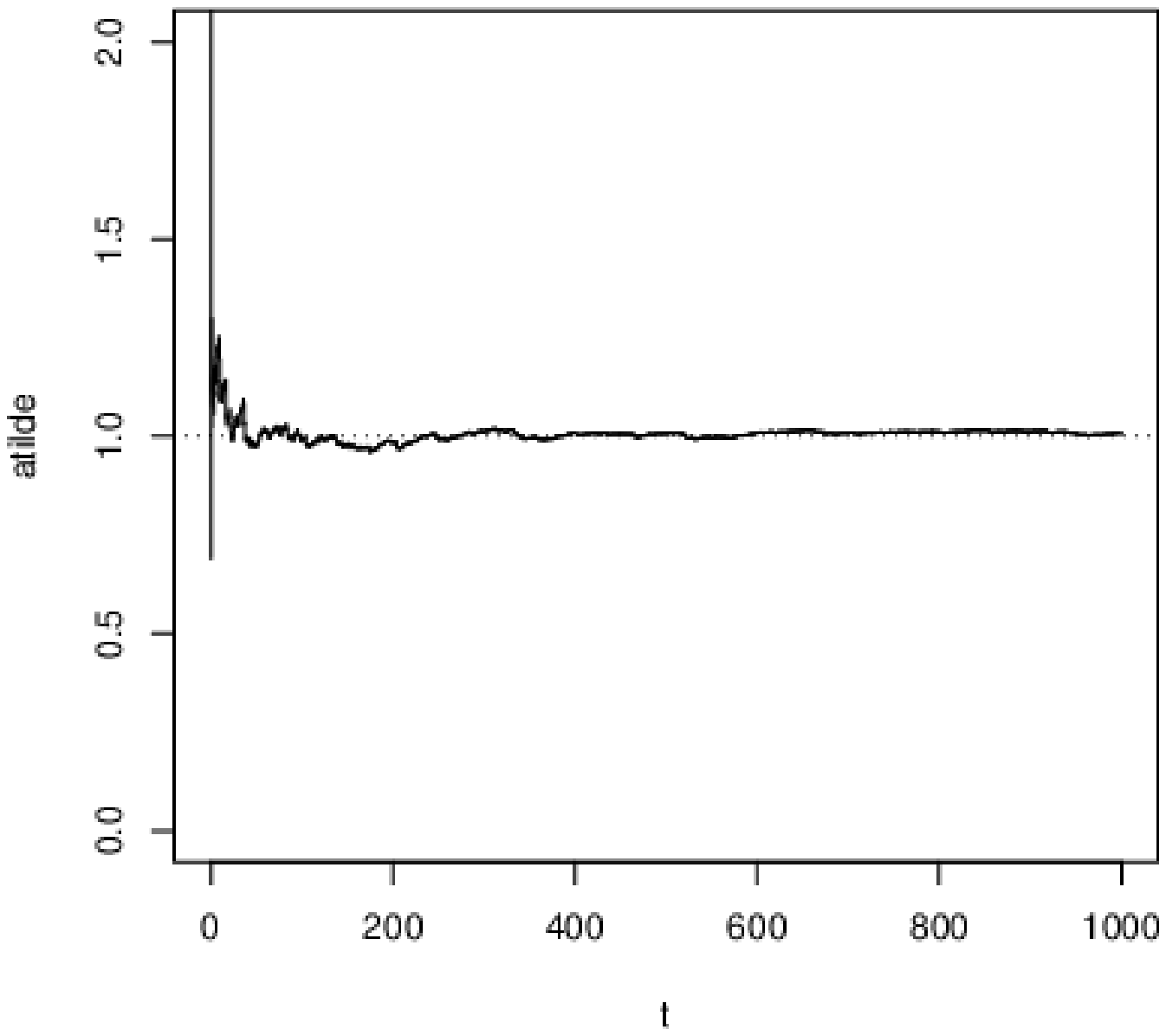,height=0.4\textwidth}}
\hspace{50pt}
\subfigure[The estimator $\tilde{b}_t$]{\epsfig{file=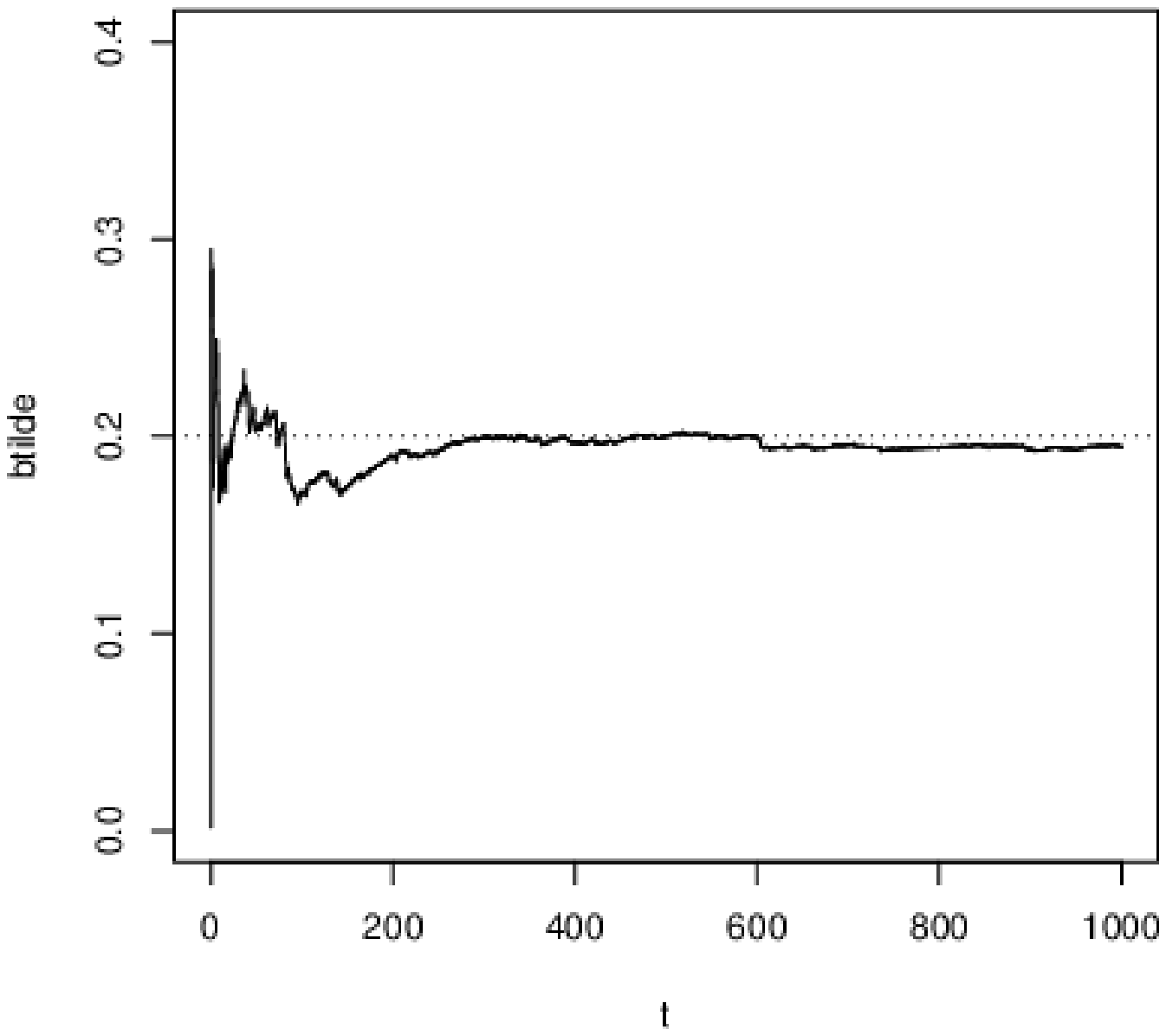,height=0.4\textwidth}}
\caption{The time evolution of the estimators $\tilde{a}_t$ and $\tilde{b}_t$, $T = 1000$} \label{Figure 8}
\end{figure}

From this one particular trajectory it seems that the families $(\hat{a}_T, \hat{b}_T)$ and $(\tilde{a}_T, \tilde{b}_T)$ do not differ much, but let us take a closer look at the results of $100$ simulations. Figures \ref{Figure 9} and \ref{Figure 10} show values of all obtained estimators with corresponding Q--Q plots depicted in Figures \ref{Figure 11} and \ref{Figure 12}. The overall statistics can be found in Table \ref{Table 2} with the same meaning as above.

\begin{figure}[h!]
\centering
\subfigure[The values of $\hat{a}_T$ -- Overall]{\epsfig{file=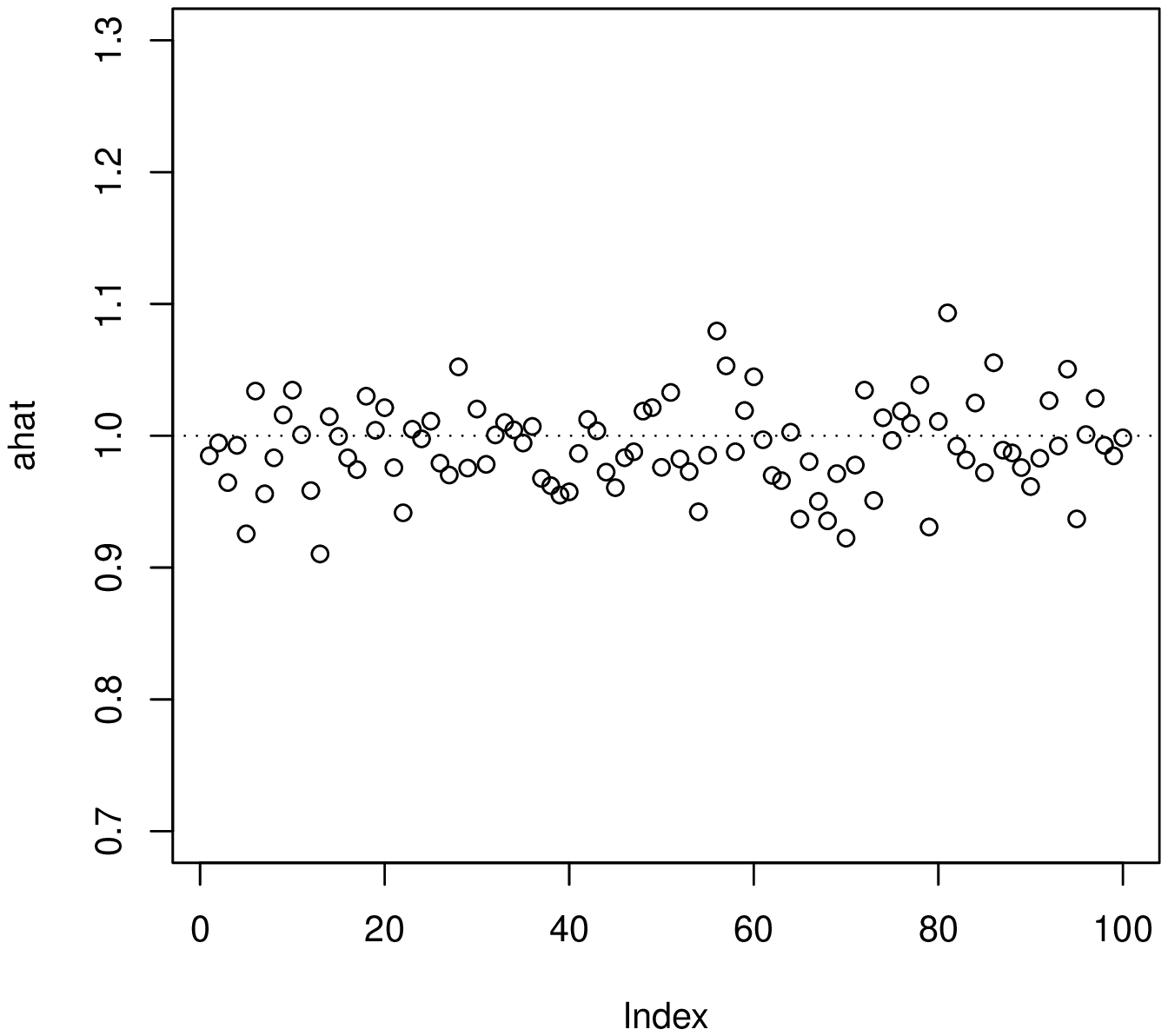,height=0.4\textwidth}}
\hspace{50pt}
\subfigure[The values of $\hat{b}_T$ -- Overall]{\epsfig{file=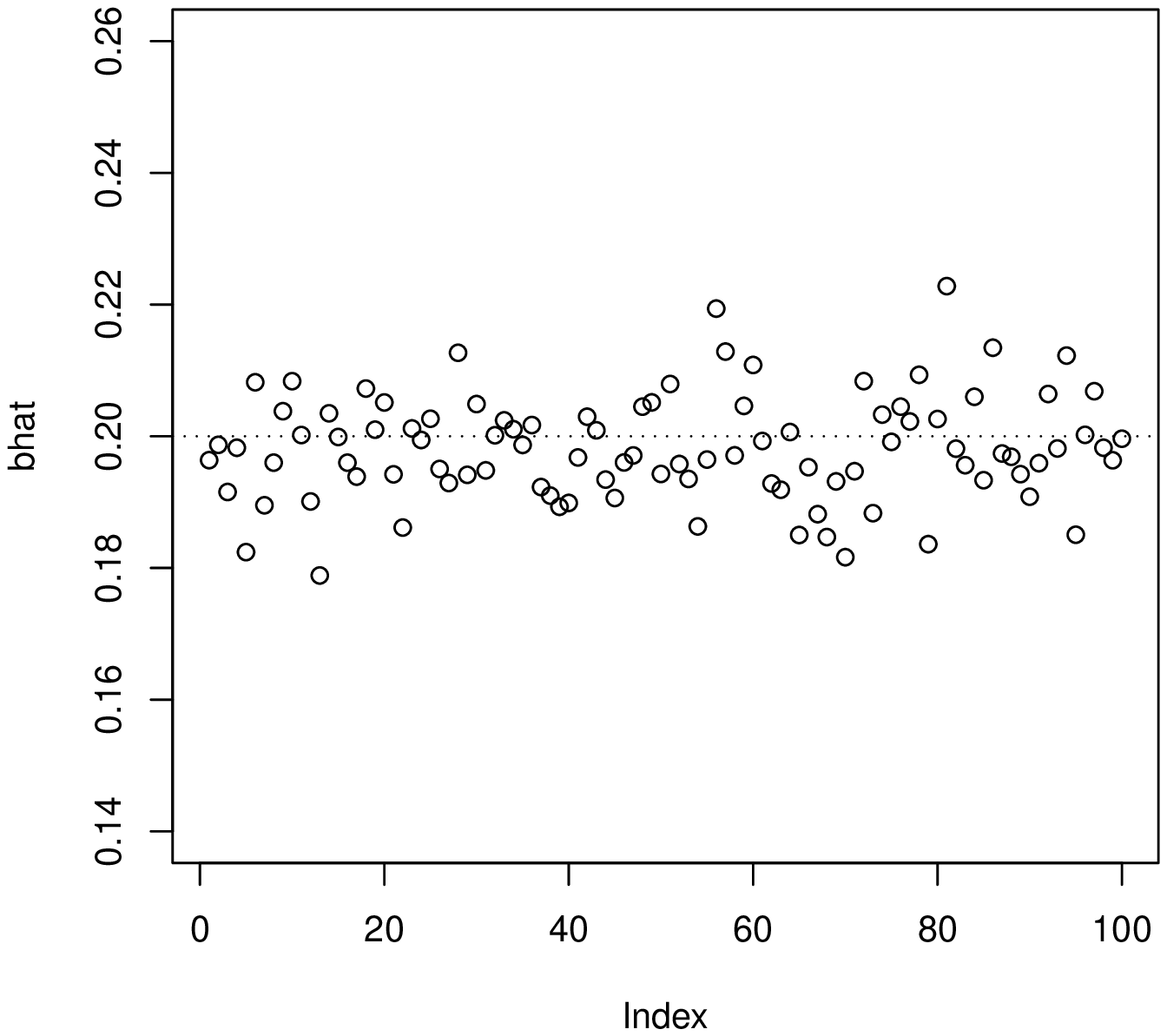,height=0.4\textwidth}}
\caption{The estimators $\hat{a}_T$ and $\hat{b}_T$ based on larger sample, $T = 1000$} \label{Figure 9}
\end{figure}

\begin{figure}[h!]
\centering
\subfigure[The values of $\tilde{a}_T$ -- Overall]{\epsfig{file=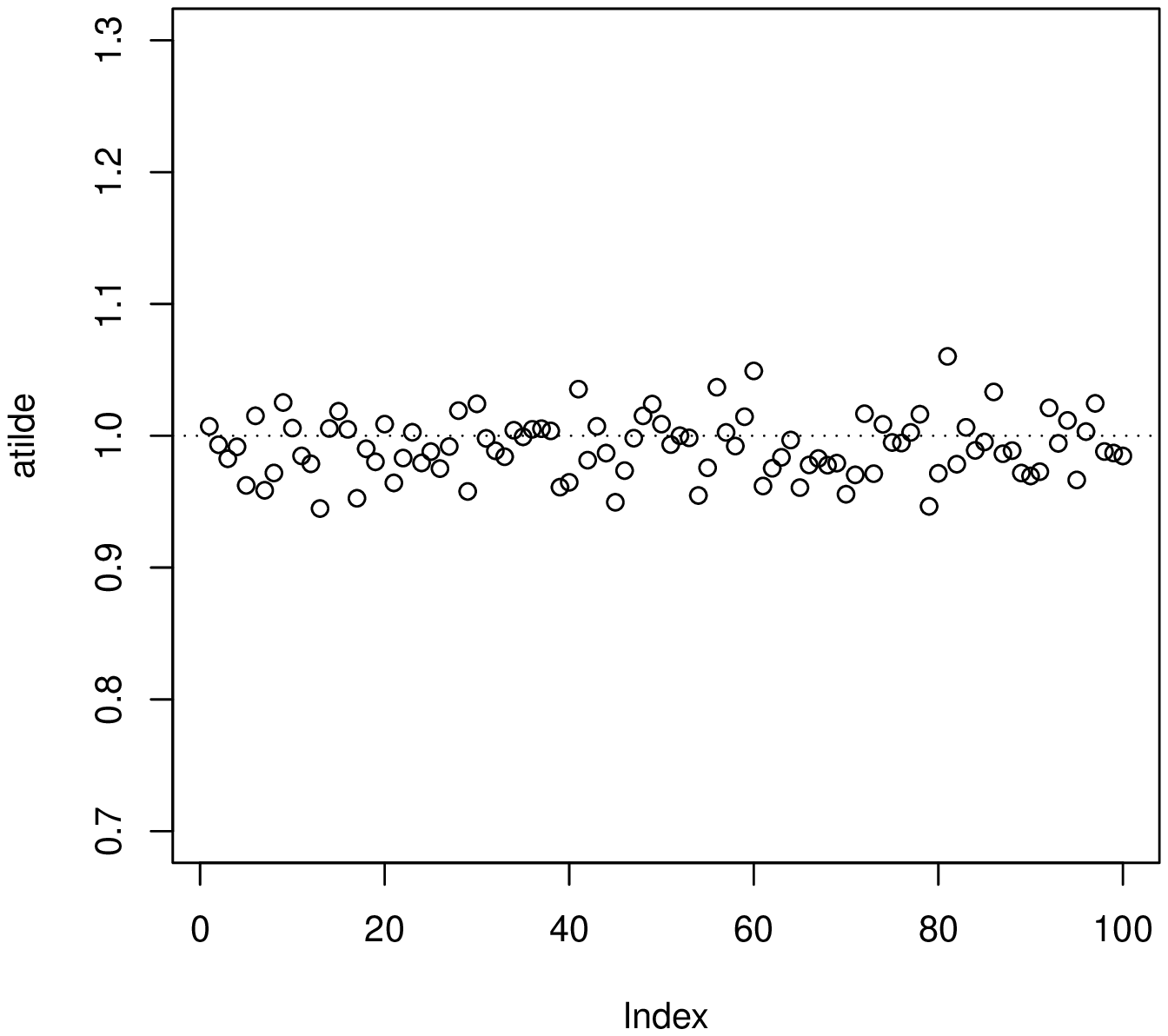,height=0.4\textwidth}}
\hspace{50pt}
\subfigure[The values of $\tilde{b}_T$ -- Overall]{\epsfig{file=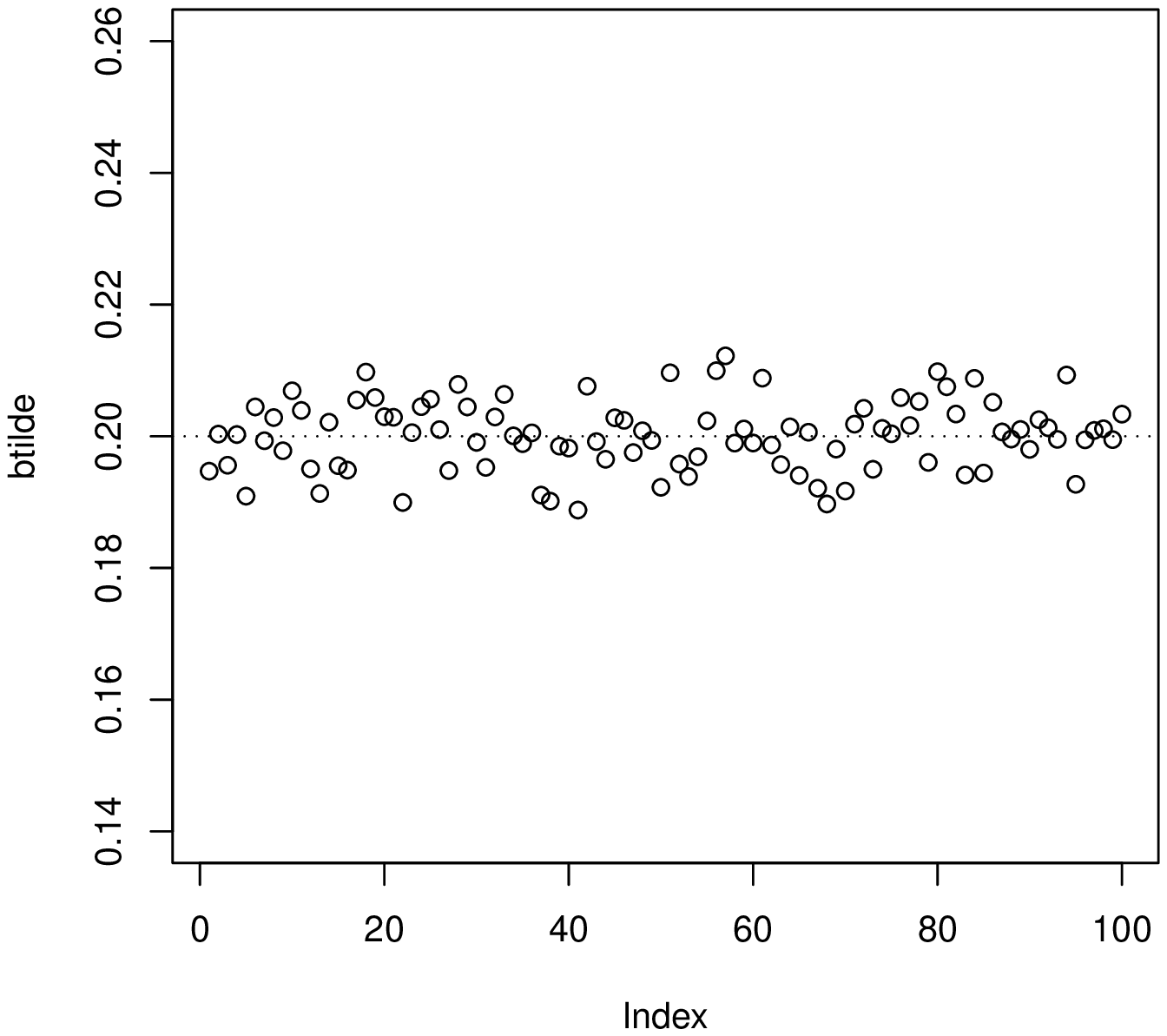,height=0.4\textwidth}}
\caption{The estimators $\tilde{a}_T$ and $\tilde{b}_T$ based on larger sample, $T = 1000$} \label{Figure 10}
\end{figure}

\begin{figure}[h!]
\centering
\subfigure[Q--Q plot of $\sqrt{T}(\hat{a}_T - a)$]{\epsfig{file=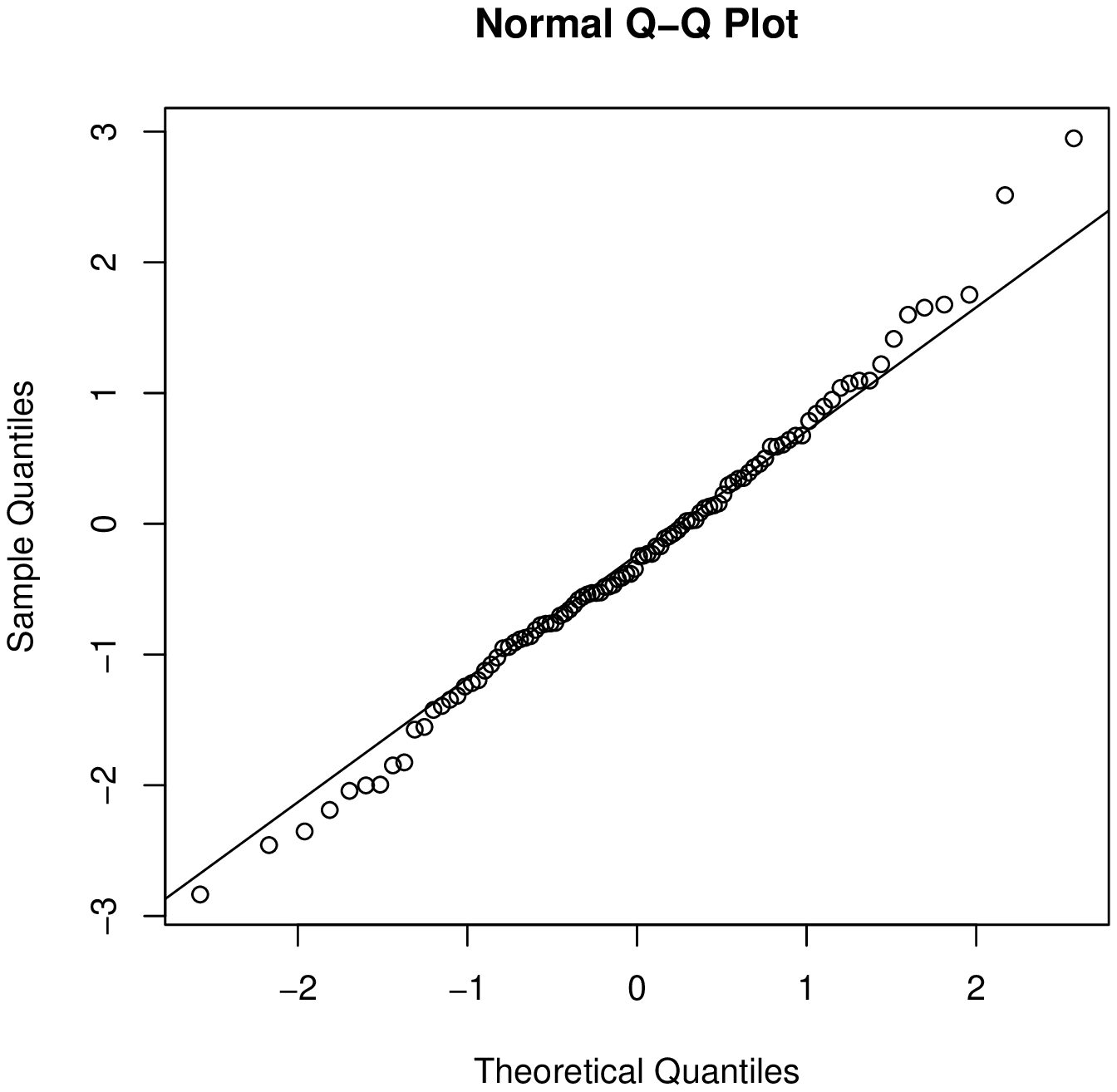,height=0.4\textwidth}}
\hspace{50pt}
\subfigure[Q--Q plot of $\sqrt{T}(\hat{b}_T - b)$]{\epsfig{file=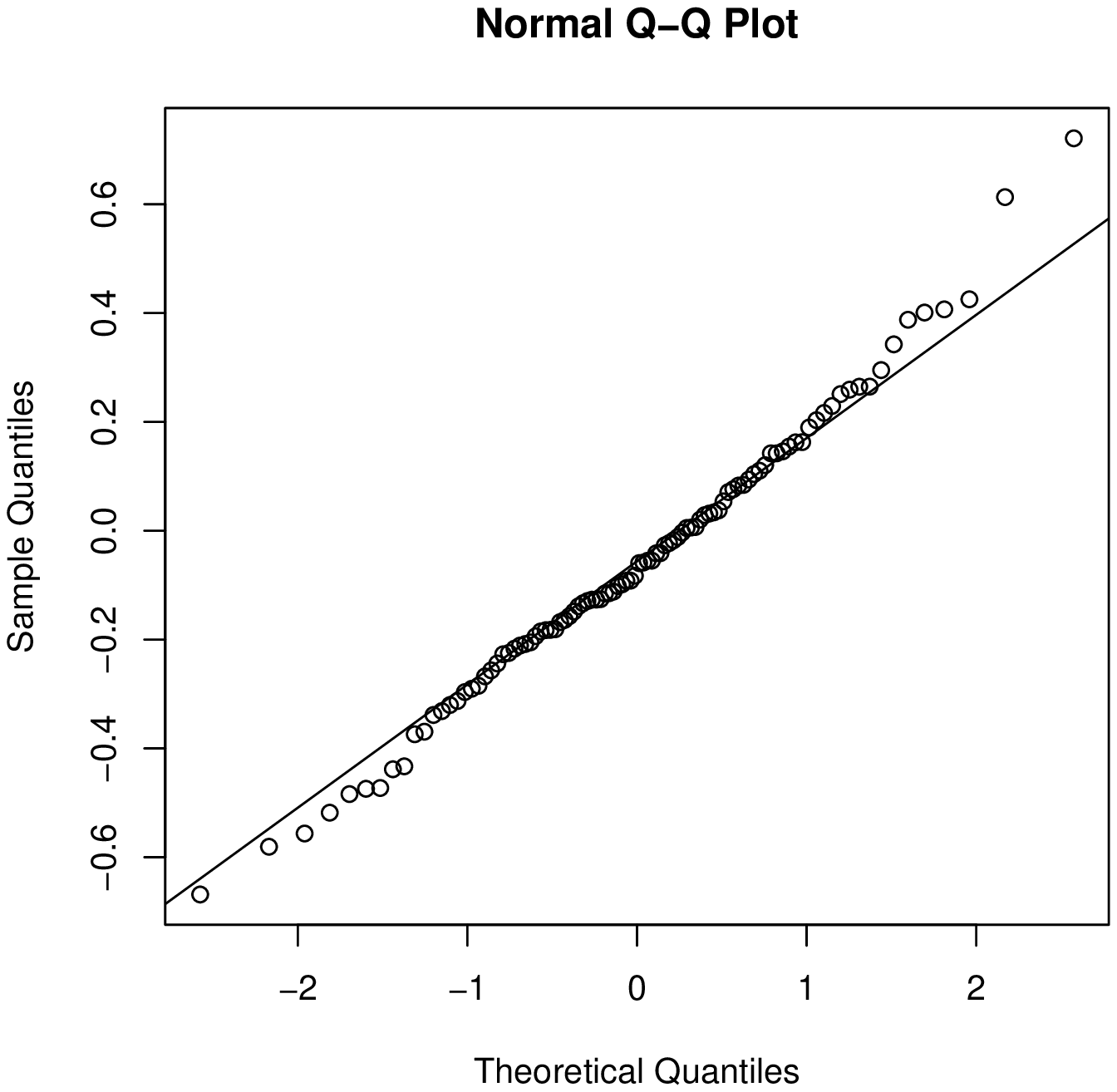,height=0.4\textwidth}}
\caption{Asymptotic normality of $\hat{a}_T$ and $\hat{b}_T$, $T = 1000$} \label{Figure 11}
\end{figure}

\begin{figure}[h!]
\centering
\subfigure[Q--Q plot of $\sqrt{T}(\tilde{a}_T - a)$]{\epsfig{file=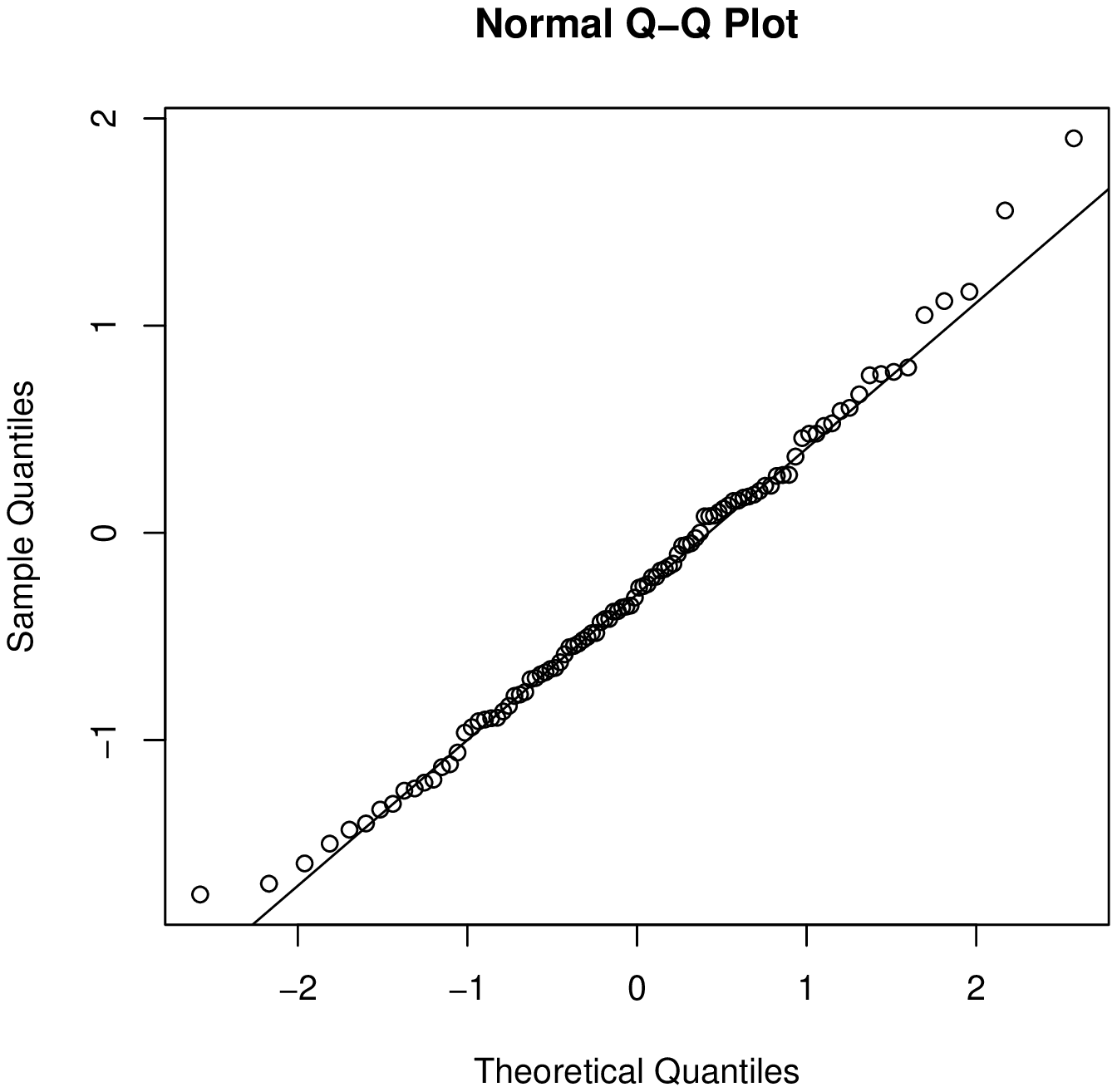,height=0.4\textwidth}}
\hspace{50pt}
\subfigure[Q--Q plot of $\sqrt{T}(\tilde{b}_T - b)$]{\epsfig{file=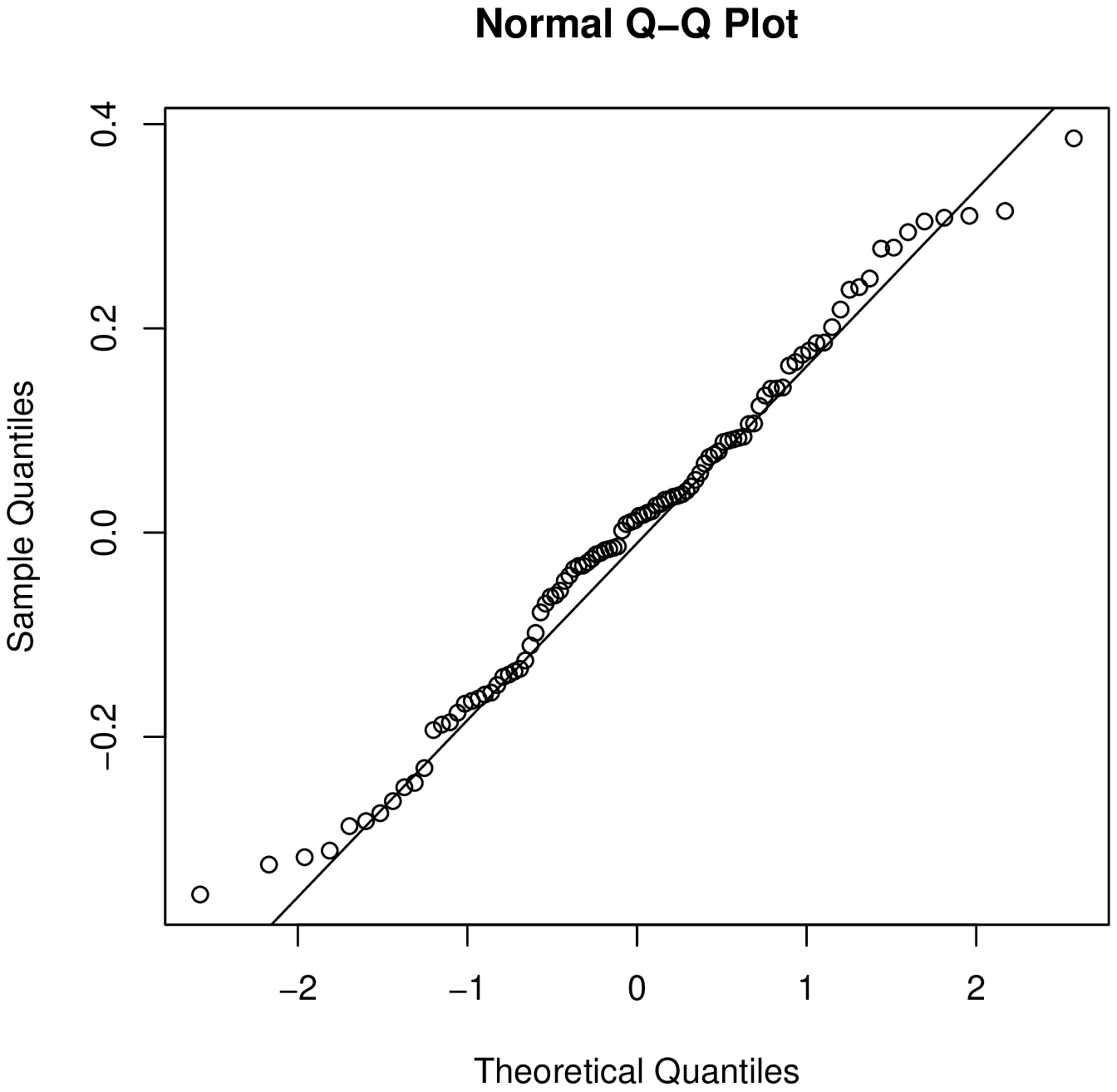,height=0.4\textwidth}}
\caption{Asymptotic normality of $\tilde{a}_T$ and $\tilde{b}_T$, $T = 1000$} \label{Figure 12}
\end{figure}

\begin{table}[h!]
\centering
\begin{tabular}{|l||c|c|c|c|} \hline
& $\hat{a}_T$ & $\hat{b}_T$ & $\tilde{a}_T$ & $\tilde{b}_T$ \\ \hline \hline
Mean & 0.9921 & 0.1982 & 0.9916 & 0.2001 \\ \hline
Var & 1.1285 & 0.0648 & 0.5186 & 0.0280 \\ \hline
Var -- Theoretical & 1.0466 & 0.0776 & 0.4505 & 0.0343 \\ \hline
Relative error -- Maximal & 9 \% & 12 \% & 6 \% & 6 \% \\ \hline
Relative error -- Typical & $\leq$ 4 \% & $\leq$ 5 \% & $\leq$ 3 \% & $\leq$ 3 \% \\ \hline
$p$--value & 0.8690 & 0.7913 & 0.7093 & 0.4192 \\ \hline
\end{tabular}
\bigskip
\caption{The results of the simulation for time $T = 1000$} \label{Table 2}
\end{table}

The conclusions of these simulations are similar as above: The family of estimators $(\tilde{a}_T, \tilde{b}_T)$ can be viewed better as the family $(\hat{a}_T, \hat{b}_T)$ since it has smaller variances and smaller relative errors. Moreover, we can compare the results from Tables \ref{Table 1} and \ref{Table 2}:

\begin{itemize}
\item The estimators for the time $T = 1000$ have $10$ times lesser variances than those for the time $T = 100$. (The actual variances of the estimators for the time $T = 1000$ are $1000$ times smaller than the numbers in the raw "Var" in Table \ref{Table 2}.)
\item The estimators for the time $T = 1000$ have about two times smaller relative errors than those for the time $T = 100$.
\item From the Q--Q plots and from the results of the Wilk--Shapiro tests, it seems that the asymptotic normality of estimators is better for greater time $T$.
\end{itemize}

After running many simulations (also with different parameters $a$, $b$, $N$, $T$, $\Delta t$, $u_1$, $u_2$, $\lambda_n$), we claim that all estimators have their derived properties and that our implementation is correct and fully functional.

%\newpage

\end{document}